\documentclass[a4paper,11pt]{article}

\usepackage{amsthm,stmaryrd}   
\usepackage{amsfonts, mathrsfs,amssymb}
\usepackage{amsmath,bbm}
\usepackage[numbers, sort&compress]{natbib}
\usepackage{vmargin, enumerate}
\usepackage{paralist}
\usepackage{times}
\usepackage{color,soul}
\usepackage{hyperref} 
\usepackage[utf8]{inputenc} 
\usepackage[labelfont=bf,font={it,small}]{caption}

\usepackage[normalem]{ulem}

\usepackage{pgf}
\usepackage{tikz}
\usetikzlibrary{arrows,shapes}
\usepackage{graphicx} 

\newtheorem{thm}{Theorem} 

\newtheorem{lem}[thm]{Lemma}
\newtheorem{prop}[thm]{Proposition}
\newtheorem{cor}[thm]{Corollary}
\theoremstyle{definition}
\newtheorem{rem}{Remark}

\newcommand{\Hoelder}{H{\"o}lder\ }
\newcommand{\Leb}{\operatorname{Leb}}

\newcommand{\cI}{{\mathcal I}}

\newcommand{\scL}{{\mathcal L}}
\newcommand{\cM}{{\mathcal M}}
\newcommand{\cP}{{\mathcal P}}
\newcommand{\cR}{{\mathcal R}}

\newcommand{\cS}{{\mathcal S}}
\newcommand{\cE}{{\mathcal E}}
\newcommand{\cV}{{\mathcal V}}
\newcommand{\cT}{{\mathcal T}}
\newcommand{\cY}{{\mathcal Y}}
\newcommand{\cX}{{\mathcal X}}
\newcommand{\cN}{\mathcal N}
\newcommand{\bbM}{\mathbb M}

\newcommand{\bbT}{\mathbb T}
\newcommand{\bbK}{\mathbb K}

\newcommand{\sL}{\mathscr L}

\newcommand{\sD}{\mathscr D}

\newcommand{\bJ}{\mathbf J}

\newcommand{\sT}{\mathcal T}

\newcommand{\sH}{\mathscr H}
\newcommand{\sZ}{\mathscr Z}

\newcommand{\dpp}{\mathrm{d_{\text{\textsc{p}}}^Z}}
\newcommand{\dha}{\mathrm{d_{\text{\textsc{h}}}^Z}}

\newcommand{\rr}{\operatorname{rk}}

\newcommand{\noi}{\noindent}
\newcommand{\GHP}{{\text{\textsc{ghp}}}}
\newcommand{\GH}{{\text{\textsc{gh}}}}
\newcommand{\GP}{{\text{\textsc{gp}}}}
\newcommand{\dgp}{\mathrm{d_{\GP}}}
\newcommand{\dgh}{\mathrm{d_{\GH}}}
\newcommand{\dghp}{\mathrm{d_{\GHP}}}

\newcommand{\bin}{\mathrm{Bin}}

\newcommand{\bj}{{\mathbf j}}
\newcommand{\br}{{\mathbf r}}
\newcommand{\bs}{{\mathbf s}}

\newcommand{\bu}{{\mathbf u}}

\newcommand{\fm}{{\mathfrak m}}
\newcommand{\fT}{{\mathfrak T}}
\newcommand{\fX}{{\mathfrak X}}

\newcommand{\Po}{{\operatorname{Po}}}
\newcommand{\e}{{\mathbf E}}
\newcommand{\tf}{{\mathrm f}}

\newcommand{\Xex}{\mathcal Z}
\newcommand{\cp}{\varphi^\circ}
\newcommand{\cC}{\mathcal C}

\newcommand{\llb}{\llbracket}
\newcommand{\rrb}{\rrbracket}

\newcommand{\Z}{\ensuremath{\mathbb Z}}
\newcommand{\N}{\ensuremath{\mathbb N}}

\newcommand{\R}{\ensuremath{\mathbb{R}}}

\newcommand{\p}[1]{{\mathbf P}\left(#1\right)}
\newcommand{\pc}[1]{{\mathbf P}(#1)}

\newcommand{\Ec}[1]{\ensuremath{\mathbf{E} [#1]}}

\newcommand{\eqdist}{\ensuremath{\stackrel{d}{=}}}

\newcommand{\E}[1]{\ensuremath{\mathbf{E} \left[#1 \right]}}

\newcommand{\Prob}[1]{\ensuremath{\mathbf{P} \left(#1 \right)}}
\newcommand{\I}[1]{\ensuremath{\mathbf{1}_{ \{ #1 \} }}}
\newcommand{\In}[1]{\ensuremath{\mathbf{1}_{  #1 }}}

\newcommand{\Law}{\mathfrak{L}}

\newcommand{\Co} {\mathbb{C}}
\newcommand{\Dex} {\mathbb{D}_{\text{ex}}}
\newcommand{\Cex} {\mathbb{C}_{\text{ex}}}
\newcommand{\Do} {\mathbb{D}}

\newcommand{\DimuM}{\operatorname{\underline{\dim}_{\text{\textsc{m}}}}}
\newcommand{\DimoM}{\operatorname{\overline{\dim}_{\text{\textsc{m}}}}}
\newcommand{\DimH}{\operatorname{\dim_{\text{\textsc{h}}}}}
\newcommand{\DimM}{\operatorname{\dim_{\text{\textsc{m}}}}}

\makeatletter
\newsavebox\myboxA
\newsavebox\myboxB
\newlength\mylenA

\newcommand*\xoverline[2][0.75]{%
    \sbox{\myboxA}{$\m@th#2$}%
    \setbox\myboxB\null
    \ht\myboxB=\ht\myboxA%
    \dp\myboxB=\dp\myboxA%
    \wd\myboxB=#1\wd\myboxA
    \sbox\myboxB{$\m@th\overline{\copy\myboxB}$}
    \setlength\mylenA{\the\wd\myboxA}
    \addtolength\mylenA{-\the\wd\myboxB}%
    \ifdim\wd\myboxB<\wd\myboxA%
       \rlap{\hskip 0.5\mylenA\usebox\myboxB}{\usebox\myboxA}%
    \else
        \hskip -0.5\mylenA\rlap{\usebox\myboxA}{\hskip 0.5\mylenA\usebox\myboxB}%
    \fi}
\makeatother

\hypersetup{
    bookmarks=true,         
    unicode=false,          
    pdftoolbar=true,        
    pdfmenubar=true,        
    pdffitwindow=true,      
    pdftitle={My title},    
    pdfauthor={Author},     
    pdfsubject={Subject},   
    pdfnewwindow=true,      
    pdfkeywords={keywords}, 
    colorlinks=true,       
    linkcolor=blue,          
    citecolor=blue,        
    filecolor=blue,      
    urlcolor=blue           
}

\usepackage[refpage,noprefix,intoc]{nomencl}

\makenomenclature                                            
\begin{document}
\title{\bf Self-similar real trees defined as fixed points and \\their geometric properties
\thanks{NB acknowledges the support of the grant ANR-14-CE25-0014 (ANR GRAAL). The research of HS was supported by the FSMP, 
reference: ANR-10-LABX-0098, and a Feodor Lynen Research Fellowship of the Alexander von Humboldt Foundation. }}
\author{Nicolas Broutin\thanks{Sorbonne Université, Campus Pierre et Marie Curie, 4 place Jussieu, 75252 Paris Cedex 05, France. Email: nicolas.broutin@upmc.fr} 
\and Henning Sulzbach\thanks{McGill University, 3480 University Street, H3A 0E9 Montreal, QC, Canada. 
Email: henning.sulzbach@gmail.com.  
Present address: School of Mathematics, University
of Birmingham, Birmingham B15 2TT, Great Britain
}} 

\maketitle

\begin{abstract}
We consider fixed point equations for probability measures charging measured compact metric 
spaces that naturally yield continuum random trees. On the one hand, we study the existence/uniqueness of the fixed points and the convergence of the corresponding iterative schemes. On 
the other hand, we study the geometric properties of the random measured real trees that are 
fixed points, in particular their fractal properties. We obtain bounds on the 
Minkowski and Hausdorff dimension, that are proved tight in a number of applications,
including the very classical continuum random tree, but also for the dual trees of 
random recursive triangulations of the disk introduced by Curien and Le Gall 
[\emph{Ann Probab}, vol. 39, 2011]. The method happens to be especially efficient to treat
cases for which the mass measure on the real tree induced by natural encodings only provides 
weak estimates on the Hausdorff dimensions.
\end{abstract}

%
%
%

\section{Introduction}
\label{sec:intro}

Since the pioneering work of \citet{Aldous1991a,Aldous1993a} who introduced the Brownian continuum 
random tree (Brownian CRT) as a scaling limit for uniformly random labelled trees, similar 
objects have been shown to play a crucial role in a number of limits of combinatorial problems 
that relate to computer science, physics or biology. These objects are 
real trees, or tree-like compact metric spaces (see Section~\ref{sec:results} for a formal 
definition), and they are usually equipped with a probability measure that yields a notion 
of ``mass''. 
They naturally appear when studying asymptotic properties of discrete combinatorial or 
probabilistic objects that are intrinsically ``branching'' or recursive such as branching processes and
fragmentation processes. More surprisingly, further prominent examples are that of 
random maps \cite{LeGall2011a,Miermont2013b,CoVa1981,Schaeffer1998a} and of Liouville quantum gravity 
\cite{DuMiSh2014a} that would a priori not be expected to relate to tree structures. 

In a number of cases, these continuous objects, or, more precisely, their distributions, happen 
to satisfy a stochastic fixed point equation; such fixed point equations are often formulated in terms
of the distribution of functions (later referred to as height functions) that encode the trees. 
One may think in particular of the 
Brownian CRT \cite{al94}, of trees that are dual to recursive triangulations 
of the disk \cite{legalcu}, but also of the genealogies of self-similar fragmentations 
\cite{HaMi2004a}.  We will be more precise about the equations we consider shortly, but it is
nonetheless informative to fix ideas: informally, a \emph{distributional fixed point} 
equation for a random variable (r.v.) $X$ taking values in some Polish space $\mathbb S$ 
of ``objects'' is an equation of the form
\begin{equation}\label{eq:fix_general}
X \eqdist T((X_i)_{i\ge 1}, \Xi),
\end{equation}
where $(X_i)_{i\ge 1}$ is a family of independent and identically distributed copies of $X$, $T$ is a suitable
map, and $\Xi$ incorporates additional external randomness. (A precise formulation of such an equation for random 
metric spaces is more involved. See display \eqref{fix:tree} below.) 
The fact that natural  objects satisfy equations such as the one in \eqref{eq:fix_general} 
raises many questions about the properties of such equations and of their possible fixed points:
\begin{itemize}
\item [(i)] \emph{Under which conditions does there exist a fixed point ?}
\item [(ii)]  \emph{Under which conditions is this fixed point unique ?}
\item [(iii)]  \emph{Can the fixed point be obtained by some iterative procedure ?}
\end{itemize}
The answers to these questions of course depend on the space $\mathbb S$ that is 
considered, and some special care is needed in specifying it. 

 One of the striking features of random real trees that appear 
ubiquitous is their fractal nature.
Among the most classical real trees one may cite the L\'evy trees (including the Brownian 
CRT and the stable trees), which are the scaling limits of 
rescaled Galton--Watson processes, and whose fractal properties have been investigated by 
\citet{DLG05,DLG06} and \citet{picard08}. Another important example is that of the 
fragmentation trees encoding certain self-similar fragmentation processes whose fractal properties 
have been studied by \citet{HaMi2004a} and more recently by \citet{Stephenson2013a}. 
In view of the recursive self-similarity of Equation \eqref{eq:fix_general}, this 
raises an additional question about the geometry of the fixed points:
\begin{itemize}
\item [(iv)]\emph{Can one quantify the fractal dimensions of the fixed points ? }
\end{itemize}

Finally, observe that, for instance, the Brownian CRT is binary, in the sense that the removal of any point disconnects the space into $1$, 
$2$ or $3$ connected components with probability one (the number of connected components is 
called the degree of the point removed). This is to be compared with the classical
decomposition of the Brownian CRT into three pieces \cite{al94,albgol2015, hamcro}. Another example we 
have already mentioned, dual trees of recursive triangulations of the disk happen to have 
maximal degree three, while the natural fixed point equation they satisfy only uses two pieces. 
These considerations raise yet another question about the geometry of solutions to equations 
such as \eqref{eq:fix_general}:
\begin{itemize}
\item [(v)]\emph{Can one fully characterize the degrees of points in fixed points ?}
\end{itemize}

Our aim in this paper is to provide answers to questions (i)--(v) in a general framework 
in which the limit objects are (most often) some classes of measured real trees. 
This framework allows 
for instance to deal with certain recalcitrant cases where the natural height function 
for the tree is not a ``good'' encoding, in the sense that its optimal H\"older exponent does 
not yield the fractal dimension of the metric space (we will be more precise shortly). 
At this point, let us mention that questions (i), (ii) and (iii) have recently been studied 
by \citet{albgol2015} for the specific example where the fixed point equation is the one 
described by Aldous in \cite{al94} and that is satisfied by the Brownian CRT. In passing,
our results answer a question in \cite{albgol2015} regarding point (iii) and the 
convergence to the (non-unique, but natural) fixed point. Other applications of our results concern trees arising as scaling 
limits in the problem of recursive triangulations of the disk (see \cite{legalcu} and 
\cite{BrSu2013a}), and but also other natural generalizations. \citet{ReWi2016a} have also 
very recently studied questions (i), (ii) and (iii) for a decomposition of the 
form \eqref{fix:tree} which is rather different from ours. 
See the remark at the end of Section \ref{sec:rec_trees} for details.

\medskip\noindent
\textbf{Organization of the paper.}\ The paper is organized as follows: 
In Section~\ref{sec:settings}, we first give the relevant background on the objects, 
metrics and spaces, and geometric properties we use in the document; 
we then introduce the precise setting for the recursive equations we consider, 
and the corresponding functional point of view. Section~\ref{sec:results} is devoted 
to the statements of our main results; it also contains a sample of applications and 
an overview of the techniques we use. Section~\ref{sec:proofs} contains the proofs 
of the results about existence and uniqueness of solutions to our recursive equations, 
as well the behaviour of iterative schemes. Section~\ref{sec:proofs_frac} contains 
the proofs of the geometric properties of the fixed points. 
Finally, Section~\ref{sec:app} is devoted to applications. 
Various proofs of technical results are given in appendix.

\section{Settings and preliminaries}\label{sec:settings}

\subsection{Spaces, metrics and convergence}\label{sec:metrics}

With the exception of Section \ref{subsec:fr}, we assume throughout the paper that metric spaces are \emph{compact}. General references on
the topics that we are about to discuss include \cite{Gromov1999,GrPfWi2009a,BuBuIv2001}.  For measured spaces, we restrict our attention to \emph{probability measures.}

\medskip
\noi\textsc{The Gromov--Hausdorff--Prokhorov topology.}\ 
For two compact metric spaces $(\cX,d)$ and $(\cX',d')$, 
the \emph{Gromov--Hausdorff distance} \sloppy 
$\dgh((\cX,d),(\cX',d'))$ is defined as 
\begin{align} \label{def:gh} 
\dgh((\cX,d),(\cX',d')) = \inf_{Z, \phi, \phi'} \dha(\phi(\cX),\phi'(\cX')), 
\end{align}
where the infimum is taken over all compact metric 
spaces $(Z,d^Z)$, and isometries $\phi: \cX\to Z$ and $\phi':\cX'\to Z$. 
Here, $\dha$ denotes the Hausdorff distance in $Z$, that is 
$$\dha(A,B) 
= \inf \{ \varepsilon > 0: A \subseteq B^\varepsilon \: \text{and} \: B \subseteq A^\varepsilon \},$$ 
with $A^\varepsilon = \{ x \in Z : d^Z(x,A) \leq \varepsilon\}$. If $(\cX,d)$ and $(\cX',d')$ 
are isometric, then $\dgh((\cX,d),(\cX',d'))=0$. $\dgh$ 
induces a metric on 
the set $\mathbb K^{\GH}$ of isometry classes of compact metric spaces and turns this set 
into a Polish space, see, e.g.\ \cite[Theorem 2.1]{Legall2005}.

A compact rooted (or pointed) measured metric space $(\cX, d, \mu, \rho)$ is a compact 
metric space $(\cX, d)$ endowed with a probability measure $\mu$ and one distinguished 
point $\rho$. For two  such spaces $\fX=(\cX,d, \mu, \rho)$ and $\fX'=(\cX', d', \mu', \rho')$, 
we define the \emph{Gromov--Hausdorff--Prokhorov} distance by
\begin{align*}\dghp(\fX, \fX') 
= \inf_{Z, \phi, \phi'} 
\big \{ &d^Z(\phi(\rho), \phi'(\rho')) 
+ \dha(\phi(\cX),\phi'(\cX')) \\
& + \dpp(\phi_*(\mu), \phi_*'(\mu'))\big\}. \end{align*}
Here, the infimum is to be understood as in \eqref{def:gh}, $\phi_*(\mu)$ is the 
push-forward of $\mu$ in $Z$, and $\dpp$ denotes the Prokhorov metric on the 
set of probability measures on $Z$, that is,
\begin{align*}
\dpp(\nu_1, \nu_2) 
= \inf \{ \varepsilon >0: &\: \nu_1(A) \leq \nu_2(A^\varepsilon) + \varepsilon \ 
\text{and} \ \nu_2(A) \leq \nu_1(A^\varepsilon) + \varepsilon \\ 
&  \: \text{for all measurable sets } A \subseteq \cX \}. 
\end{align*}
We call $\fX$ and $\fX'$ $\GHP$-isometric if there exists a bijective isometry $\phi$ 
between $X$ and $X'$ that maps $\rho$ to $\rho'$ and such that $\phi_*(\mu)=\mu'$. 
$\dghp$ induces a metric on the set $\bbK^{\GHP}$ of $\GHP$-isometry 
classes of  compact  rooted measured metric spaces that turns it into a Polish 
space \cite{abrdelhos}. 
For a compact rooted measured metric space $\mathfrak X$ (or an element of $\mathbb K^\GHP$), we set $\| \fX \| = \sup \{d(\rho, x) : x \in \cX\}$.

\medskip \noi\textsc{The Gromov--Prokhorov topology.} 
Analogously to the Gromov--Hausdorff--Prokhorov distance, for two  compact rooted 
measured metric spaces $\fX = (\cX, d, \mu, \rho)$ and $\fX' = (\cX', d', \mu', \rho')$ \footnote{When we do not introduce the components of $\fX$ with 
a given decoration explicitly, we always suppose that they would carry the same decoration 
as done here for $\fX'$.} 
we define
$$\dgp(\fX, \fX') 
= \inf_{Z, \phi, \phi'} \left\{ d^Z(\phi(\rho), \phi'(\rho'))  
+ \dpp(\phi_*(\mu), \phi_*'(\mu'))\right\}.$$
We call $\fX$ and $\fX'$ $\GP$-isometric if $\dgp(\fX, \fX') = 0$ which happens to 
be the case if and only if there exists a bijective isometry 
$\phi : \text{supp}(\mu) \to \text{supp}(\mu')$ with $\rho' = \phi(\rho)$ 
and $\mu' = \phi_*(\mu)$.  (Here $\text{supp}(\mu)$ denotes the support of $\mu$.)
Endowed with $\dgp$, the set $\bbK^{\GP}$ of $\GP$-isometry classes of compact rooted 
measured metric spaces becomes a Polish space \cite[Proposition 5.6]{GrPfWi2009a}. 
In general, $\GP$-equivalence classes in $\bbK^{\GP}$ contain spaces that are not 
$\GHP$-isometric. But when both $\mu$ and $\mu'$ have full support, then $\fX$ and
$\fX'$ are $\GP$-isometric if and only if they are $\GHP$-isometric. Thus, if we 
denote by $\bbK^{\GHP}_\tf$ the set of $\GHP$-isometry classes of compact rooted 
measured metric spaces satisfying
\begin{enumerate}
\item [\textbf{(C1)}] $\text{supp}(\mu) = \cX$,  
\end{enumerate}
then, there exists a natural bijection $\iota$ between the spaces $\bbK^{\GHP}_\tf$ 
and $\bbK^{\GP}$.  The set $\bbK^{\GHP}_\tf$ is measurable, and $\iota$ bimeasurable
so we can and will consider any random variable with values in $\bbK^{\GHP}_\tf$ also as random variable in $\bbK^{\GP}$ 
and vice versa. (Proving these statements makes use of Lemma 3.2 and Corollary 5.6 in \cite{ALW} as well as the Lusin--Souslin theorem. See Lemma~\ref{lem:iota} in the appendix for details.)

\medskip
\noindent \textsc{Remark}. 
We occasionally use results from \cite{Gromov1999, BuBuIv2001, GrPfWi2009a, DeGrPf} which only treat the case of unrooted compact measured metric spaces. Incorporating a root 
vertex only generates marginal modifications that we do not discuss in detail.

\subsection{Real trees, continuum trees and recursive decompositions} \label{sec:rtc}

We are mostly interested in a certain class of metric spaces that are tree-like. 

\medskip
\noi\textsc{Real trees.}\ A metric space $(\sT,d)$ is a called \emph{real tree} if it has the 
following properties: 
\begin{compactenum}
\item for every $x,y \in \sT$ there exists a unique isometry $\varphi_{x,y} : [0, d(x,y)] \to \cT$ 
with $\varphi_{x,y}(0) = x$ and $\varphi_{x,y}( d(x,y)) = y$, 
(we write $\llb x,y\rrb:=\varphi_{x,y}([0,d(x,y)])$ for the \emph{segment} between $x$ and $y$ in $\cT$),
\item if $q: [0,1] \to \cT$ is a continuous and injective map with $q(0) = x, q(1) = y$, then 
$q([0,1]) = \llb x,y\rrb.$  
\end{compactenum}
We denote by $\bbT^{\GH}$ the closed subset of $\bbK^{\GH}$ consisting of isometry classes 
of compact real trees.
For a compact real tree $(\sT, d)$ and $x \in \sT$, we denote by $\text{deg}(x)$ the number 
of connected components of $\sT \setminus \{x\}$. 
We call $x \in \sT$ a \emph{leaf} if $\text{deg}(x) = 1$, and abbreviate $\mathscr L$ for 
the set of leaves. 
We call $x \in \sT$ a \emph{branch point} 
if $\text{deg}(x) \geq 3$.  By compactness, the set of branch 
points $\mathscr B$ is at most countable.

\medskip
\noi\textsc{Measured and continuum real trees.}\ 
A  compact rooted measured real tree $\mathfrak T = (\sT, d, \mu, \rho)$ is a 
compact real tree $(\sT, d)$ endowed with a probability measure $\mu$ and a 
distinguished point $\rho \in \sT$ called the \emph{root}. 
(Recall that we restrict our attention to the setting where the measure is a probability distribution.)
For $x\in \cT$, the distance $d(x,\rho)$ is called the height of $x$ and $\| \fT \|:=\sup\{d(x,\rho):x\in \cT\}$ the \emph{height} of $\mathfrak T$. 
By $\bbT^{\GHP} \subseteq \bbK^{\GHP}$ we denote the closed subset of $\GHP$-isometry 
classes of  compact rooted measured real trees. 
Spaces carrying a measure with full support are particularly important, and we let 
$\bbT^{\GHP}_\tf = \bbK^{\GHP}_\tf \cap \bbT^{\GHP}$ and call 
elements in $\bbT^{\GHP}_\tf$ \emph{continuum real trees} (or simply continuum trees).
Note that both $\bbK^{\GHP}_\tf$ and $\bbT^{\GHP}_\tf$ are non-closed subsets 
of $\bbK^{\GHP}$. In the literature on continuum real trees, see, e.g.\ \cite{Aldous1993a, albgol2015}, 
one often finds the following two additional conditions: 
\[
\textbf{(C2)}\quad \mu \text{ has no atoms}\qquad \qquad \text{and} \qquad \qquad 
\textbf{(C3)}\quad \mu(\mathscr L) = 1\,. 
\]
All continuum trees playing a role in this paper satisfy both \textbf{C2} and 
\textbf{C3}. However, we emphasize the fact that we do not impose these conditions 
beforehand: they can be proved to hold as a non-trivial consequence of our setting; see 
Proposition~\ref{prop:degrees}.

\medskip 
\noi\textsc{Real trees encoded by excursions.}\  
One natural way to define real trees is via an encoding by continuous excursions (see e.g., 
\cite{Legall2005,evans}). Let $\Co$ be the space of continuous functions on $[0,1]$, which we always endow with the uniform norm
$\| f \| = \sup_{t \in [0,1]} |f(t)|$.
 Let also $\Cex$ denote the set of unit-length non-negative continuous excursions, that is, the set of functions $f\in \Co$ such that
 $f(0) = f(1) = 0$ and $f(t) \geq 0$ for all $t \in (0,1)$. For $f \in \Cex$, define the pseudometric $d_f$ by
$$d_f(x,y) := f(x) + f(y) -2 \inf\{f(s): x\wedge y \le s \le x \vee y\}\,.$$
Let $\sT_f = [0,1] /\!\!\sim $ where $x \sim y$ if and only if $d_f(x,y) = 0$. 
The compact metric space $(\sT_f, d_f)$ is a real tree, which we call the real tree 
encoded by $f$; we will also sometimes denote the continuous excursion $f$ as being 
a \emph{height process} for the real tree $\cT_f$. 
We use $\mathscr L_f$ for the sets of leaves of $\sT_f$.

As noted in \cite[][Remark following Theorem~2.2]{Legall2005} (see also 
\cite[Corollary 1.2]{duq1}), for every compact real tree $(\sT, d)$, there 
exists $f \in \Cex$ such that $(\sT, d)$ and $(\sT_f, d_f)$ are isometric. Two real trees $(\sT_f, d_f)$ and $(\sT_g, d_g)$, encoded by continuous excursions $f$ and $g$ respectively, 
are isometric, if, for instance, $f = g \circ \phi$ for a continuous and strictly 
increasing function $\phi : [0,1] \to [0,1]$. In this case, we call the encoding excursions $f$ and $g$ \emph{time-change-equivalent} or simply \emph{equivalent}.

Let $\pi_f: [0,1]\to \cT_f$ be the canonical surjection.
Then, we can turn the real tree $(\sT_f, d_f)$ into a  compact rooted measured real tree $\fT_f=(\cT_f, d_f, 
\mu_f, \rho_f)$ using the push-forward measure $\mu_f := \Leb \circ \, \pi_f^{-1}$, where $\Leb$ 
denotes the Lebesgue measure on $[0,1]$, and the root is $\rho_f := \pi_f(0)$.
Then, $\mu_f$ has full support, and thus $\fT_f$ 
satisfies \textbf{C1}. Finally, it is well-known that the map between $(\Cex, \| \cdot \|)$ and
$(\bbT^{\GHP}_\tf, \dghp)$ that associates $\fT_f$ to an $f\in \Cex$ is (Lipschitz) 
continuous \cite[Proposition 3.3]{abrdelhos}). Hence, for any random variable
$\Xex$ with values in $\Cex$, the corresponding (\GHP-equivalence class of 
the) compact rooted measured real tree $\fT_{\Xex}$ is a random variable with values 
in~$\bbT^{\GHP}_\tf$.

\subsection{Metric spaces described by recursive decompositions}
\label{sec:rec_trees}
 
We now introduce a general framework for random compact rooted measured metric spaces 
satisfying recursive distributional decompositions. A similar construction has already been 
used in the specific context of the Brownian CRT, see for instance \cite{al94} and
\cite{albgol2015}.

Let $\Gamma$ be a rooted plane tree\footnotemark \footnotetext{A rooted plane tree is a subset $t$ of $\cup_{n\ge 0} (\N\setminus \{0\})^n$, such that (a) if a word $u\in t$ 
then all its prefixes also are in $t$ and (b) if $ui \in t$ then also $uj\in t$ for $j=1,\dots, i-1$. The depth-first order on $t$ is the order induced on $t$ by the lexicographic order.} with $K$ vertices, where $K \geq 2$. 
We call $\Gamma$ the \emph{structural 
tree} \addtocounter{footnote}{-1} of the recursive decomposition, and accordingly, the decomposition of a space (or tree) 
will involve $K$ subparts. We consider $\Gamma$ as a labelled tree upon labelling the root by 1 and the remaining nodes in the \emph{depth-first order}\footnotemark.
We write $i \prec j$ if $j$ lies in the subtree rooted at $i$. (We always have $i \prec i$.) 
For $i \in [K] := \{1, \ldots, K\}$, we set 
\begin{align}
&\Gamma_i = \{ i \leq j \leq K : i \prec j\}, 
\quad \text{and }\quad 
E_i = \{1 \leq j < i : j \prec i\}. \label{def:subtree}
\end{align}
For $i \geq 2$, we denote by $\varpi_i = \max E_i$ the label of the parent of 
node $i$. Next, fix $\alpha > 0$, and $\br, \bs \in 
\Sigma_K := \{ (x_1,x_2,\dots, x_K) \in (0,1)^K : x_1 + \ldots + x_K = 1\}$. 
We consider the following construction (see Section~1.4 in \cite{albgol2015} for a related construction): 
given compact rooted measured metric spaces $(\cX_i, d_i, \mu_i, \rho_i)$, $i \in [K]$,
 construct a compact rooted measured metric space $(\cX, d, \mu, \rho)$ as follows: \label{list}
\begin{enumerate}
\item Independently sample points $\eta_i \in \cX_i, i \in [K]$, according to the probability measures $\mu_i$;
\item Let $\cX^\circ = \sqcup_{i=1}^K \cX_i$ denote the disjoint union of the $\cX_i$, $i=1,\dots, K$; \label{it:construction}
let $\sim_\circ$ be the  smallest equivalence relation\footnotemark \footnotetext{Formally, $\sim_\circ$ can be defined as follows: first, 
set $\rho_i\sim_1 \eta_{\varpi_i}$ and $\eta_{\varpi_i} \sim_1 \rho_i$ for all $i=2,\dots, K$. Then, for $x,y \in \cX^\circ$, set $x \sim_\circ y$ if and only if $x=y$ or there exist $k \geq 0$ and 
$x_1, \ldots, x_k \in \cX^\circ$ such that $x \sim_1 x_1 \sim_1 x_2 \ldots \sim_1 x_k \sim_1 y$.} on $X^\circ$ for which
$\rho_i\sim_\circ \eta_{\varpi_i}$ for all $i=2,\dots, K$. 
We define $\cX$ as the quotient $\cX^\circ/\!\!\sim_\circ$ and write $\cp$ 
for the canonical surjection from $\cX^\circ$ onto $\cX$. 
\item Let $d^\circ$ be the maximal pseudometric on $\cX^\circ$ that is not greater than 
$r_i^\alpha d_i$ on $\cX_i$, and for which $d^\circ(x,y)=0$  if $x\sim_\circ y$; 
define $d$ as the metric induced on $\cX$ by $d^\circ$ under $\cp$; 
(see Section 3.1.3 of \cite{BuBuIv2001}, and especially Lemma~3.1.23, Corollary~3.1.24 
and Theorem~2.1.27 there which guarantee existence and uniqueness of $d^\circ$.)
\item Let $\mu^\circ(\cdot) = \sum_{i=1}^K s_i \mu_i(\cdot \cap \cX_i)$ be the unique probability measure on $\cX^\circ$ that is compatible with 
$s_i \mu_i$ when restricted to $\cX_i$; define $\mu$ as the push-forward of $\mu^\circ$ under $\cp$.
\item Finally, let $\rho = \cp(\rho_1)$ be the root. 
\end{enumerate}
Note that, because of the need to sample  $\eta_1, \ldots, \eta_K$, the $\GHP$-equivalence 
class of the resulting space $\fX=(\cX,d, \mu, \rho)$ is random. It is crucial to observe that 
it is a random variable (Lemma~\ref{lem:conti1} in Appendix) whose 
 distribution only depends on the $\GHP$-isometry classes of $\mathfrak X_1, \ldots, \mathfrak X_K$. Hence, denoting 
by $\cM_1(\mathbb K^\GHP)$ the set of probability measures on $\mathbb K^\GHP$, the map 
$\psi : (\mathbb K^\GHP)^K \times \Sigma_K^2 \to \cM_1(\mathbb K^\GHP)$ (where $\Sigma_K^2$ 
incorporates the choice of $(\br,\bs)$) whose image is described by the construction above 
is well-defined and continuous when considering $\cM_1(\bbK^\GHP)$ equipped with the 
Prokhorov distance. (For a technical proof of continuity using the concept of correspondences and a coupling theorem due to Strassen \cite{str65}, we refer to Lemma~\ref{lem:continuity_psi} in the appendix.)
For probability measures $\tau$ on $\Sigma_K^2$ and $\aleph$ on $\mathbb K^\GHP$, we 
define the intensity measure 
\begin{equation}\label{eq:def_Psi}
\Psi(\aleph, \tau) (A) := \E{\psi(\fX_1, \ldots, \fX_K, \cR, \cS)(A)}, 
\quad A \subseteq \bbK^\GHP \quad \text{measurable},
\end{equation}
 where $\Law((\cR, \cS)) = \tau$,
$\Law(\fX_1) = \ldots = \Law(\fX_K) = \aleph$ and $(\cR, \cS), \fX_1, \ldots, \fX_K$ are 
independent. Here, and throughout the document, we use $\Law(\cdot)$ to denote the distribution of a random variable. 

Given $\Gamma$, $\alpha > 0$ and a probability distribution $\tau$ on $\Sigma_K^2$, we are interested in non-trivial laws $\aleph \in \cM_1(\mathbb K^\GHP)$ 
satisfying 
\begin{align} \label{fix:tree} 
\aleph = \Psi(\aleph, \tau).
\end{align} 
We refer to \eqref{fix:tree} as a stochastic fixed point 
equation at the level of compact rooted measured metric spaces. 
Note that, while we have introduced the map $\Psi$ for distributions on the space of $\GHP$-isometry classes 
of spaces, in the same way, one can define $\Psi$  relying on $\GP$-isometry classes, and we will 
occasionally use $\Psi$ in this sense.

\begin{figure}[!htbp] 
  \centering
  \includegraphics[width=9cm]{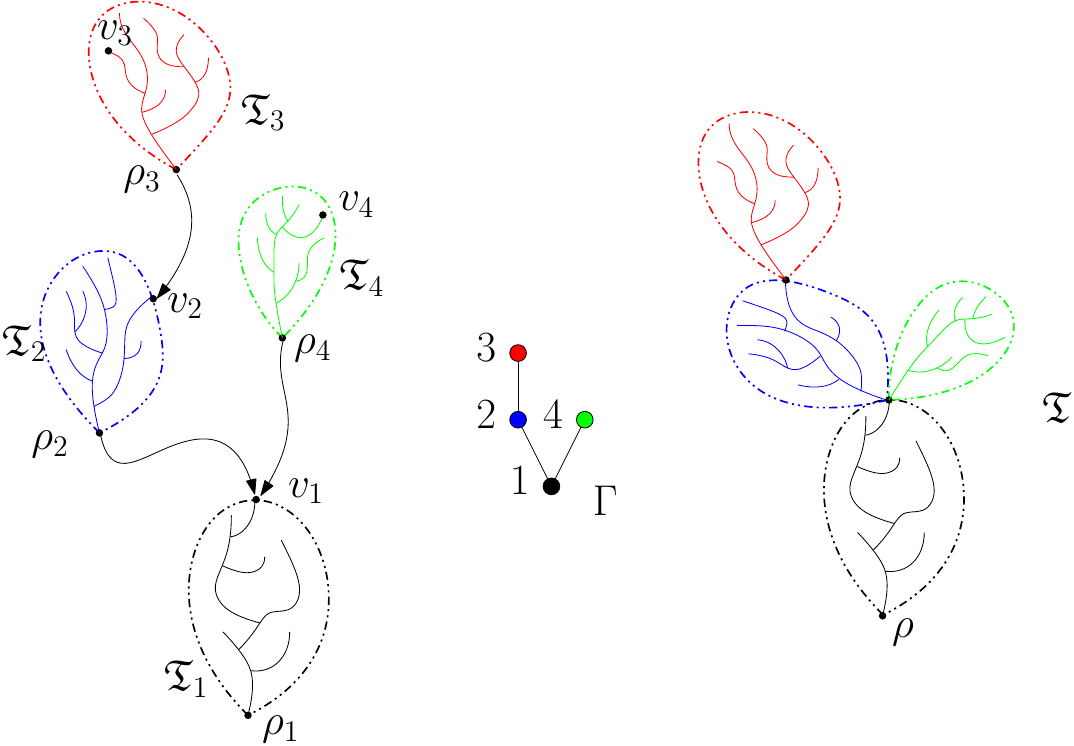}           
  \caption{\label{fig:comb_tree}The construction of $\fX$ with law $\psi(\fX_1,\fX_2,\fX_3, \fX_4,\cR, \cS)$ 
  with the structural tree $\Gamma$ shown in the middle.
  For simplicity, we have not rescaled any of the distances. An excursion point of view is depicted 
  in Figure~\ref{fig:comb_map}. (For the sake of representation, we have chosen the spaces to be tree-like.) Note that the roots $\rho_2, \rho_4$ are identified with the same point $v_1$ in the parent space $\cX_1$.}
\end{figure}

Given $\Gamma$ and $\tau$, not every $\alpha$ is admissible for \eqref{fix:tree} to have 
non-trivial solutions. The parameter $\alpha$ is chosen as the unique value such that 
 the height of an independent point sampled according to $\mu$ has finite mean, 
see the discussion of \eqref{fix:Y} below. Here, this condition can be expressed as follows. 
Recall $\Gamma_i$ defined in \eqref{def:subtree},
and define $\alpha >0$ as the unique solution to 
\begin{align} \label{def:alpha} 
\mathbf E \Bigg[\sum_{1\le i\le K} \cR_i^\alpha \I{J\in \Gamma_i}\Bigg] = 1
\quad\text{with}\quad 
\pc{J=j\,|\, \cR,\cS}= \cS_j, \quad 1\le j\le K.  
\end{align}
Such an $\alpha$ always exists and lies in the interval $(0,1)$ by monotonicity and continuity in $\alpha$ of the expected value 
and the values for $\alpha\in \{0,1\}$.
\emph{From now on, unless specified otherwise, we will always assume that $\alpha$ has been chosen 
to satisfy \eqref{def:alpha}}.

\medskip 
\noi \textsc{The height of a random point.}\
In a random compact rooted measured metric space $\fX = (\cX, d, \mu, \rho)$, heights of points
sampled according to $\mu$ play an important role. Recall $E_i$ from \eqref{def:subtree}.
In our construction,  with $\fX$ (or, rather its distribution) satisfying \eqref{fix:tree} and abbreviating $Y := d(\rho, \zeta)$ 
where $\zeta$ has distribution $\mu$ (given $\mathfrak X$), we have
\begin{align} \label{fix:Y} 
Y \stackrel{d}{=} 
\sum_{i=1}^K \I{J=i} \Bigg[ \mathcal R_i^\alpha Y^{(i)} + \sum_{j \in E_i}  \mathcal R_j^\alpha Y^{(j)}\Bigg] 
\stackrel{d}{=} \sum_{i=1}^K \beta_i^\alpha Y^{(i)},
\quad
\beta_i = \cR_i \I{J\in \Gamma_i}\,,
\end{align}
and $Y^{(1)}, \ldots, Y^{(K)}$ are distributed like $Y$, and $(J, \mathcal R, \mathcal S), 
Y^{(1)}, \ldots, Y^{(K)}$ are independent with $J$ as in \eqref{def:alpha}.
By \cite[Theorem 2]{durlig}, the distribution of $d(\rho, \zeta)$ is 
uniquely determined by \eqref{fix:Y} in the space of probability distributions on $[0, \infty)$ 
up to a multiplicative constant. (To be more precise, in the notation of \cite{durlig}, we have\footnote{The $\alpha$ in the paper \cite{durlig} should not be confused with the one 
defined in \eqref{def:alpha}.}
$\alpha = 1$ and $v'(1) < 0$ and therefore, as explained in the discussion following
Theorem 2 in \cite{durlig}, the fixed points of \eqref{fix:Y} are parametrized by their means provided their existence.) 
 \cite[Theorem 2]{durlig} further implies that, for any different choice of $\alpha$, the random variable 
$d(\rho, \zeta)$ has either infinite mean, or, almost surely, $\cX = \{\rho \}$.

Finally, we note that, from the last display, one can easily deduce that \begin{align} \label{01law} \Prob{d(\rho, \zeta) > 0} \in \{0,1\}. \end{align}

\noi\textsc{Examples.}\ {\bf 1)} The most celebrated example of a random measured real 
tree that satisfies a fixed point equation such as \eqref{fix:tree} is Aldous' 
Brownian CRT \cite{Aldous1991a,Aldous1993a}. 
Its recursive structure has been investigated in \cite{al94}, where it has been 
proved that it satisfies a fixed point equation of the type \eqref{fix:tree}, 
with $K=3$, $\Gamma$ is the tree on $\{1,2,3\}$ with $2$ and $3$ that are children 
of $1$, $\cR=\cS=(\cS_1,\cS_2, \cS_3)\sim \text{Dirichlet}(1/2,1/2,1/2)$, 
and in this case, $\alpha=1/2$.

{\bf 2)} An instance has also appeared in the context of recursive triangulations 
of the disk \cite{legalcu, BrSu2013a}. There, $K=2$ (and thus $\Gamma$ is the tree 
on $\{1,2\}$ where $2$ is a child of $1$), $\cS=(\cS_1, \cS_2)\sim \text{Dirichlet}(2,1)$,
$\cR=\cS$ and $\alpha$ turns out to be given by $(\sqrt{17}-3)/2$.

{\bf 3)} Another tree related to {\bf 2)} is given by $K=2$, $\cS=(\cS_1, \cS_2)\sim \text{Dirichlet}(2,1)$,
$\cR\sim \text{Dirichlet}(1,1)$ independent of $\cS$, and for this case, one finds $\alpha=1/3$. The tree 
has not been considered explicitly, but it appears in \cite{legalcu, BrSu2013a} via one of its encoding 
processes.

\medskip
\noi \textsc{Remark}. Recently, \citet{ReWi2016a} have also studied recursive constructions 
of continuum random trees and the corresponding geometries using fixed points arguments. 
The decompositions they consider amount to seeing the trees as a forest of rescaled copies 
of random tree grafted on some segment (a so-called (random) \emph{string of beads}) rather 
than on a copy of the tree itself.
They typically involve an infinite number of pieces. Technically, Rembart and Winkel 
use a variant of the contraction method relying on Wasserstein distances on (a slightly 
modified version of) $\bbK_\tf^{\GHP} \cap \bbT^\GH$ 
to verify uniqueness and attractiveness of fixed points. While both approaches rely 
on recursive decompositions and contraction arguments, the results seem largely disjoint. 
In particular many of the examples our results cover do not seem easily amenable to a spinal 
decomposition (e.g.\ lamination of the disk), and conversely, many examples covered by the 
results in \cite{ReWi2016a} are not covered by our combinatorial decompositions.

\subsection{Recursive decompositions: An excursion point of view}
\label{sec:rec_exc}

It is convenient to express the construction in Section~\ref{sec:rec_trees} in terms of 
excursions. It is important to note that the excursion point of view developed here 
forces the spaces to be \emph{real trees}, while in the previous section spaces were only required 
to be compact. However, we will see later on that this restriction is an important technical 
ingredient: we will prove that the support of the mass measure of any fixed point of the 
equations of interest is actually almost surely a real tree, see Theorem~\ref{thm:treeuniq} 
{\em ii)}.

Let $\Gamma$ be as in Section~\ref{sec:rec_trees}, $\br, \bs \in \Sigma_K$ and assume 
for now that $\alpha > 0$ is arbitrary. We now describe a decomposition of the unit 
interval based on the tree $\Gamma$. Each node of $\Gamma$ will be assigned two 
intervals, except the leaves that will be assigned a single one. In the following, 
we write $\partial \Gamma$ for the set of leaves of $\Gamma$ (the nodes $i$ such 
that $\Gamma_i=\{i\}$), and $\Gamma^o$ for $\Gamma\setminus \partial \Gamma$.
Set $L := 2K - |\partial \Gamma|$ and let $w_1, \ldots, w_L \in \{1, \ldots, K\}$ be 
the sequence of nodes visited by the depth first search process upon counting only 
the \emph{first} and the \emph{last} visit of a node. All nodes appear twice in this 
sequence except for leaves which appear once. 
For $i \in [K]$, let $V_i$ be the set of times $1 \leq j \leq L$ with $w_j = i$. 
More formally, for $i\in [K]$, we can set 
$v_i=i+\#\{1\le j<i: j\not\in E_i \cup \partial\Gamma\}$ and obtain 
$V_i=\{v_i, v_i+2 |\Gamma_i| - \partial \Gamma_i -1 \}$. 
Observe that $V_i$ contains a unique element if $i\in \partial \Gamma$, and two otherwise. 
For $\bu = (u_1, \ldots, u_K) \in (0,1)^K$, there is a unique decomposition of the 
unit interval into $L$ half-open\footnote{We call an interval half-open if it is of 
the form $(a,b]$ with $0 < a \leq b \leq 1$ or $[0,a]$ with $0 \leq a  \leq 1$.} 
intervals $I_1, \ldots, I_L$ (following the natural order on $[0,1]$), such that, 
\[s_i = \sum_{k \in V_i} \Leb(I_k) \text{ for all }i \in [K],
\quad \text{and} \quad 
u_i = \Leb(I_{\min V_i})/ s_i \text{ for all }i\in \Gamma^o.\]
Then, define 
\begin{equation}\label{eq:def_sets}
\Lambda_i = \bigcup_{k \in V_i} I_k 
\qquad \text{and} \qquad 
\varphi_i : \overline{\Lambda_i} \to  [0,1]
\end{equation} 
as the unique function which is bijective, monotonically increasing and piecewise 
linear with constant slope. (Here, we $\overline{A}$ denotes
the closure of a set $A$.) We can now define the continuous operator
$\Phi: \Cex^K \times \Sigma_K^2 \times (0,1)^K \to \Cex$, such 
that $g = \Phi(f_1, \ldots, f_K, \br, \bs, \bu)$ is the unique excursion satisfying 
\begin{align}\label{def:phi} 
g(x) - g(y) =  r_{w_\ell}^\alpha 
\left[f_{w_\ell}(\varphi_{w_\ell}(x)) -  f_{w_\ell}(\varphi_{w_\ell}(y))\right],  
\end{align}
for all $1 \leq \ell \leq L$ and $x,y \in I_\ell$. 
In other words, up to the scaling factor  $r_i^\alpha$ in space, the function $f_i$ 
is first fitted to an interval of length $s_i=\Leb(\Lambda_i)$ and then used on 
the set $\Lambda_i$  (which may consist of two intervals). 
For an illustration see Figure~\ref{fig:comb_map} (and compare 
it with the corresponding version involving trees on Figure~\ref{fig:comb_tree}). 
By construction, for $f_1, \ldots, f_K \in \Cex, 
\mathbf{r}, \mathbf{s} \in \Sigma_K$ and $\Xi$ uniformly distributed on $(0,1)^K$, we have
$$\Law \left( \mathfrak T_{\Phi(f_1, \ldots, f_K, \mathbf{r}, \mathbf{s}, \Xi)} \right) 
= \psi (\fT_{f_1}, \ldots, \fT_{f_K}, \mathbf{r}, \mathbf{s}).$$
Thus, the distribution of the random compact rooted measured real 
tree $(\sT_{\Xex}, d_{\Xex}, \mu_{\Xex}, \rho_{\Xex})$ satisfies \eqref{fix:tree} if 
\begin{align}\label{fix:X}
{\Xex} \stackrel{d}{=} \Phi \big({\Xex}^{(1)}, \ldots,  {\Xex}^{(K)}, \mathcal R, \mathcal S, \Xi \big), 
\end{align}
where ${\Xex}^{(1)}, \ldots, {\Xex}^{(K)}$ are independent copies of ${\Xex}$, independent of $(\cR, \cS, \Xi)$, $\Xi$ 
and $(\cR, \cS)$ being independent and $\Xi = (\xi_1, \ldots, \xi_K)$ being uniformly distributed on 
$(0,1)^K$; of course, in this case, $\alpha$ shall be chosen as in \eqref{def:alpha}.
The fixed point equation in \eqref{fix:X} can be expressed alternatively as
\begin{align*} 
{\Xex}(\cdot) & \stackrel{d}{=} 
\sum_{i=1}^K \mathcal R_i^\alpha
\Bigg [   \mathbf{1}_{\Lambda_i}(\cdot) {\Xex}^{(i)}(\varphi_i(\cdot)) +  \sum_{j\in \Gamma_i\setminus \{i\}}\mathbf{1}_{\Lambda_{j}} (\cdot)  {\Xex}^{(i)}(\xi_{i})\Bigg ]\, . 
\end{align*}
Let us also note that, when asking 
for tree solutions to \eqref{fix:tree}, the excursion point of view of the recursive 
decomposition is technically preferable since it grants access to random excursions (and their 
corresponding encoded real trees) as well as to nodes sampled independently according to the 
mass measure using canonical external randomness. 

\begin{figure}[t] \label{fig:phi}
\begin{center}
\includegraphics[scale=.8]{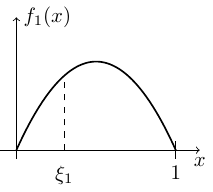}
\includegraphics[scale=.8]{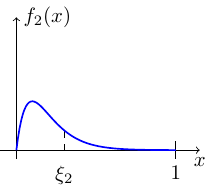}
\includegraphics[scale=.8]{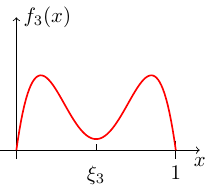}
\includegraphics[scale=.8]{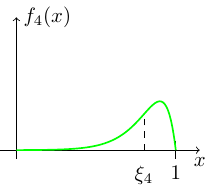}
\begin{minipage}{1.5cm}
\begin{tikzpicture}[xscale=8,yscale=1.5, domain=0:1]
\put(150,-100) {\includegraphics[width=1.2cm]{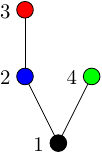} }   
\end{tikzpicture}
\end{minipage}
\end{center}
\begin{center}
\includegraphics[scale=.8]{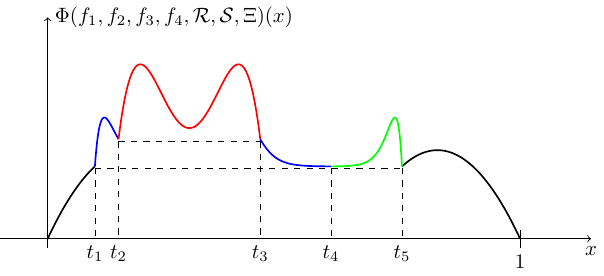}
\caption{\label{fig:comb_map}
An example of the functional construction of Section~\ref{sec:rec_exc}:
Here $K=4$, $L=6$, the structural tree $\Gamma$ is the tree shown in the center right. 
The functions $f_1,f_2,f_3,f_4$ are composed into $\Phi(f_1,f_2,f_3,f_4, \cR, \cS, \Xi)$. 
In order to keep the focus on the structure of the construction, we have not used scalings 
for the distances and the scaling $\cS=(0.35,0.20,0.30,0.15)$ for time. 
Observe that, in the tree encoded by $f_1$, the point corresponding to $\xi_1$ is not a leaf. 
The corresponding point of view using trees is depicted in Figure~\ref{fig:comb_tree}.}
\end{center}
\end{figure}

\medskip \noi\textsc{Examples.}\ \label{ex:examples}
{\bf 1)} An identity of the kind in \eqref{fix:X} holds for the Brownian 
excursion $\mathbf e$, and is of course intimately related to the corresponding decomposition of the Brownian
CRT. By \cite[Corollary 3]{al94}, we have
\begin{align} \mathbf e \stackrel{d}{=}  \Phi \big(\mathbf e^{(1)}, \mathbf e^{(2)}, \mathbf e^{(3)}, \Delta, \Delta, \Xi \big), \label{eqe}
\end{align}
with conditions as in \eqref{fix:X}, where $\Delta\sim \text{Dirichlet}(1/2,1/2,1/2)$, $\alpha = 1/2$, and
$\Gamma$ is the tree of size three with nodes $2$ and $3$ attached to the root.

{\bf 2)} The functional version of the fixed point equation appearing in Example {\bf 2)} of 
Section~\ref{sec:rec_trees} concerns a certain process $\sZ$ that satisfies
\begin{align}\label{eqZ}
\sZ \stackrel{d}{=}  \Phi \big(\sZ^{(1)}, \sZ^{(2)}, \Delta, \Delta, \Xi \big) \end{align}
with conditions as in \eqref{fix:X}, where $\Delta\sim \text{Dirichlet}(2,1)$ and $\alpha = (\sqrt{17}-3)/2$.

{\bf 3)} Similarly to {\bf 2)} above, there is a functional version to Example~{\bf 3)} of 
Section~\ref{sec:rec_trees}. The process $\sH$ involved satisfies the following modified fixed point
equation 
\begin{align}\label{eqH}
\sH \stackrel{d}{=}  \Phi \big(\sH^{(1)}, \sH^{(2)}, \Delta, (W,1-W), \Xi \big) \end{align}
with conditions as in \eqref{fix:X}, where $\Delta\sim \text{Dirichlet}(2,1)$, $W$ follows the uniform distribution on $[0,1]$, $\Delta, W$ are independent and $\alpha = 1/3$.

\subsection{Fractal properties of metric spaces} \label{subsec:fr}

Let $(S,d)$ be a metric space. (In this section, we do not assume the space to be compact.)
For $\delta > 0$ and a non-empty relatively compact set $B$ 
let $N_B(\delta)$ be the smallest number $m$ such that there exist
$m$ open balls of radius $\delta$ covering $B$. 
We define the \emph{lower 
Minkowski dimension} $\DimuM (B)$ and the \emph{upper Minkowski dimension} $\DimoM (B)$ by
\begin{align*}
\DimuM (B)  := \liminf_{\delta \to 0} \frac{\log N_B(\delta)}{-\log \delta}, \quad\text{and}\quad \DimoM (B)  := \limsup_{\delta \to 0} \frac{\log N_B(\delta)}{-\log \delta}. 
\end{align*}
If both values coincide, we simply call $\DimM (B) := \DimuM (B)$ the 
\emph{Minkowski dimension} of $B$. 

The \emph{Hausdorff dimension} of a set $A \subseteq S$ is defined using 
the family of (outer) Hausdorff measures $(H^s)_{s \geq 0}$ given by
\begin{align*}
H^s (A) 
:= \lim_{\delta \to 0} \Bigg\{ \inf \Bigg \{\sum_{i \geq 1} |U_i|^s: 
A \subseteq \bigcup_{i \geq 1} U_i \text{ and } |U_i| \leq  \delta 
\text{ for all } i \ge 1  \Bigg \} \Bigg \}.
\end{align*}
(The map $H^s: \mathbb K^\GH \to [0, \infty]$ is measurable, see Lemma~\ref{lem:Hs_meas} in the appendix for a proof.)
Here, for a set $U \subseteq S$, we let $|U| = \sup \{|x -y| : x, y \in U\}$. 
The Hausdorff dimension of $A \subseteq S$ is now defined by
$$\DimH (A) := \inf \{ s \geq 0: H^s(A) = 0 \},$$ 
where one should notice that $H^t(A) < \infty$ implies $H^s(A) = 0$ for $s > t$.
We will need the following version of the mass distribution principle 
(see, e.g.\ \cite[Proposition 4.9]{Falconer1990a} for a formulation in $\R^d$ which, 
together with its proof, applies analogously in any separable metric space). 
Let $B_r(x) := \{y \in S: d(x,y) < r\}$ denote the ball of radius $r$ around $x$.
Then, for a measurable set $A \subseteq S$, a finite measure $\nu$ on $S$ with 
$\nu(A) > 0$ and $c > 0$, we have
\begin{align} \label{massdistributionprinciple}
\limsup_{r \to 0} \nu(B_r(x))/r^s \leq c \: \text{ for all} \: x \in A 
\quad \Rightarrow \quad 
\DimH(A) \geq s. 
\end{align}
Lower and upper Minkowski dimension as well as Hausdorff dimension are invariant 
under bijective isometries. Furthermore, for a non-empty relatively compact set $B$, we have 
$\DimH (B) \leq \DimuM (B) \leq \DimoM (B).$

Recall that a function $f: [0,1] \to \R$ is $\omega$-H{\"o}lder continuous 
with $0 < \omega \leq 1$ (or H\"older with exponent $\omega$)
if there exists a constant $C > 0$ such that 
$$|f(x) - f(y)| \leq C |x-y|^\omega, \quad 0 \leq x,y \leq 1.$$
For $f \in \Cex$, we let $\omega_f$ be the supremum of all H{\"o}lder exponents over
excursions which are time-change equivalent to $f$. 

Next, for $p > 0$, the $p$-variation $[f]_p(t), t \in [0,1],$ of a function $f \in \Co$ is 
defined by
$$[f]_p(t) = \sup\left\{\sum_{i = 0}^{n-1} |f(t_{i+1}) - f(t_i)|^p 
: n \ge 1 \text{ and } 0 \le t_0 \leq \ldots \leq t_n \leq t\right\}\,.$$ 
Note that $[f]_p$ is monotonically increasing and, if $[f]_p(1) < \infty$, then $[f]_p \in \Co$. 
It is easy to see that $\inf \{p > 0 : [f]_p(1) < \infty\} \leq \omega_f^{-1}$. 
Furthermore, if $f$ is nowhere constant, the converse inequality follows from 
Theorem 3.1 in \cite{chisgalk}. Combining this with Theorem 3.1 in \citet{picard08} shows that, for any 
$f \in \Cex$ which is nowhere constant, we have
\begin{align} \label{aiz}
\DimoM (\cT_f) = \frac 1 {\omega_f}.
\end{align}
The last identity is used in Corollary \ref{cor:opt_Hol}.  

\section{Main results}\label{sec:results}

\subsection{Characterizations of solutions to \eqref{fix:tree} and \eqref{fix:X}}

Our first theorem clarifies the set of solutions to the fixed point equations \eqref{fix:tree} 
and \eqref{fix:X}. Recall that these involve the structural tree $\Gamma$, a probability distribution $\tau$ on $\Sigma_K^2$ and the 
value $\alpha \in (0,1)$ satisfying 
\eqref{def:alpha}.

\begin{thm} \label{thm:treeuniq}
Let $c > 0$. 
Then, 
\begin{compactenum}[i)]
\item there exists a unique continuous excursion ${\Xex}$ (in distribution) that 
satisfies \eqref{fix:X} and $\Ec{ {\Xex}(\xi)} = c$ where $\xi$ follows the uniform distribution on $[0,1]$ and $\Xex, \xi$ are independent;
\item for any $\bbK^\GHP$-valued r.v.\ $\fX$ satisfying 
\eqref{fix:tree} with  $\Ec{d(\rho, \zeta)} = c$, where $\zeta$ is chosen according to $\mu$ (given $\fX$), 
the random tree $\fT_{\Xex}$ encoded by $\Xex$, and  $(\text{supp}(\mu), d, \mu, \rho)$ have 
the same distribution. 
\end{compactenum} 
Furthermore, for all $m \geq 1$, we have $\E{\|{\Xex}\|^m}  < \infty$, and, almost surely,
\begin{compactenum}
\item[iii)] ${\Xex}(s) > 0$ for all $s \in (0,1)$;
\item[iv)] ${\Xex}$ is nowhere monotonic.
\end{compactenum} 
\end{thm}

Some comments are in order. First of all, as motivated by the formulation of point \emph{ii)}, there can exist further solutions to \eqref{fix:tree} which are not almost surely continuum trees, see Proposition \ref{prop:counter_ex} below and the example discussed following Theorem 1.6  in \cite{albgol2015}. 
Next, there does not exist any random compact rooted measured metric space 
with values in $\bbK^{\GHP}$ or 
$\bbK^{\GP}$ solving \eqref{fix:tree} for which one would have $\Ec{d(\rho, \zeta)} = \infty$.
Similarly, any solution to \eqref{fix:tree} with values in $\bbK^{\GP}$ or $\bbK^\GHP_{\tf}$ satisfying $\Ec{d(\rho, \zeta)} = 0$ must a.s.\ be the trivial space reduced to $\{\rho\}$.

Albenque and Goldschmidt \cite{albgol2015} show that, in the case of the Brownian CRT, the fixed point in Example {\bf 1)} of Section~\ref{sec:rec_trees} is attractive 
with respect to weak convergence on the space of probability measures on $\bbK ^{\GP}$. 
They also raise the question whether this was true with respect to the 
Gromov--Hausdorff--Prokhorov distance. Our next result confirms 
that it is indeed the case, under certain moment conditions. (Note however that the results 
in \cite{albgol2015} are not directly comparable to ours since the trees there are 
\emph{unrooted}.)

In the following, we let $\phi_{\GHP}: \cM_1(\bbK^{\GHP}) \to \cM_1(\bbK^{\GHP})$ be 
the map that to  $\aleph\in \cM_1(\bbK^\GHP)$ associates $\phi_{\GHP} (\aleph) = 
\Psi(\aleph, \tau)$, where $\tau$ is the given probability distribution on
$\Sigma_K^2$ and $\Psi$ is the map defined in \eqref{eq:def_Psi}. 
For $n \geq 1$, we write $\phi^n_{\GHP}$ for the $n$-th iterate of $\phi_{\GHP}$. 
Analogously, we define $\phi_{\GP}$ and $\phi_\GP^n, n \geq 1$ for probability measures 
on $\bbK^{\GP}$. Finally, for $\nu \in \cM_1(\bbK^{\GHP})$ or $\cM_1(\bbK^\GP)$, 
we write $\e_\nu$ for the expectation with respect to spaces sampled from $\nu$.

\begin{thm}\label{thm:unique_tree-GHP}
 Fix $c>0$ and let ${\Xex}$ be the unique continuous excursion (in distribution) satisfying \eqref{fix:X} with 
$\Ec{{\Xex}(\xi)} = c$ from Theorem \ref{thm:treeuniq}.
Then, we have the following two statements: 
\begin{compactenum}[i)]
\item if $\nu \in \mathcal M_1(\bbK^\GP)$ with $\e_\nu[d(\rho, \zeta)]=c$, 
then the sequence  of distributions $(\phi_{\GP}^n(\nu))_{n\ge 1}$ converges weakly to the law of $\fT_{\Xex}$.
\item if $\nu \in \mathcal M_1(\bbT^\GHP_\tf)$ 
with $\e_\nu[d(\rho, \zeta)]=c$ and $\e_\nu[\|\mathfrak T\|^{\fm}] < \infty$, where 
\begin{equation}\label{eq:def_fm}
\fm = 1+\lfloor 1/\alpha \rfloor = \min \Bigg\{ m \in \N : \sum_{i=1}^K \E{\mathcal R_i^{m\alpha}} < 1\Bigg\}\,,
\end{equation}
then the sequence of
distributions $(\phi_{\GHP}^n (\nu))_{n\geq 1}$ converges weakly to the law of $\fT_{\Xex}$.
\end{compactenum}
\end{thm}

Note that the assumptions in Theorem~\ref{thm:unique_tree-GHP} {\em i)}  are rather weak as 
no conditions on the probability distribution $\nu$ are imposed apart from 
$\e_\nu[d(\rho, \zeta)]=c$. For instance, if $\nu$ charges only two-point metric 
spaces $\{\rho,x\}$ such that $d(\rho,x)=c$, the theorem applies. Although the moment 
condition in Theorem~\ref{thm:unique_tree-GHP} {\em ii)} is probably not optimal,
one certainly needs some condition as demonstrated by the following proposition. 

\begin{prop}\label{prop:counter_ex}
Let $c>0$. Fix $\Gamma$ and let $\eta$ be a probability distribution on $\Sigma_K$. 
Then, there exists a distribution $\tau$ on $\Sigma_K^2$ such that
$\tau(\Sigma_K\times \cdot\,) = \eta(\cdot)$, and furthermore, with $\alpha$ as in \eqref{def:alpha},
\begin{compactenum}[i)]
\item  there exist infinitely many mutually singular fixed points of \eqref{fix:tree} on 
$\bbK^\GHP$  including some that are almost surely not real trees, such that, 
writing $\nu \in \cM_1(\bbK^\GHP)$ for such a distribution, and $\fX=(\cX,d,\mu,\rho)$ 
distributed according to $\nu$, we have $\e_{\nu}[d(\rho,\zeta)]=c$ and
$\Ec{ \|\fX\| ^{1/\alpha}} =\infty$;
\item there exists $\nu \in \cM_1(\bbK^\GHP) $ concentrated on $\bbT^{\GHP}_\tf$ such 
that $\e_{\nu}[d(\rho,\zeta)]=c$, $\e_{\nu}[\|\fT\|^{1/\alpha}]=\infty$ and
$(\phi_\GHP^n(\nu))_{n\ge 1}$ does not converge weakly to the law of $\fT_{\Xex}$ with
$\Xex$ as in Theorem~\ref{thm:unique_tree-GHP}.
\end{compactenum}
\end{prop}

Let us comment on {\em ii)} above: According to Theorem~\ref{thm:unique_tree-GHP} 
{\em i)}, under these conditions, one has convergence in the sense of Gromov--Prokhorov 
to $\fT_\Xex$, where $\Xex$ is the random excursion of Theorem~\ref{thm:unique_tree-GHP}; 
however, as one iterates $\phi_\GHP$, the support of the mass measure converges 
to $\fT_\Xex$, but some of the branches of the tree are drained of their mass without 
becoming shorter (see the proof on page~\pageref{pf:counter_ex} for details).

\subsection{Geometry, fractal dimensions and optimal H\"older exponents} \label{subsec:main2}

It is informative to first present a heuristic argument that, at the very least, gives an 
idea of the value of the Minkowski dimension that one should expect. 
Consider $\fT = (\sT, d, \mu, \rho)$ 
satisfying \eqref{fix:tree}. In a covering of $\sT$ by open balls, if one neglects the contribution of 
the balls that might intersect more than one subtree in the recursive decomposition, the fixed point 
equation \eqref{fix:tree} should imply that we approximately have
\begin{align}\label{heuris}
N_{\sT} (\delta) 
\approx & \sum_{i=1}^K N_{\mathcal R_i^\alpha \sT_{i}}(\delta) 
= \sum_{i=1}^K N_{\sT_{i}}(\mathcal R_i^{-\alpha} \delta). 
\end{align}
In particular, if one roughly has $N_{\sT}  (\delta) \approx \delta^{-s}$ for some $s>0$, then 
it should be the case that $1 = \sum_{i=1}^K \cR_i^{\alpha s}$, and thus the constant $s$ 
should be given by $s = \alpha^{-1}$. We now provide results that justify that this is 
indeed the case under some conditions which happen to be satisfied in most examples.

In the following theorem and subsequently, we use the  generic random variable
$\xoverline { \cR}$ which is distributed like $\cR_I$ with $I$  independent of $\cR$
and uniformly chosen among $1, \ldots, K$. 

\begin{thm}[Upper bound on $\DimoM$] \label{thm:main2}
Assume that $\fT = (\sT, d, \mu, \rho)$ satisfies \eqref{fix:tree} with values in
$\bbT_\tf^\GHP$. Moreover, assume that $\xoverline{\cR}$ admits a density on $(0,1)$ (although an atom at $0$ is allowed).
Then, almost surely, $$\DimoM (\sT) \leq 1/  \alpha.$$  
\end{thm} 


\begin{thm}[Lower bound on $\DimH$]\label{thm:main3}
Suppose that $\fT = (\sT, d, \mu, \rho)$  satisfies \eqref{fix:tree} with values 
in $\bbT_\tf^\GHP$, \sloppy that $\pc{\sT \neq \{\rho\}} > 0$,  and, for some 
$\delta > 0$, $\Ec{\xoverline {\cR}^{-\delta}} < \infty$.  
Then $$\DimH (\sT) \geq 1/\alpha.$$
\end{thm}


\medskip
In the next result, recall that two excursions $f$ and $g$ are said to be equivalent if
$f = g \circ \phi$ for a function $\phi : [0,1] \to [0,1]$ that is continuous and strictly 
increasing. 
\begin{cor}[Optimal H\"older exponents]\label{cor:opt_Hol}
Suppose that ${\Xex}$ satisfies \eqref{fix:X}, that $\pc{\| {\Xex} \| > 0} > 0$ and that 
the conditions of Theorems~\ref{thm:main2} and~\ref{thm:main3} are satisfied. Then, almost surely:
\begin{compactenum}[(a)]
    \item $\DimH(\cT_{\Xex})= \DimuM(\cT_{\Xex}) = \DimoM(\cT_{\Xex}) = \alpha^{-1}$;
    \item for any $\gamma < \alpha$, there exists a process $\tilde \Xex$ which is equivalent to ${\Xex}$ and has 
    $\gamma$-\Hoelder continuous paths. 
        \item for any $\gamma > \alpha$ and any distribution $\tilde \nu$ on $\Cex$ such that, for $\Law(\tilde \Xex) = \tilde \nu$, $\cT_{\tilde\Xex}$ and $\sT_{\Xex}$ have the same distribution, $\tilde \Xex$ has $\gamma$-\Hoelder continuous paths with probability zero. 
\end{compactenum}
\end{cor}

Finally, we have the following results about the degrees in $\fT_{\Xex}$. Recall 
that the degree of a point in a real tree is defined as the number of connected components in which the space decomposes upon removal of the point. 
Let $\sD(\cT_{\Xex})$ be the (random) set of degrees of $\cT_{\Xex}$. 
Let also $\sD(\Gamma)=\{1+\#\{j:\varpi_j=i\}: 1\le i\le K\}$; $\sD(\Gamma)$ is the 
set of degrees in the tree obtained from $\Gamma$ by connecting an additional node to its root.
Observe that $1\in \sD(\Gamma)$, but that it is possible that $2$ is not an element 
of $\sD(\Gamma)$. 

\begin{prop}\label{prop:degrees} Let $0 < c < \infty$ and ${\Xex}$ be the unique solution (in distribution) of \eqref{fix:X} in Theorem~\ref{thm:treeuniq} with $\Ec{{\Xex}(\xi)} = c$ and consider $\fT_\Xex=(\cT_{\Xex}, d_{\Xex}, \mu_{\Xex}, \rho_{\Xex})$. Then, 
almost surely, 
\begin{compactenum}[i)]
\item the root $\rho_{\Xex}$ is a leaf;
\item a point sampled from $\mu_{\Xex}$ is a leaf;
\item $\mu_{\Xex}$ has no atoms;
\item the set $\sD(\cT_{\Xex})$ of degrees of points in $\cT_{\Xex}$ is fully determined by $\sD(\Gamma)$: we have
$$
\sD(\cT_{\Xex})=
\left\{
\begin{array}{ll}
\{1,2,3\} & \text{if } \sD(\Gamma)=\{1,2\}\\
\sD(\Gamma)\cup \{2\}& \text{otherwise}\,.
\end{array}
\right.
$$
\end{compactenum}
\end{prop}

\subsection{A taste of applications}\label{sec:taste_app}

All applications in this work are discussed in detail in Section~\ref{sec:app}; they cover in particular
generalizations of the trees dual to laminations of the disk. Here, we only state 
immediate consequences for the three trees encoded by the functions $\sZ$, $\sH$ and
 $\mathbf e$ that we have used 
as examples earlier (on page~\pageref{ex:examples}). Observe that, the structural tree $\Gamma$ is fixed if $K=2$, while for $K \geq 3$, 
in order to describe it, it suffices to specify the parents $\varpi_3, \ldots, \varpi_K$ of the nodes $i= 3,\dots, K$.  Recall that continuum real tree refers to an element of $\bbT^{\GP}$ or $\bbT^{\GHP}_\tf$.

\begin{cor} \label{coro1}
Up to a multiplicative constant for the distance function, we have (uniqueness being understood in 
distribution):
\begin{compactenum}[(a)]
\item The Brownian CRT $(\sT_{\mathbf e}, d_{\mathbf e}, \mu_{\mathbf e}, \pi_{\mathbf e}(0))$ is the 
unique continuum real tree satisfying \eqref{fix:tree} with $K = 3$, $\varpi_3 = 1$, $\alpha = 1/2$,  
$\mathcal S = \cR= \Delta$ 
 as in \eqref{eqe}.
\item The tree $(\sT_\sZ, d_\sZ, \mu_\sZ, \pi_\sZ(0))$ is the unique  continuum real tree satisfying \eqref{fix:tree} with $K = 2$, $\cS = \cR = \Delta$, 
  as in \eqref{eqZ} and 
$\alpha=(\sqrt{17}-3)/2$.
\item The  tree $(\sT_{\sH}, d_{\sH}, \mu_{\sH}, \pi_\sH(0))$ is the unique  continuum real tree satisfying \eqref{fix:tree} with $K = 2$, $\alpha = 1/3$, $\cS = \Delta$, 
$\cR = (W, 1-W)$ with $\Delta, W$ as in \eqref{eqH}.
\end{compactenum}
Furthermore, these uniqueness results also hold in the spaces $\bbK_\tf^\GHP$ and $\bbK^\GP$. 
\end{cor}

The assertion concerning the Brownian CRT (in the unrooted set-up and using a slightly different 
definition of continuum trees) in Corollary~\ref{coro1} 
is the main result in \cite{albgol2015}. Similarly, the claim for the process $\sZ$ in Corollary~\ref{cor:uni} 
has already been given in \cite{BrSu2013a} in the space $\Co$ under additional moment assumptions. In terms of processes, we have the following

\begin{cor} \label{cor:uni}
Up to a multiplicative constant, we have (uniqueness being understood in the sense of distributions):
\begin{compactenum}[(a)]
\item The Brownian excursion $\mathbf e$ is the unique continuous excursion satisfying \eqref{eqe}. 
\item The process $\sZ$ is the unique continuous excursion satisfying \eqref{eqZ}.   
\item The process $\sH$ is the unique continuous excursion satisfying \eqref{eqH}. 
\end{compactenum}
\end{cor}

We now formulate the implications of results on the fractal dimensions of $\cT_\sZ$, $\cT_{\sH}$ and $\cT_{\mathbf e}$. 
Note that the fractal dimension of the continuum random tree has 
already been established in \cite{Aldous1991b} (see also \cite{DLG05} for the more general case 
of L{\'e}vy trees) and the Minkowski dimension of $\sT_\sZ$ in \cite{BrSu2013a}. 

\begin{cor} \label{cor:main}
Almost surely, for the processes $\mathbf e$, $\sZ$ and $\sH$ satisfying \eqref{eqe}, \eqref{eqZ} and \eqref{eqH}, \begin{compactenum}[(a)]
\item $\DimM (\sT_{\mathbf e}) = \DimH (\sT_{\mathbf e}) = 2$.
\item $\DimM (\sT_\sZ) = \DimH (\sT_\sZ) = 2/(\sqrt{17}-3)$,
\item $\DimM (\sT_{\sH}) = \DimH (\sT_{\sH}) = 3$.  
\end{compactenum}
\end{cor} 

\subsection{Overview of the main techniques} \label{tech-expan}

Most of our proofs  rely on an expansion of the fixed point equation  \eqref{fix:X}. 
For fixed point equations describing real-valued distributions, this idea is classical, 
see, e.g.\ \cite[Section 2.3]{alba}. For the Brownian CRT $\sT_{\mathbf e}$, it was used in \cite{hamcro} as well as, in \cite{albgol2015}.

Let $\Theta = \bigcup_{n \geq 0} [K]^n$ 
be the complete infinite $K$-ary tree and, for each $n \geq 0$, denote 
by $\Theta_n = [K]^n \subseteq \Theta$ the set of vertices on level $n$ in $\Theta$. 
Next, 
$$
\{(\mathcal R^\vartheta, \mathcal S^\vartheta): \vartheta \in \Theta\}, 
\qquad \text{and}\qquad 
\{\Xi^\vartheta = (\xi_1^\vartheta, \ldots, \xi_K^\vartheta) : 
\vartheta \in \Theta\}$$ 
be two independent sets of independent random variables where 
each $(\cR^\vartheta, \cS^\vartheta)$ has distribution $\tau$ and each $\Xi^\vartheta$ is 
uniformly distributed on $(0,1)^K$.
The components of $\cR^\vartheta$ and $\cS^\vartheta$ are assigned to the edges out of $\vartheta$ 
as follows: for the edge $e^\vartheta_i$ between $\vartheta$ and $\vartheta i$, define 
$\cR(e^\vartheta_i)=\cR^\vartheta_i$ and $\cS(e^\vartheta_i)=\cS^\vartheta_i$.
These edge-weights then induce values for the vertices that we define multiplicatively: 
each node $\vartheta\in \Theta$ is assigned a \emph{length} $\scL(\vartheta)$ and 
a \emph{rescaling
factor for distances} $\cV(\vartheta)$ which are given by  
\begin{align}
\cV(\vartheta) = \prod_{e \in \pi_\vartheta} \cR(e)
\qquad \text{and} \qquad
\scL(\vartheta) = \prod_{e \in \pi_\vartheta} \cS(e) \label{def_ell}, 
\end{align} 
where $\pi_\vartheta$ denotes the set of edges on the path from the root $\emptyset$ to $\vartheta$. 
We can think of $\Theta$ as providing the parameters that are required by the recursive 
decomposition. 

In Section~\ref{sec:construction_as}, we will see that, for any $c > 0$, one can construct a 
family of random excursions $\{ \Xex^\vartheta : \vartheta \in \Theta\}$, such that
\begin{align*} 
\Xex^{\vartheta} =  
\Phi \big(\Xex^{\vartheta 1}, \ldots, 
\Xex^{\vartheta K}, \mathcal R^\vartheta, \mathcal S^\vartheta, \Xi^\vartheta \big),
\end{align*}
where, for all $\vartheta \in \Theta$, the distribution of $\Xex^\vartheta$ does not depend 
on $\vartheta$, $\Xex^\vartheta$ is measurable with respect to $\{ (\cR^{\vartheta \sigma}, \cS^{\vartheta \sigma}, 
\Xi^{\vartheta \sigma}) : \sigma \in \Theta \}$, and $\Ec{\Xex^\vartheta(\xi)} = c$.

It should be clear that, for any $\vartheta \in \Theta$ and $n \geq 0$, 
the unit interval is decomposed in half-open intervals such that, for 
$\sigma \in \Theta_n$,  $\Xex^{\vartheta \sigma}$ multiplied by $(\cV(\vartheta \sigma)/\cV(\vartheta))^\alpha$ 
governs the behaviour of the process $\Xex^\vartheta$ on a set $\Lambda^\vartheta_\sigma$ of 
Lebesgue measure $\scL(\vartheta \sigma)/\scL(\vartheta)$ composed of a subset of these intervals. 
Let us give a precise formulation of this decomposition: first, for all $\vartheta \in \Theta$ 
set $\Lambda_\emptyset^\vartheta = [0,1]$ and 
$j \in [K]$, let the set $\Lambda^{\vartheta}_{j}$ and the function 
$\varphi^{\vartheta}_{j}$ be defined as $\Lambda_j$ and $\varphi_j$ in \eqref{eq:def_sets} 
using the vector $(\cS^\vartheta, \Xi^\vartheta)$. Then, given $\Lambda^\vartheta_\sigma$ and 
$\varphi^\vartheta_\sigma$ for $\sigma \in \Theta_n, n \geq 1$, for $j \in [K]$, let 
\begin{align} \label{phie} \Lambda^\vartheta_{\sigma j} 
= \big( \varphi^\vartheta_\sigma \big)^{-1} (\Lambda^{\vartheta \sigma}_j), \end{align}
and let $\varphi^\vartheta_{\sigma j}$ be the unique piece-wise linear bijective and 
increasing function with constant slope mapping $\overline{\Lambda^\vartheta_{\sigma j}}$ onto $[0,1]$. 
For $n \geq 0$, $\Lambda_\sigma^\vartheta$ is the disjoint union of the 
sets $\Lambda^\vartheta_{\sigma \omega}, \omega \in \Theta_n$. In particular, 
$\Lambda^\vartheta_{\omega}, \omega \in \Theta_n$ is a partition of the unit interval.
Throughout the paper, we write $\Lambda_\vartheta = \Lambda^\emptyset_\vartheta$,
$\vartheta \in \Theta$ and $\Xex = \Xex^\emptyset$ for the quantities at the root of $\Theta$. 
(No confusion should arise in the notation as $\Xex^\emptyset$ is indeed the process from 
Theorem~\ref{thm:treeuniq}.)

\subsection{Organization of the proofs}

The remainder of the paper is organized as follows: In Section~\ref{sec:proofs}, we prove 
Theorems~\ref{thm:treeuniq} \emph{i), ii)}, Theorem~\ref{thm:unique_tree-GHP} and 
Proposition~\ref{prop:counter_ex}. Section~\ref{sec:un1} is devoted to showing that, 
up to a scaling constant, there exists at most one solution to \eqref{fix:tree} in 
$\bbK_\tf^{\GHP}$ (or $\bbK^{\GP}$). In Section~\ref{sec:construction_as} this solution is 
constructed together with a unique continuous excursion satisfying \eqref{fix:X}. 
Parts \emph{i)} and \emph{ii)} of Theorem~\ref{thm:treeuniq} are proved at the end of 
this section. Section~\ref{sec:attractiveness_fix} contains the proofs of 
Theorem~\ref{thm:unique_tree-GHP} and Proposition~\ref{prop:counter_ex}.

Section~\ref{sec:proofs_frac} contains the proofs to the remaining statements presented in 
Section~\ref{sec:results}. In Section~\ref{sec:upbound}, we start with the verification 
of Theorem~\ref{thm:main2}. In Section \ref{sec:lowerhaus}, we give the proof of the 
lower bound on the Hausdorff dimension of Theorem~\ref{thm:main3}. In 
Section~\ref{sec:proof_degrees}, we discuss the proofs of Proposition~\ref{prop:degrees} 
and the statements \emph{iii)} and \emph{iv)} of Theorem~\ref{thm:treeuniq}. 
Corollary~\ref{cor:opt_Hol} is discussed in Section~\ref{sec:proof_hoelder}. 

Finally, Section~\ref{sec:app} is dedicated to applications. We discuss the results 
formulated in Section~\ref{sec:taste_app} for the processes $\mathbf e, \sZ$ and $\sH$ in 
detail in Sections~\ref{sec:e}--\ref{sec:H}. Section~\ref{sec:kang} contains new results 
concerning a generalization of the lamination model \cite{legalcu, BrSu2013a}.

\section{Proofs of the existence and uniqueness results}\label{sec:proofs}

\subsection{Uniqueness of the encoding function}  \label{sec:un1}

The proof of Theorem \ref{thm:treeuniq} consists of two steps. First, in Proposition~\ref{prop:treeuniq}, we show that, in distribution, there exists at most 
one compact rooted measured metric space with full support satisfying \eqref{fix:tree}. Second,  
as in \cite{BrSu2013a}, we use a variant of the functional contraction method  developed in \cite{NeSu2014a} to construct a solution 
to \eqref{fix:X} whose supremum has finite moments of all orders. We indicate how to obtain the statements \emph{i)} and \emph{ii)}  in Theorem~\ref{thm:treeuniq} 
from the next two results right after proving Proposition \ref{prop:conlimit}.

For $n \in \N = \{0, 1, 2, \ldots\}$, let $\mathbb M_n = \{ (m_{ij})_{0 \leq i,j \leq n} : 
m_{ij} = m_{ji} \geq 0, m_{ii}=0 \} $ be the set of symmetric $(n+1)$ by $(n+1)$ matrices with 
non-negative entries and zeros on the diagonal. We also write $\mathbb M_{\N}$ for the set of 
infinite dimensional matrices satisfying these properties. We always endow this set with the 
product topology and the corresponding Borel $\sigma$-field. For a fixed compact rooted 
measured metric space $\fX = (\cX, d, \mu, \rho)$, let $(\zeta_i)_{i\ge 1}$ be independent and identically distributed (i.i.d.) test
points with distribution $\mu$ on $\cX$, and set $\zeta_0 = \rho$. 
Observe that the distribution of the random infinite matrix $\mathfrak D_\fX = 
(d(\zeta_i, \zeta_j))_{i, j \geq 0}$ 
does not depend on the representative of the Gromov--Hausdorff--Prokhorov 
(or Gromov--Prokhorov) isometry class of $\fX$. Hence, we can define the  
\emph{distance matrix distribution} $\nu^\fX \in \cM_1(\mathbb M_{\N})$ for elements $\fX$ of 
$\bbK^{\GHP}$ (or $\bbK^{\GP})$\footnote{The map $\fX\to \nu^\fX$ is 
continuous for the Gromov--Prokhorov topology (via the equivalence with the so-called Gromov-weak 
topology; see \cite{GrPfWi2009a,DeGrPf} for details).}. 
For a probability distribution $\kappa \in \cM_1(\bbK^{\GHP})$
(or $\cM_1(\bbK^{\GP})$), we also define the probability measure
$$\nu^\kappa(A) := \int \nu^x(A) d \kappa(x), 
\quad A \subseteq \mathbb M_{\N} \quad \text{measurable}.$$ 
The importance of $\nu^\fX$  and $\nu^\kappa$ is highlighted by the following well-known proposition. For $\GP$-isometry classes, part \emph{i)} is typically referred to as Gromov's reconstruction theorem, see \citet[Section 3 1/2]{Gromov1999} or \citet[Theorem 4]{vershik}. 
Statements \emph{ii)}  and \emph{iii)}  (again in the $\GP$ case and for unrooted structures) are covered by Corollary 3.1 in \cite{GrPfWi2009a} (see also \cite[Corollary 2.8]{Loehr}). For $\GHP$-isometry classes, these results immediately follow from the bimeasurability of~$\iota$  
discussed at the end of Section \ref{sec:metrics}.

\begin{prop} \label{prop:distghp} We have the following results:
\begin{compactenum}[i)]
\item For fixed $\fX_1, \fX_2 \in \bbK^{\GHP}_\tf$ (or $\bbK^{\GP}$) we have $\nu^{\fX_1} = \nu^{\fX_2}$ 
if and only if $\fX_1 = \fX_2$. 
\item For $\kappa_1, \kappa_2 \in \cM_1(\bbK^{\GHP}_\tf)$ (or $\cM_1(\bbK^{\GP})$) we have $\nu^{\kappa_1} = \nu^{\kappa_2}$ if and only if $\kappa_1 = \kappa_2$.
\item  For probability distributions $\kappa$ and $\kappa_n$, $n\ge 1$, on  $\bbK^{\GP}$, we have 
$\kappa_n \to \kappa$ weakly if and only if $\nu^{\kappa_n} \to \nu^\kappa$ weakly.

\end{compactenum}
\end{prop}

With Proposition~\ref{prop:distghp} at hand, we can now compare the distance matrix 
distributions of two solutions of the fixed point equation~\eqref{fix:tree}. To this end, 
for a random variable $\mathfrak X$ with values in $\bbK^{\GHP}$ (or $\bbK^{\GP}$), 
set $\Ec{\nu^{\mathfrak X}} = \nu^{\Law(\mathfrak X)}$, that is, $\Ec{\nu^{\fX}}(A) = \Ec{\nu^{\fX}(A)}$ for a measurable set $A$.
Further, for $n \geq 0$ and 
fixed $\mathfrak X \in \bbK^{\GHP}$ (or $\bbK^{\GP})$, we 
write $\nu_{n}^{\mathfrak X} \in \cM_1(\mathbb M_n)$ for the distribution of the distance 
matrix induced by the first $n+1$ points $\zeta_0, \zeta_1, \ldots, \zeta_n$. The quantities
$\nu_{n}^{\kappa}$ for a probability distribution $\kappa\in \cM_1(\bbK^\GHP)$ 
(or $\cM_1(\bbK^\GP)$) and $\Ec{\nu_n^{\fX}}$ for a random variable $\fX$ shall be 
defined correspondingly.

\begin{prop} \label{prop:treeuniq}
Let $\mathfrak X = (\cX, d, \mu, \rho)$ and $ {\mathfrak U} = (\mathcal U, d', \mu', 
\rho')$  be $\bbK^{\GHP}_\tf$ (or $\bbK^{\GP}$)-valued r.v.'s satisfying \eqref{fix:tree} 
(in distribution) and 
$\Ec{d(\rho, \zeta)} = \Ec{d'(\rho', \zeta')}$ where $\zeta$ (resp.\ $\zeta'$) is 
chosen on $X$ (resp.\ $\mathcal U$) according to $\mu$ (resp.\ $\mu'$). Then, 
$\Ec{\nu^{\mathfrak X}} = \Ec{\nu^{\mathfrak U}}$, hence $\Law(\mathfrak X) = \Law( \mathfrak U)$.
\end{prop}

\begin{proof}
For $n \geq 1$, let $D_n$ (resp. $D_n'$) be a random variable on $\mathbb M_n$ with 
distribution $\Ec{\nu_n^ {\mathfrak X}}$ (resp. $\Ec{\nu_n^ {\mathfrak U}}$). 
 Both $D_1$ and $D_1'$ satisfy fixed point equation \eqref{fix:Y}. As elaborated in the discussion of \eqref{fix:Y} it follows from the 
 results in \cite{durlig} that $D_1$ and $D_1'$ are identically distributed since both random variables have mean $c$.
 Our aim is to show by induction on $n$ that, for all $n \geq 2$, 
both $D_n$ and $D_n'$ satisfy the same 
stochastic fixed point equation known to admit at most one solution.
Taking that for granted for now, it follows immediately that 
$\Ec{\nu_n^ {\mathfrak X}} = \Ec{\nu_n^{\mathfrak U}}$ for all $n \geq 1$; 
Proposition~\ref{prop:distghp} \emph{ii)} then proves the assertion.

Let $A \subseteq \mathbb M_n$ be measurable. 
With $\tau = \Law((\cR, \cS))$, $\aleph = \Law(\mathfrak X)$ and $\mathfrak X_1, \ldots, 
\mathfrak X_K$ i.i.d.\ with distribution $\aleph$, also independent of $(\cR, \cS)$, 
we deduce from \eqref{fix:tree} that
\begin{align} 
\pc{D_n \in A}  
& = \nu_n^{\aleph}(A) 
= \nu_n^{\Psi(\aleph, \tau)}(A) 
= \Ec{\nu_n^{\psi(\mathfrak X_1, \ldots, \mathfrak X_K, \cR, \cS)}(A)} \nonumber\\
& = \int_{\Sigma_K^2} d \tau(\mathbf{r},\mathbf{s}) 
\int_{(\mathbb K^{\GHP})^K} d \aleph^{\otimes K} (x_1, \ldots, x_K) 
\Ec{\nu_n^{\mathfrak X^*(x_1, \ldots, x_K,\mathbf{r}, \mathbf{s})}}(A), \label{eq:inte}
\end{align}
where $\mathfrak X^*(x_1, \ldots, x_k, \mathbf{r}, \mathbf{s})$ has distribution
$\psi(x_1, \ldots, x_k, \mathbf{r}, \mathbf{s})$. (Here, we could apply Fubini's theorem thanks 
to the product measurability of the map $\psi$ proved in the supplementary material.) 
We now keep $\br=(r_1,\dots, r_K)$, $\bs=(s_1,\dots, s_K)$ and $x_1, \ldots, x_K$ fixed.  
Let $\bar x_1, \ldots, \bar x_K$ be arbitrary representatives of $x_1, \ldots, x_K$ and write
$\bar x_i = (\bar x_i, \bar d_i, \bar \mu_i, \bar \rho_i)$. 
Let $\eta_i, i \in [K]$ be independent points on $\bar x_i$ with distribution $\bar \mu_i$. 
Let $x_* = (x_*, d_*, \mu_*, \rho_*)$ be the space constructed with the help of 
$\Gamma, \alpha, \mathbf{r},\mathbf{s}, \bar x_1, \ldots, \bar x_K, \eta_1, \ldots, \eta_K$ 
following the steps ii) - iv) on page \pageref{list}. 
To sample independent test points 
$(\zeta_\ell)_{1\le \ell \le n}$, on $x_*$ according to $\mu_*$, we consider a family of 
independent random variables $\{\theta_{i,j} : i \in [K], j=1, \ldots, n\}$ which is independent 
of the glue points $\eta_i$, $i \in [K]$, where each $\theta_{i,j}$ takes values in $\bar x_i$ 
and has distribution $\bar \mu_i$. Let also $\bJ=(J_1,\dots, J_n)$ be a vector of i.i.d.\ random 
variables with values in $[K]$, independent of the remaining quantities, 
with $\pc{J_\ell = j} = s_j$ for $j \in [K]$. 
Define $\zeta_\ell = \cp(\theta_{J_\ell,\ell})$, where $\cp$ is introduced in step \emph{ii)} 
of the construction on page~\pageref{it:construction}.

 Then $(\zeta_i)_{1\le i \le n}$ is a family of i.i.d.\ points with distribution $\mu_*$. Set $\zeta_0 := \rho_*$. By construction, the matrix $W_n = (d(\zeta_i, \zeta_j))_ {0 \leq i, j \leq n}$ has distribution 
$\Ec{\nu_n^{\mathfrak X^*(x_1, \ldots, x_k,\mathbf{r},\mathbf{s})}}$. (In particular, its distribution does not depend on the choice of the representatives $\bar x_1, \ldots, \bar x_K$.) 

We now decompose $W_n$ by looking in which of the subspaces (or, more precisely, their images under $\cp$) the random points $\zeta_1,\dots, \zeta_n$ fall. In each space, the trace of the paths induce a rescaled distance matrix with a potentially different number of points which maybe the original points, the root of the subspace or glue points. 
To this end, first recall that the root $\rho_*=\zeta_0$ 
is the (image of the) root $\bar \rho_1$ of $\bar x_1$. 
Set $j_0 := 1$ and, for $\mathbf j = (j_1, \ldots, j_n) \in [K]^n$ and $i \in [K]$ 
define $L_i^{\mathbf j}  = \{0 \leq  \ell \leq n : j_\ell = i\}$. 
Then, set $\ell_1 = \# L_1^{\mathbf j} - 1$ and, for $2 \leq i \leq K$ 
define $\ell_i=\# L_i^{\mathbf j}$. Further, for $i \in [K]$, let $\ell_i^* = \ell_i + 1$ 
if there exists some $1\le \ell\le n$ such that $j_\ell \in \Gamma_i \setminus \{i\}$ 
and $\ell_i^* = \ell_i$ otherwise.
 ($\ell_i$ is the number of test points falling in (the image of) $\bar x_i$, and $\ell_i^*$ accounts 
for the glue point $\eta_i$ which plays a role if there is some segment
$\llb \zeta_p,\zeta_q\rrb$ containing its image under $\cp$.)

 For $\mathbf j \in [K]^n$ let $\cE^{\mathbf j}$ be the event that $J_\ell = j_\ell$ for all $\ell \in [n]$. 
For integers $p \geq 1,  i \in [K],$ we write $Y^{(p)}_{x_i}$ for a generic random variable 
with distribution $\nu_p^{x_i}$.
Now note that, on $\cE^{\mathbf j}$, the distance matrix $W_n$ has the same 
distribution as a linear combination of deterministic linear operators evaluated at independent copies of $Y^{(\ell_i^*)}_{ x_i}, i \in [K]$. 
To this end, for $\ell=1,2, \dots,n$, 
let $\rr(\ell):=\#\{p\in L_{j_\ell}^{\mathbf j}: 1\le p\le \ell\}$ be the rank of $\ell$ in 
the set $L_{j_\ell}^{\mathbf j}$, and set $\rr(0) = 0$.
Then, for $i \in [K]$, we define operators
$G_i^{(\bj)} : \bbM_{\ell_i^*} \to \bbM_{n}$ in two steps. 
Let $A\in \bbM_{\ell_i^*}$ and $1\le p,q\le n$. First, set
\begin{align} \label{g1} 
G_i^{(\mathbf j)}(A)_{p,q} 
= 
\begin{cases} 
A_{\rr(p),\rr(q)} & \text{if } p, q \in L_i^{\bj}, \\
A_{0,\rr(p)} & \text{if } p  \in L_i^{\bj}, j_{q} \notin \Gamma_i,  \\
A_{0,\rr(q)} & \text{if } q  \in L_i^{\bj}, j_{p} \notin \Gamma_i.  
\end{cases}
\end{align}  
This does not define all the entries of $G_i^{(\mathbf j)}(A)$. For the remaining 
entries: if $\ell_i^* = \ell_i$, then $G_i^{(\bj)}(A)_{p,q} = 0$ for all $p, q$ which 
are not covered in one of the cases in \eqref{g1}; or else we have $\ell_i^* = \ell_i + 1$, 
and then
\begin{align} \label{g2} 
G_i^{(\bj)}(A)_{p, q} 
= 
\begin{cases} 
A_{\rr(p),\ell_i^*} & \text{if } p  \in L_i^{\bj}, j_{q} \in \Gamma_i \setminus \{i\},\\
A_{\rr(q),\ell_i^*} & \text{if } q  \in L_i^{\bj}, j_{p} \in \Gamma_i \setminus \{i\},\\
A_{0, \ell_i^*} & \text{if } j_{p}  \in \Gamma_i \setminus \{i\}, j_{q} \notin  \Gamma_i 
\text{ or }  j_{q}  \in \Gamma_i \setminus \{i\}, j_{p} \notin  \Gamma_i,  
\end{cases} 
\end{align} 
and  $G_i^{(\bj)}(A)_{p, q}  = 0$ for all $p, q$ which are covered neither in \eqref{g1} 
nor in \eqref{g2}. Note that, if $\ell_i^* = 0$, we have $G_i^{(\bj)} \equiv 0$.

The following observation is the crucial ingredient of the proof: When conditioning 
on $\cE^{\bj}$, as the points $\theta_{k, \ell}$, $\ell = 1, \ldots, n$, and $\eta_k$, 
are all independent and distributed on $\bar x_k$ according to $\bar \mu_k$, 
we may think of $\eta_k$ as an additional test point on $\bar x_k$. Thus, on $\cE^{\bj}$, 
we obtain the following distributional equality
\begin{align} \label{fix:W}
W_n \eqdist \sum_{i=1}^K r_i^\alpha \cdot G_i^{(\mathbf j)}(Y_{x_i}^{(\ell_i^*)}), 
\end{align}
where the $ Y^{(\ell_i^*)}_{x_i}$ , $i \in [K]$, are independent.
(For the sake of convenience, we agree to set $Y^{(0)}_{x_i}$ to be the matrix containing
a single entry which is $0$.)

The cases where $\ell_i^* = n$ for some $i \in [K]$ need to be considered in detail:
 indeed, they are the cases that yield $(n+1)\times (n+1)$ distance matrices from the 
constituant subspaces, and are thus crucial to the fixed point argument.
To this end, we define the following subsets of $[K]^n$: $C_{i} = \{(i, \ldots, i)\}$ and 
\begin{align*}
C^*_i = \bigcup_{k=1}^K \Big\{(j_1, \ldots, j_n) : j_\ell = i 
\text{ for all } \ell \neq k, j_k \in \Gamma_i \setminus \{i\}\Big\},
\end{align*}
as well as $C = C_1 \cup \dots \cup C_K$ and $C^* = C^*_1 \cup \dots \cup C^*_K$.
Then, for $i \in [K]$ we have $\ell_i^* = n$ if and only if $\bj \in C_i \cup C_i^*$.
In the following, we distinguish these two cases. Recall the definition of $E_i$ from 
\eqref{def:subtree}.

{\bf a)} If $\bj \in C_{i} $, the operator $G_{i}^{(\bj)}$ is the identity and, as observed 
previously, $G_k^{(\bj)} \equiv 0$ if $k \notin E_i \cup \{i\}$ since $\ell_k^* = 0$.  
For  $k \in E_{i}$, however, we have $\ell_k^* = 1$ and $G_{k}^{(\bj)} : \bbM_{1} \to \bbM_{n}$ 
is defined in \eqref{g1}  and \eqref{g2}.
Thus, in this case, \eqref{fix:W} can be written as 
\begin{align} \label{fi}
W_n \eqdist 
r_{i}^\alpha \cdot Y_{x_{i}}^{(n)} 
+ \sum_{k=1, k \neq i}^K r_{k}^\alpha \cdot G_k^{(\mathbf j)}(Y_{x_k}^{(\ell_k^*)}),
\end{align}
where $Y^{(n)}_{x_{i}}, Y_{x_k}^{(1)}$ , $k \in [K] \setminus \{i\}$ are 
independent. 
As just observed, $\ell_k^* \in \{0, 1\}$ for all $k \neq i$.

{\bf b)} For $\bj \in C^*_{i} $, letting $k^* \neq i$ denote the unique value with $L^{\mathbf j}_{k^*} \neq \emptyset$, we have 
$\ell_k^* = 1$ for all $k \in \{k^* \} \cup (E_{k^*} \setminus \{i\})$ and $\ell_k^* = 0$ for all $k \notin \{k^*\} \cup E_{k^*}$.

The operators $G_k^{(\bj)}$ for $k \neq i$ are defined in \eqref{g1}  and \eqref{g2}.
The operator $G_i^{(\bj)}$ acts on a matrix $A \in \bbM_n$ by permuting the indices as 
follows: $0 \to 0, k \to \rr(k)$ for $k \neq m$ and $m \to \ell_i^*=n$.
As the distribution of $Y_{x_i}^{(n)}$ is invariant under such permutations, 
the random variable $W_n$ satisfies \eqref{fi}.

Upon collecting our findings for the different cases and performing the integration 
in \eqref{eq:inte}, it follows that the random matrix $D_n$ satisfies 
\begin{align} \label{perpp}
D_n \eqdist U_n D_n + V_n,
\end{align}
where $U_n\in [0,1]$, $V_n\in \bbM_n$, the $(U_n, V_n), D_n$ are independent, and we have
\begin{align} \label{def:aa}
U_n = \sum_{i=1}^K 
\cR_{i}^\alpha \cdot \mathbf 1_{\bJ \in C_{i} \cup C^*_i}
\end{align}
and
\begin{align} \label{def:bb} 
V_n & = 
\sum_{i=1}^K \sum_{\bj \in C_i \cup C_i^*} \mathbf 1_{\bJ = \bj} 
\sum_{k \neq i }^K \cR_k^\alpha \cdot G_k^{(\bj)}(D_{\ell_k^*}^{(k)})  
+ \sum_{\bj \notin C \cup C^*} \mathbf 1_{\bJ = \bj} 
\sum_{k =1 }^K \cR_k^\alpha \cdot G_k^{(\bj)}(D_{\ell_k^*}^{(k)}). 
\end{align}
Here, recall that we have $\bJ = (J_1, \ldots, J_n)$, where, given $(\cR, \cS)$, the random 
variables $J_1, \ldots, J_n$ are independent and each $J_i$ is distributed as $J$ 
in \eqref{def:alpha}, and $\{D_\ell^{(k)}: 0 \leq \ell \leq n-1, k \in [K]\}$ is an 
independent family of random variables, which is independent of $(\cR, \cS, \bJ)$, 
where each $D_\ell^{(k)}$ is distributed like $D_\ell$. 

A random variable satisfying a fixed point equation of type \eqref{perpp} is called a \emph{perpetuity.} 
It follows from classical results on perpetuities, 
e.g.\ from \cite[Theorem 1.5]{Vervaat1979a}, that \eqref{perpp} has at most 
one solution (in distribution). Repeating the arguments shows that $D_n'$ satisfies a 
distributional identity of the form $D_n' = U'_n  D_n' + V'_n$ with $U'_n = U_n$ and the 
additive term $V'_n$, can be obtained from $V_n$ by replacing each $D^{(k)}_\ell$ by a copy 
of $D_\ell'$ while maintaining the independence structure. Hence, by our induction hypothesis, 
$(U_n,V_n)$ and $(U_n', V_n')$ are identically distributed which shows that $D_n$ and $D'_n$ 
are identically distributed and concludes the proof of the induction. 
\end{proof}

\subsection{Construction of a solution}\label{sec:construction_as}

In this section, we construct the family of processes $\{\Xex^\vartheta : \vartheta\in \Theta \}$
mentioned in Section~\ref{tech-expan} that plays a central role in a number of proofs later on.
The following proposition is a generalization of Theorem~6 and Theorem~17 in \cite{BrSu2013a}. 
We keep the presentation rather compact and refer to \cite{BrSu2013a} for more details on 
technical points; this applies in particular to a number of tedious but straightforward 
inductions occurring throughout the proof.  
Recall that $\| \cdot \|$ denotes the uniform norm on $\Cex$ and the 
definition of $\mathfrak m$ in \eqref{eq:def_fm}. For  $a > 0$, let
$$\mathcal M^a = \left \{ \mu  \in \mathcal M_1(\Cex): 
\int \!\!\!\int_0^1 x(t) dt \mu(dx) = a \text{~and~}
\int \|x\|^{\mathfrak m} \mu(dx) < \infty \right\}.$$
Fix $c > 0$ and let $Q_0^\vartheta = c,  \vartheta \in \Theta$. 
Recall the map $\Phi$ defined in Section~\ref{sec:rec_exc}. 
Recursively, for $n \geq 1$ and $\vartheta \in \Theta$, define
\begin{align} \label{def:qn}
Q_{n}^\vartheta 
= \Phi(Q_{n-1}^{\vartheta1}, \ldots, Q_{n-1}^{\vartheta K}, 
\cR^\vartheta, \cS^\vartheta, \Xi^\vartheta).
\end{align}

\begin{prop} \label{prop:conlimit}
For any $\vartheta \in \Theta$,  almost surely, the sequence $Q_n^\vartheta$ defined 
in \eqref{def:qn} converges uniformly to a process $\Xex^\vartheta$. 
For any $\vartheta \in \Theta$, we have, almost surely,
\begin{align} \label{fix:as}
\Xex ^{\vartheta} =  
\Phi( \Xex^{\vartheta 1}, \ldots, 
\Xex^{\vartheta K}, \mathcal R^\vartheta, \mathcal S^\vartheta, \Xi^\vartheta).
\end{align}
Furthermore, $\Law(\Xex^\vartheta)$ is the unique solution to \eqref{fix:X} 
in the set $\cM^c$ and $\Ec{\|\Xex^\vartheta\|^m} < \infty$ for all $m \geq 1$. 
\end{prop}

\begin{rem} \label{rem:al}
The proof of the proposition also shows the following: 
Let $\mu \in \mathcal M^c$, and let
$\{\cY^\vartheta: \vartheta \in \Theta\}$ be a family of independent random processes 
with $\Law(\cY^\vartheta) = \mu$  for all $\vartheta \in \Theta$ which is 
independent of $\{(\cR^\vartheta, \cS^\vartheta, \Xi^\vartheta): \vartheta \in \Theta\}$. 
Then, the sequence $Q_n^\vartheta$ initiated with $\{\cY^\vartheta: \vartheta \in \Theta \}$, 
that is, $Q_0^\vartheta = \cY^\vartheta$,
converges almost surely uniformly to $\Xex^\vartheta$. \end{rem}

\begin{proof}[Proof of Proposition~\ref{prop:conlimit}]
First of all, note that, for $n \geq 1$, \eqref{def:qn} says that, for all $t\in [0,1]$,
\begin{align} \label{def:qn_detailed}
Q_{n}^\vartheta (t) =
\sum_{i=1}^K \mathbf{1}_{ \Lambda_i^\vartheta}(t) 
\Bigg [ & (\mathcal R_i^\vartheta)^\alpha Q_{n-1}^{\vartheta i}(\varphi^\vartheta_i(t))  + \sum_{j \in E_i}  (\mathcal R_j^\vartheta)^\alpha  Q_{n-1}^{\vartheta j}(\xi^\vartheta_j)  \Bigg ]\, . 
\end{align}
Throughout the proof, $\xi$ denotes a random variable with the uniform 
distribution on $[0,1]$ which is independent of all remaining quantities.
Recalling the choice of $\alpha$ in \eqref{def:alpha} and the fact 
that $\Ec{Q_0^\vartheta(\xi)} = c$, it follows by induction that $\Ec{Q_n^\vartheta(\xi)} = c$ for all $n \geq 0$. From \eqref{def:qn} (or \eqref{def:qn_detailed}), it should be clear that the sequences 
$(Q_n^\vartheta)_{n \geq 0}, \vartheta \in \Theta$, are identically distributed. 
(A formal proof could again be given by induction.)
Next, set $\Delta Q_n^\vartheta := Q_{n+1}^\vartheta - Q_n^\vartheta$ for $n \geq 0$. 
From \eqref{def:qn}, we have for all $t\in [0,1]$, 
\begin{align} \label{diffq}
\Delta Q_n^\vartheta(t)  =
\sum_{i=1}^K \mathbf{1}_{ \Lambda_i^\vartheta}(t) 
\Bigg [ 
(\cR_i^\vartheta)^\alpha \Delta Q_{n-1}^{\vartheta i}(\varphi^\vartheta_i(t)) 
+ \sum_{j \in E_i}  (\cR_j^\vartheta)^\alpha  \Delta Q_{n-1}^{\vartheta j}(\xi^\vartheta_j) 
\Bigg ]\, . 
\end{align}
In particular, one finds
$\Delta Q_n^\vartheta = h(Q_{n-1}^{\vartheta 1}, \ldots, 
Q_{n-1}^{\vartheta K}, \cR^\vartheta, \cS^\vartheta, \Xi^\vartheta)$ for a suitable 
deterministic continuous
function $h$  (see \eqref{diffq}). 
By induction on $n$, it follows that, for any fixed $n \geq 1$, the random 
variables $\Delta Q_n^\vartheta$, $\vartheta \in \Theta$, are identically distributed.

For $i \in [K]$, we now define $\beta_i^\vartheta = \cR^\vartheta_i$ 
if $\xi \in \bigcup_{k \in \Gamma_i} \Lambda_k^\vartheta$ and $\beta_i^\vartheta = 0$ otherwise. 
The random variable $\beta_i^\vartheta$ is distributed like $\beta_i$ defined in \eqref{fix:Y}.
From \eqref{diffq}, using the independence of $(\cR^\vartheta, \cS^\vartheta)$ and
$\Xi^\vartheta$ and the fact that, conditional on $\xi \in \Lambda_i^\vartheta$, 
the relative position of $\xi$ in this interval is uniform and independent 
of $(\cR^\vartheta, \cS^\vartheta)$, it follows from the last display that $\Ec{\Delta Q_n^\vartheta(\xi)^2}$ is equal to 
\begin{align*}
\sum_{i=1}^K & \Ec{\beta_i^{2 \alpha}} \Ec{\Delta Q_{n-1}^\vartheta(\xi)^2} \\
& + \sum_{i = 1}^K \sum_{j_1 \neq j_2 \in E_i \cup \{ i \} } 
\E{\mathbf{1}_{\Lambda_i^\vartheta}(\xi)  (\cR_{j_1}^\vartheta)^\alpha 
(\cR_{j_2}^\vartheta)^\alpha  \Delta Q_{n-1}^{\vartheta j_1}(\xi^\vartheta_{j_1}) 
\Delta Q_{n-1}^{\vartheta j_2}(\xi^\vartheta_{j_2}) }.
\end{align*}
Here the $\Lambda^\vartheta_i$ have distjoint interior, and all the squared terms are collected in the first sum (recall the definition of $\beta_i$ from \eqref{fix:Y}).
Conditional on $\xi \in \Lambda_i^\vartheta$ and on $\cR_{j_1}^\vartheta$ 
and $\cR_{j_2}^\vartheta$, 
the random variables $\Delta Q_{n-1}^{\vartheta j_1}(\xi^\vartheta_{j_1})$ 
and $\Delta Q_{n-1}^{\vartheta j_2}(\xi^\vartheta_{j_2})$ are zero-mean independent random 
variables. Thus, the second term in the last display vanishes. Hence:
\begin{align} \label{boun2}
\Ec{\Delta Q_n^\vartheta(\xi)^2} 
= \sum_{i=1}^K \Ec{\beta_i^{2 \alpha}} \Ec{\Delta Q_{n-1}^\vartheta(\xi)^2} 
= \Ec{\Delta Q_{1}^\vartheta(\xi)^2}  \left(\sum_{i=1}^K \Ec{\beta_i^{2 \alpha}}\right)^{n-1}\,.
\end{align}
 Recalling that $q_2:=\sum_{i=1}^K \Ec{\beta_i^{2 \alpha}} \in (0,1)$, this implies that 
$\Ec{\Delta Q_n^\vartheta(\xi)^2}\le C_2 q_2^n$ for $C_2=1/q_2$.

Next, we aim at showing that, for all $m \geq 1$, we have $\Ec{| \Delta Q_{n}^\vartheta(\xi)| ^m} 
\leq C_m q_m^n$ for some constants $C_m > 0$ and $q_m\in (0,1)$. The previous argument verifies 
this claim for $m = 2$ (and $m=1$ by the Cauchy--Schwarz inequality). Let $m \geq 3$ and 
assume it is true for all $1 \leq \ell \leq m-1$ and let $C_* = \max(C_1, \ldots, C_{m-1})$, 
$q_* = \max(q_1, \ldots, q_{m-1})<1$. Then, again from \eqref{diffq}, we deduce that
\begin{align} 
\Ec{| \Delta Q_n^\vartheta(\xi)|^m}  
&  \leq \sum_{i=1}^K \Ec{\beta_i^{m \alpha}} \Ec{|\Delta Q_{n-1}^\vartheta(\xi)|^m}  \nonumber  \\
& + \sum_{i = 1}^K \sum_{j_1, \ldots, j_{m}} \E{\mathbf{1}_{\Lambda_i^\vartheta}(\xi)  
\prod_{k = 1}^{m} |    \Delta Q_{n-1}^{\vartheta j_k}(\xi^\vartheta_{j_k})| }, \label{summ} 
\end{align}
where the inner sum of the second term ranges over all tuples $(j_1, \ldots, j_{m}) 
\in (E_i \cup \{i \})^m$ for which $\#\{j_1, \ldots, j_m\} \geq 2$. 
(Note that we have dropped some factors $(\cR^\vartheta_j)^\alpha$ in the product in the 
right-hand side.) Now, similarly to the argument above, on the event that
$\xi \in \Lambda_i^\vartheta$, the random variables in the product on the right-hand side 
of the last display are independent for different values of $j_k$.
Since no $j_k$ appears more than $m-1$ times in the product, we can use the induction 
hypothesis to obtain the loose, but sufficient, bound
\begin{align}
\Ec{| \Delta Q_n^\vartheta(\xi)|^m}  
&  \leq \sum_{i=1}^K \Ec{\beta_i^{m \alpha}} \Ec{|\Delta Q_{n-1}^\vartheta(\xi)|^m} \nonumber \\
& + \sum_{i = 1}^K \sum_{j_1, \ldots, j_{m}} (C_* q_*^{n-1})^{\#\{j_1,j_2,\dots, j_m\}} ,  \label{defjj}
\end{align}
with $j_1, \ldots, j_{m}$ as in the sum in \eqref{summ}. From here, since $m\ge 3$ we have
$\sum_{i=1}^K \Ec{\beta_i^{m \alpha}} < 1$ and a simple induction on $n$ shows that 
$\Ec{| \Delta Q_n^\vartheta(\xi)|^m}$ decays exponentially in $n$, as desired. 

The exponential decay for all moments at a uniform point $\xi$ 
can be bootstrapped to yield exponential decay for sufficiently high (hence all) 
moments of the supremum. 
For $m \geq 1$ it follows from \eqref{diffq} that $\Ec{\| \Delta Q_n^\vartheta\|^{m}}$ equals
\begin{align*}
& \E{\max_{1 \leq i \leq K} \left \{\left\| (\cR_i^\vartheta)^\alpha \Delta Q_{n-1}^{\vartheta i}(\varphi^\vartheta_i(t)) + \sum_{j \in E_i}  (\cR_j^\vartheta)^\alpha  \Delta Q_{n-1}^{\vartheta j}(\xi^\vartheta_j)\right \|^m\right \}} \\ 
& \leq \E{\max_{1 \leq i \leq K} \left \{ (\cR_i^\vartheta)^{ m\alpha} \|\Delta Q_{n-1}^{\vartheta i}\|^m + 
\sum_{j_1, \ldots, j_m} \|\Delta Q_{n-1}^{\vartheta i}\|^{\ell_i} \prod_k |\Delta Q_{n-1}^{\vartheta j_k}
(\xi^\vartheta_{j_k})| \right \}}
\end{align*}
with $j_1, \ldots, j_m$ as in the sum in \eqref{summ},
$\ell_i = \#\{1 \leq k \leq m : j_k = i\}$ and the product over $k$ only ranges over those 
values $1 \leq k \leq m$ with $j_k \neq i$. (Observe that $\ell_i < m$ for all $i$.) 
Bounding the maximum by the sum, abbreviating $C_{**} = \max(C_1, \ldots, C_{m})$ 
and $q_{**} = \max(q_1, \ldots, q_{m})$ and using the stochastic independence of
$\Delta Q_{n-1}^{\vartheta 1}, \ldots  \Delta Q_{n-1}^{\vartheta K}, \xi^{\vartheta}_{1}, \ldots, \xi^{\vartheta}_{K}$,
gives
\begin{align*}
& \Ec{\| \Delta Q_n^\vartheta\|^{m}}  \\ & \leq \sum_{i=1}^K \Ec{ \mathcal R_i^{{m} \alpha}} \Ec{\|\Delta Q_{n-1}^\vartheta\|^{m}} 
+ \sum_{i = 1}^K \sum_{j_1, \ldots, j_{{ m}}} \Ec{\| \Delta Q_{n-1}^\vartheta\|^{\ell_i}} (C_{**} q_{**}^{n-1})^{p_i^*} \\
& \leq   \sum_{i=1}^K \Ec{ \mathcal R_i^{{m} \alpha}} \Ec{\|\Delta Q_{n-1}^\vartheta\|^{m}}  + \sum_{i = 1}^K \sum_{j_1, \ldots, j_{{ m}}} \Ec{\|  \Delta Q_{n-1}^\vartheta \|^m}^{\ell_i/m} (C_{**} q_{**}^{n-1})^{p_i^*},
\end{align*}
where we used $p_i^* = \# \{j_k : 1 \leq  k \leq m, j_k  \neq i \}$.
Recall $\fm$ from \eqref{eq:def_fm}; for $m \geq \fm$ we 
have $\sum_{i=1}^K\Ec{ \mathcal R_i^{{ m} \alpha}}< 1$ and a simple induction on $n$ 
shows that $\Ec{\| \Delta Q_n^\vartheta\|^{m}}$ decays exponentially in $n$. 
From there, standard arguments (see, e.g.\ the proof Theorem 6 in \cite{BrSu2013a}) imply that, almost surely, 
$Q^\vartheta_{n}$ converges uniformly and, writing $\Xex^\vartheta$ for its limit, the
identity in \eqref{fix:as} holds.
Furthermore, the random variable $\|\Xex^\vartheta\|$ has  finite polynomial moments of all orders. 

It remains to show that the constructed process is the unique solution to \eqref{fix:X} in distribution in $\mathcal M^c$. As this part does not require significantly new ideas, we remain brief. Let $\mu \in \mathcal M^c$ and consider a set of independent random variables $\{\mathcal Y_0^\vartheta, \vartheta\in \Theta\}$  
with $\Law(\mathcal Y_0^\vartheta) = \mu$ that is independent of $\{(\mathcal R^\vartheta, \mathcal S^\vartheta, \Xi^\vartheta): \vartheta \in \Theta\}$. Then, analogously to \eqref{def:qn}, 
for $n \geq 1$ and $\vartheta \in \Theta$, recursively define
\begin{align*}
\mathcal Y_{n}^\vartheta 
= \Phi(\mathcal Y_{n-1}^{\vartheta1}, \ldots, \mathcal Y_{n-1}^{\vartheta K}, \mathcal R^\vartheta, \mathcal S^\vartheta, \Xi^\vartheta).
\end{align*}
By repeating the steps of the proof of uniform convergence of $(Q_n^\vartheta)_{n \geq 0}$, 
we find that the sequence $(\mathcal Y_n^\vartheta)_{n\ge 0}$ converges almost surely uniformly to a 
solution $\mathcal Y^\vartheta$ of \eqref{fix:X}. Furthermore, by following the same inductive 
arguments as before, one can show that $\Ec{(Q_n^\vartheta(\xi) - \mathcal Y_n^\vartheta(\xi))^2} \to 0$ 
exponentially fast. (To be precise, we have the same bound as in \eqref{boun2} for this term 
with $\Ec{\Delta Q^\vartheta_1(\xi)^2}$ replaced by $\Ec{(\mathcal Y_0^\vartheta(\xi) - c)^2}.$)
The fact that $\Ec{|Q_n^\vartheta(\xi) - \mathcal Y_n^\vartheta(\xi)|^{m}} \to 0$ exponentially fast for all $1 \leq m \leq \mathfrak m$ follows as in \eqref{summ} and \eqref{defjj}.
Based on this, the same steps as before yield that $\Ec{\|Q_n^\vartheta - \mathcal Y_n^\vartheta\|^{m}} \to 0$ for all $1 \leq m \leq \fm$.
Hence, $\mathcal Y^\vartheta = \Xex^\vartheta$ almost surely.
This concludes the proof.
\end{proof}

\begin{proof}[Proof of Theorem~\ref{thm:treeuniq} i) and ii)]
The statement in \emph{ii)} is an immediate consequence of Proposition~\ref{prop:treeuniq} 
and Proposition~\ref{prop:distghp} \emph{ii)} since $\Ec{\nu^\fX}$ remains invariant upon 
replacing $\cX$ by the support of the measure. The uniqueness claim for the process in 
Theorem~\ref{thm:treeuniq} \emph{i)} can be deduced as follows: 
Let $\Xex$ be the process constructed in Proposition~\ref{prop:conlimit}, and
assume that $\cY$ is a continuous excursion satisfying \eqref{fix:X} with
$\Ec{\mathcal Y(\xi)} = \Ec{\Xex(\xi)}$. Then, by Proposition~\ref{prop:treeuniq}, 
$\Ec{\nu_n^{\fT_\Xex}} = \Ec{\nu_n^{\fT_{\cY}}}$ for all $n \geq 1$. 
Let $f_n :  \cM_1(\bbM_n) \to \cM_1([0, \infty))$ be  the map that, 
to $\nu \in \cM_1(\bbM_n)$, associates the law of $\sup \{A_{0,i}: 0 \leq i \leq n\}$,
where $\Law(A) = \nu$.
Then, as $n \to \infty$, we have the following weak convergences: 
$$
f_n \big( \Ec{\nu_n^{\fT_\Xex}} \big) \to \Law(\|\Xex\|), 
\quad \text{and}\quad 
f_n \big( \Ec{\nu_n^{\fT_{\cY}}} \big) \to \Law(\|{\cY}\|).$$ 
Hence, $\Law(\|{\cY}\|) = \Law(\|\Xex\|)$, and in particular, by Proposition~\ref{prop:conlimit}, 
$\|{\mathcal Y}\|$ must have finite moments of all orders. The uniqueness statement under the 
finite moment condition in Proposition~\ref{prop:conlimit} then implies 
that $\Law({\cY}) = \Law(\Xex)$.
\end{proof}

\subsection{Attractiveness of the fixed points of \eqref{fix:tree}}
\label{sec:attractiveness_fix}

In this section, we prove Theorem~\ref{thm:unique_tree-GHP} and 
construct the counter-examples of Proposition~\ref{prop:counter_ex}. 
We start with the following lemma that provides a height function representation of 
random elements in $\bbT^\GHP_\tf$. 
For technical reasons, we work in the space $\Do$ of c{\`a}gl{\`a}d functions $f : [0,1] \to \R$ satisfying
$f(t) = \lim_{s \uparrow t} f(s)$ for all $t \in (0,1]$ and for which the right-hand limits $f(t+) = \lim_{s \downarrow t} f(s)$ exist for all $t \in [0,1)$. The set $\Do$ is equipped with the Skorokhod $J_1$-topology (\cite[Chapter 3]{bil68}). 
Let $\Dex \subset \Do$  be the set of non-negative functions $f \in \Do$ with $f(0) = f(0+) = 0$ and $f(t) - f(t+) \geq 0$ for all $t \in [0,1]$. 
Analogously to continuous excursions, every $f \in \Dex$ encodes a compact rooted measured real tree $(\sT_f, d_f, \mu_f, \rho_f)$ satisfying \textbf{C1} via the construction outlined in 
Section~\ref{sec:rtc} (\cite[Lemma 2.1]{duq1}). The proof of the following lemma is found in the appendix.

\begin{lem}\label{lem:extra}
Let $\nu$ be a probability distribution on $\bbT^{\GHP}_\tf$. 
Then, there exists a probability distribution $\eta$ on $\mathbb D_{\rm{ex}}$ such that $\Law(\mathfrak T_\Xex) = \nu$ for a random variable 
$\Xex$ with law $\eta$.
\end{lem}


\begin{proof}[Proof of Theorem~\ref{thm:unique_tree-GHP}]
{\em i)} For $m, n \geq 1$,  write  $\tilde D^{(m)}_n$ for a generic random variable with 
distribution $\nu_n^{\phi^m(\nu)}$. Recall from the proof of 
Proposition~\ref{prop:treeuniq} that, if $\fT$ satisfies \eqref{fix:X}, then 
for $n \geq 1$, a r.v.\ $D_n$ with distribution $\Ec{\nu_n^{\fT}}$ satisfies the fixed point equation \eqref{perpp}, that is,
$D_n \stackrel{d}{=} U_n D_n + V_n$,
where $U_n$ is a real-valued r.v.\ with $U_n\in(0,1)$ a.s.\ that is given in \eqref{def:aa},
$V_n$ is a random matrix with non-negative entries given in \eqref{def:bb}, 
and $(U_n, V_n), D_n$ are independent.  
Since $\phi^{m+1} = \phi \circ \phi^m$, the arguments from the proof of 
Proposition~\ref{prop:treeuniq} show that
$$\tilde D^{(m+1)}_n \stackrel{d}{=} U_n \tilde D^{(m)}_n + V_n^{(m)}\,,$$ 
where $(U_n, V_n^{(m)})$ and $\tilde D^{(m)}_n$ are independent and $V_n^{(m)}$ is given 
as $V_n$ in \eqref{def:bb} upon replacing each copy of $D_\ell^{(k)}$ 
by a copy of $\tilde D^{(m)}_\ell$ and maintaining the independence between $\cR, \cS, \bJ$ and the copies
of  $\tilde D^{(m)}_1, \ldots, \tilde D^{(m)}_{n-1}$.
The remainder of the proof consists in showing that, in distribution 
(and in mean)
for all $n \geq 1$, we have $\tilde D^{(m)}_n \to  D_n$ as $m \to \infty$. 
By Proposition \ref{prop:distghp} (iii), this implies the assertion. 
The proof relies on a contraction argument.
In the following, as matrices in $\bbM_{n}$ are symmetric with zeros on the diagonal, 
we use the natural identification of $\bbM_n$ with $\R^{{n+1 \choose 2}}$. 
For $p \in \N, p \geq 1$, let $\cM_1^1(\R^p)$ be the set of probability measures on $\R^p$ 
whose max-norm $\|\cdot \|_\infty$ is integrable. 
Recall the \emph{Wasserstein} distance on $\cM_1^1(\R^p)$ defined by, 
for $\mu_1, \mu_2 \in \mathcal \cM_1^1(\R^p)$, 
$$\ell_1(\mu_1,\mu_2) 
= \inf \big\{\Ec{\|X - Y\|_{\infty}} : \Law(X) = \mu_1, \Law(Y) = \mu_2 \big\}.$$
For random variables $X, Y$ in $\R^p$, we abbreviate $\ell_1(X,Y) := \ell_1(\Law(X), \Law(Y))$.
We proceed by induction on $n\ge 1$, and assume that, for all $1 \leq i < n$,
$\ell_1(\tilde D^{(m)}_i, D_i) \to 0$ as $m \to \infty$. 
First, observe that $\|G^{\mathbf j}_i(A)\|_\infty \leq \|A\|_\infty$ for all 
matrices $A$ and linear operators $G^{\bj}_i$ in and around \eqref{g1}--\eqref{g2}.  
Since the random variables $\cR_i, i \in [K]$, lie in $[0,1]$, 
conditioning on $(\cR, \cS, \bJ)$ shows that
$$\ell_1(\tilde D^{(m+1)}_n,  D_n) 
\leq \Ec{U_n} \cdot \ell_1(\tilde D^{(m)}_n,  D_n) 
+ \sup_{1 \leq \ell \leq n-1} \ell_1(\tilde D^{(m)}_\ell,  D_\ell).$$ 
From here, since $\Ec{U_n} < 1$ and $\ell_1(\tilde D^{(m)}_\ell,  D_\ell) \to 0$ 
for all $1 \leq \ell \leq n-1$ by the induction hypothesis, it follows by induction on $m$ 
that the sequence $(\ell_1(\tilde D^{(m)}_n,  D_n))_{m \geq 1}$, is bounded. 
Taking the limit superior  in the last display then shows that
$\ell_1(\tilde D^{(m)}_n, D_n) \to 0 $ as $m \to \infty$ since $\Ec{A_n} < 1$. 
 This concludes the proof of the induction step, and it only remains to
establish the base case $n = 1$. 
Note that, under our identification, $\tilde D^{(m)}_1$ and $ D_1$ are  real-valued
and non-negative random variables. 
Distributional convergence $\tilde D^{(m)}_1 \to D_1$ as $m \to \infty$ 
follows immediately from Theorem~2 b) in \cite{durlig}. 
Furthermore, by construction $\phi^m$ preserves the expected distance 
between the root and an independent random point and therefore
$\Ec{\tilde D^{(m)}_1} = \Ec{D_1}$ for all $m \geq 1$. 
But convergence in $\ell_1$ for non-negative random variables is equivalent 
to distributional convergence together with convergence of the mean, see, 
e.g.\ \cite[Lemma 8.3]{BickelFreedman}. This concludes the proof of {\em i)}.

{\em ii)} If $\nu\in \cM_1(\bbT_\tf^\GHP)$ the claim 
follows easily from the proof of Proposition~\ref{prop:conlimit}. Indeed, by 
Lemma~\ref{lem:extra}, there exists a probability distribution $\nu^*$ on $\Dex$
such that the (isometry class of the) tree encoded by a random variable with law $\nu^*$ has 
distribution $\nu$.
Then, for any $n\ge 0$, $\phi^n(\nu)$ is the distribution of the real tree encoded 
by $Q_n^\emptyset$ from the proof of Proposition~\ref{prop:conlimit} when $Q_0^\vartheta$,
$\vartheta\in \Theta_n$, are i.i.d.\ with distribution $\nu^*$; see Remark \ref{rem:al}. 
(All arguments in the proof of Proposition~\ref{prop:conlimit} apply analogously to
c{\`a}gl{\`a}d functions.)
As we have seen there, $Q_n^\emptyset$ converges almost surely to a solution 
of \eqref{fix:tree}. (It is here where we need the moment assumption on $\|\mathfrak T\|$.)
This proves the claim.
\end{proof}

\begin{proof}[Proof of Proposition~\ref{prop:counter_ex}]\label{pf:counter_ex}
{\em i)} The example we provide generalizes the one by \citet{albgol2015} in the special case of 
Example {\bf 1)}. Let $\tau = \Law((\cS, \cS))$ where $\Law(\cS) = \eta$. 
(In other words, we choose $\Law(\cR) = \eta$ and the coupling $\cR = \cS$.) 
Let $\nu_c$ be the unique law solving \eqref{fix:tree} in $\bbT^\GHP_\tf$ with $\Ec{d_c(\rho_c,\zeta_c)}=c$, where 
$\zeta_c$ is sampled according to  $\mu_c$ on $\cT_c$ and $(\cT_c, d_c, \mu_c, \rho_c)$ has distribution $\nu_c$. Such a solution exists by Theorem~\ref{thm:treeuniq}. 
The idea is to construct another random compact rooted measured 
real tree in $\bbT^\GHP$ by appending massless hair to a tree sampled from $\nu_c$. 
 
Choose an integer $\kappa\ge 1$. Let $c_0^+$ be the set of all non-negative sequences converging to zero. 
Let $\fX = (\cX, d, \mu, \rho)$ be a fixed compact measured 
metric space.  For a sequence of points $u = (u_n)_{n \geq 0}$ in $\cX$ and $s \in c_0^+$, let
$\chi_1(u,s)$ be the isometry class of the space obtained upon attaching $\kappa$ disjoint segments of length $s_i$ at the point $u_i$, each by one extremity, for all $i \geq 1$. As
$s \in c_0^+$, the resulting space is compact. Hence, $\chi_1: \cX^\N \times c_0^+ \to \bbK^{\GHP}$ is well-defined. 
The map is continuous as proved in  Lemma~\ref{lem:psi1_cont} in the appendix.

Let $\cP$ be a Poisson point process with intensity measure $\mu \otimes s^{-1-1/\alpha} ds$ 
on $\cX \times [0,\infty)$. $\cP$ can be considered a $(\cX^\N \times c_0^+)$-valued random variable. The isometry class $f(\mathcal P)$ is a random variable whose distribution does not depend on the choice of the representative of the isometry class of $\fX$. 
Hence, this operation defines a map $\psi_1 : \bbK^{\GHP} \to \cM_1(\bbK^{\GHP})$, 
 which is continuous (Lemma~\ref{lem:psi1_cont} in the appendix) and a version at the level of 
measures $\Psi_1 : \cM_1(\bbK^{\GHP}) \to \cM_1(\bbK^{\GHP})$ that is defined by
$$\Psi_1 (\upsilon)(A) = \Ec{ \psi_1(Y)(A)}, \qquad \text{with} \qquad \Law(Y) = \upsilon. $$
The important observation is that, because of the choice of intensity measure with exponent $-1-1/\alpha$ in the length, the effects of a multiplication of the mass by $C$ and of a multiplication lengths by $C^\alpha$ are equivalent, and thus $\Psi$ and $\Psi_1$ commute. 
In other words, for $\upsilon \in \cM_1(\mathbb K^\GHP)$, we have 
$$\Psi(\Psi_1 (\upsilon), \Law((\cS,\cS))) 
=   \Psi_1 (\Psi(\upsilon, \Law((\cS,\cS)))).$$
It follows immediately, that for any fixed point $\nu$ of \eqref{fix:tree}, the measure 
$\Psi_1 (\nu)$ also solves \eqref{fix:tree}. 
In particular, $\Psi_1(\nu_c)$ is such a fixed point and it charges only 
$\bbT^\GHP \setminus \mathbb T^\GHP_\tf$. 

Let $\fT = (\sT, d, \mu, \rho)$ be a random variable with distribution $ \Psi_1 (\nu_c)$. 
The height $\|\mathfrak T\|$ is at least as large as the length of the longest attached segment. Hence,  
for $h>0$, we have 
\begin{align}
\pc{\|\mathfrak T\|\ge h} \ge \textstyle \pc{\Po(\int_{h}^\infty s^{-1 - 1/\alpha}ds) \geq 1}  = 1 - \exp(-\alpha h^{-1/\alpha}), \label{comm1}
\end{align} 
where $\Po(\lambda)$ denotes a Poisson random variable with parameter $\lambda$. 
Since $\fm\ge 1/\alpha$  (see above \eqref{eq:def_fm}) it follows readily that 
$\Ec{\|\mathfrak T\|^{\fm}}=\infty$. 

Furthermore, for different values of $\kappa$ the corresponding laws are mutually 
singular. Finally, note that there is nothing specific to massless \emph{segments} in 
the argument. Indeed, we can replace the segments by loops, or equivalently 
identify the extremities of all the segments together when $\kappa\ge 2$. 
This proves that there exist fixed points which are real trees with probability zero.

{\em ii)}  The measure $\Psi_1(\nu_c)$ we have just constructed charges only 
spaces in $\bbT^\GHP\setminus \bbT^\GHP_\tf$ since the hair has no mass; we now provide a modified 
example where (a) the metric part of the spaces is the unique (up to scaling) 
fixed point in $\bbT_\tf^\GHP$ 
on which we graft hair, just as in {\em i)} before, but (b) we modify the mass measure so that it 
also charges the hair. We then prove that the diameter of spaces obtained by iteration does not converge in distribution to the diameter of $\fT_\Xex$. 

 Let $\fX^\circ = (\cX^\circ, d^\circ, \mu^\circ, \rho^\circ)$ be any \emph{fixed} compact measured 
metric space where $\mu^\circ$ has no atoms; to fix ideas, one may take $\cX^\circ = [0,1]$ equipped with the
Euclidean metric, the Lebesgue measure, and $\rho^\circ = 0$. From $\fX^\circ$, we construct the ``hairy''
compact measured metric space $(\cX, d, \mu, \rho)$ using the decoration by a Poisson point process as above. Here, and subsequently, we focus on the case $\kappa = 1$.
Now, we construct a new measure $\mu_\cP$ on $\cX$
as follows. We first set $\mu_\cP(\cX^\circ)=0$. Then, for $(u,s)\in \cP$, we associate a total 
$\mu_\cP$ mass $\min\{s, 1\}^{1/\alpha+1}$ to the segment of length $s$ attached to $u$; we distribute this mass along this segment with density $\alpha \cdot (1+\alpha)^{-1} \cdot r^{1/\alpha}\I{0\le r\le s}dr$ 
if $s\le 1$, and with density $e^{-r}(1-e^{-r})^{-1}\I{0\le r\le s}dr$ if $s>1$.
Then, 
\begin{align*} \Ec{\mu_\cP(\cX)}
=\int_0^\infty \min\{s,1\}^{1/\alpha+1} s^{-(1+1/\alpha)}ds
& =\int_0^1 ds + \int_1^{\infty} s^{-(1+1/\alpha)} ds \\ & = 1 + \alpha < \infty\,,
\end{align*}
and it follows that $\mu_\cP(\cX)<\infty$ almost surely. We let $\mu^*$ be the unique 
probability measure on $\cX$ that is proportional to $\mu+\mu_\cP$.
For a random point $\zeta^*$ sampled according to $\mu^*$, we have
\begin{align*}
\Ec{d(\rho,\zeta^*)} &
\le \E{\int d(\rho,u) (\mu+\mu_\cP)(du)} \\ & \le \Ec{d(\rho,\zeta)} + \E{\int d(\rho,u) \mu_\cP(du)}\,,
\end{align*}
and 
\begin{align*}
\E{\int d(\rho,u) \mu_\cP(du)} 
\le \|\fX^\circ\|+ 1 + \sup_{s\ge 1} \int_0^s e^{-r}(1-e^{-r})^{-1}dr
= \|\fX^\circ\| + 2.
\end{align*}
It follows that $\lambda^*:=\Ec{d(\rho,\zeta^*)}<\infty$, so we may 
rescale the metric and define $d^*(\cdot,\cdot)=d(\cdot,\cdot)\times c/\lambda^*$.  
Finally, let 
$\mathfrak X^*=(\mathcal X,d^*,\mu^*,\rho)$ which satisfies $\Ec{d^*(\rho,\zeta^*)}=c$ and $\Ec{\|\fX^*\|^{1/\alpha}}=\infty$. 
As for the function $\psi_1$ in part \emph{i)}, technical proofs which we omit show that the isometry class of $\fX^*$ is a random variable whose  
distribution only depends on the \GHP-equivalence class of 
$\mathfrak X^\circ$. We denote this distribution  by $\nu^*=\nu^*_c\in \cM_1(\bbK^\GHP_\tf)$.
Note that it is not difficult to construct explicitly a random variable $g \in \Cex$ in terms of the points in $\mathcal P$ such that $\fT_g = \fX^*$ (with respect to isometry classes).

Fix $n \geq 1$ and let $\fX^*_n$ be a random variable with distribution  $\phi^n_\GHP(\nu^*)$. We now study the diameter of $\fX^*_n$. 
To this end, let $\tilde \nu$ be the distribution of the isometry class of $\tilde \fX =(\fX, d, \mu^*, \rho)$ and $\tilde \fX_n$ be a random variable with distribution $\phi^n_\GHP(\tilde \nu)$.
As the rescaling for distances is deterministic, in distribution, we obtain $\fX_n^*$ from $\tilde \fX_n$ by multiplying distances in $\tilde \fX_n$ by $c/\lambda^*$.  
We can think of the random variable $\tilde \fX_n$ as constructed from a family of independent copies $\{\tilde \fX^\vartheta : \vartheta \in \Theta_n\}$ of $\tilde \fX$ relying on the family of independent rescaling factors $\{(\mathcal R^\vartheta, \mathcal S^\vartheta) : \vartheta \in \Theta_n\}$, where glue points are chosen with respect to the mass measures. (Formally, we can use the space $\fT_{Q_n^\emptyset}$ with $Q_n^\emptyset$ from the proof of Proposition~\ref{prop:conlimit} when $Q_0^\vartheta, \vartheta \in \Theta_n$, are i.i.d.\ copies of $g$ mentioned above.) Recall that the 
distribution of $\cP$ has been tailored precisely so that it is not affected by the rescaling: 
 the distribution of lengths of the segments on the union of the rescaled copies of the 
$\tilde \fX^\vartheta$  is equal to the distribution of lengths of the segments 
on $\tilde \fX$ \footnote{This is true, regardless of the fact that these segments may not be hair any 
longer in $\tilde \fX$, since the gluing occurs after the rescalings.}. In particular, it follows that the diameter of $\tilde \fX_n$ is bounded from below in a stochastic sense by the 
largest segment appearing in the Poisson process. From these considerations and \eqref{comm1}, it follows that, for $x > 0$, 
$$\pc{\text{diam}(\fX_n^*) \geq \lambda^* x/c} = \pc{\text{diam}(\tilde \fX_n) \geq x} \geq 1 - \exp(-\alpha x^{-1/\alpha}).$$
As $\| \fT_\Xex\|$ has finite moments of all orders and the last display is valid for all $n$, it follows that, for all $x$ 
large enough, the left hand side of the last display does not converge to $\pc{\text{diam}(\fT_\Xex^*) \geq \lambda^* x / c}$. This concludes the proof.
\end{proof}

\section{Proofs of the geometric properties}\label{sec:proofs_frac}

All arguments in this section rely on a generalization of the decomposition of the 
tree $\sT_\Xex$ into $K$ subtrees corresponding to $\sT_{\Xex^1}, \ldots, \sT_{\Xex^K}$ when
the fixed point equation \eqref{fix:X} is developed for several levels. 
Throughout Section~\ref{sec:proofs_frac}, let $\fT^\vartheta = (\sT^\vartheta, d^\vartheta, 
\mu^\vartheta, \rho^\vartheta) := \fT_{\Xex^\vartheta}$  with $\Xex^\vartheta$ as in 
Proposition~\ref{prop:conlimit} where, for the sake of convenience, 
we assume $\int_0^1 \Ec{\Xex^\vartheta(t)} dt = 1$.

The structural tree $\Gamma$ describes the way the trees $\sT^1, \dots, \sT^K$ (or rather, 
their images under the canonical surjection $\cp$), are arranged in $\sT = \sT_\Xex$. 
Similarly, we introduce a structural tree for the decomposition at level $n\ge 1$ as follows: 
Let $\Gamma^{n}$ be the plane tree on $K^n$ nodes labelled with elements of $\Theta_n$ that 
describes the adjacencies between the  subtrees at level $n$ of the decomposition, that is,
$\mathfrak T^\vartheta, \vartheta \in \Theta_n$, when carried out up to this level. 
We think of $\Gamma^n$ as rooted at $1\dots 1$. Observe that $\Gamma^n$ is a random object, 
since the adjacency relations depend on the random points used to glue the trees. The tree 
$\Gamma^n$ is measurable with respect to
$\{(\cS^\vartheta, \Xi^\vartheta) : 0 \leq |\vartheta| < n\}$. Analogously to 
the construction on page \pageref{list}, we shall consider $\sT = \sT_\Xex$ as 
the disjoint union $\sqcup_{\vartheta \in \Theta_n} \sT^\vartheta$ upon identifying each root 
of a tree $\sT^\vartheta, \vartheta \neq 1\ldots 1$ with the glue point on the associated 
parent tree. We write $\cp_n:  \sqcup_{\vartheta \in \Theta_n} \sT^\vartheta \to \sT$ for 
the corresponding canonical surjection. 
See Figure~\ref{fig:phi-squared} for an 
illustration.

\begin{figure}[htb]  
\centering
\includegraphics[scale=.40]{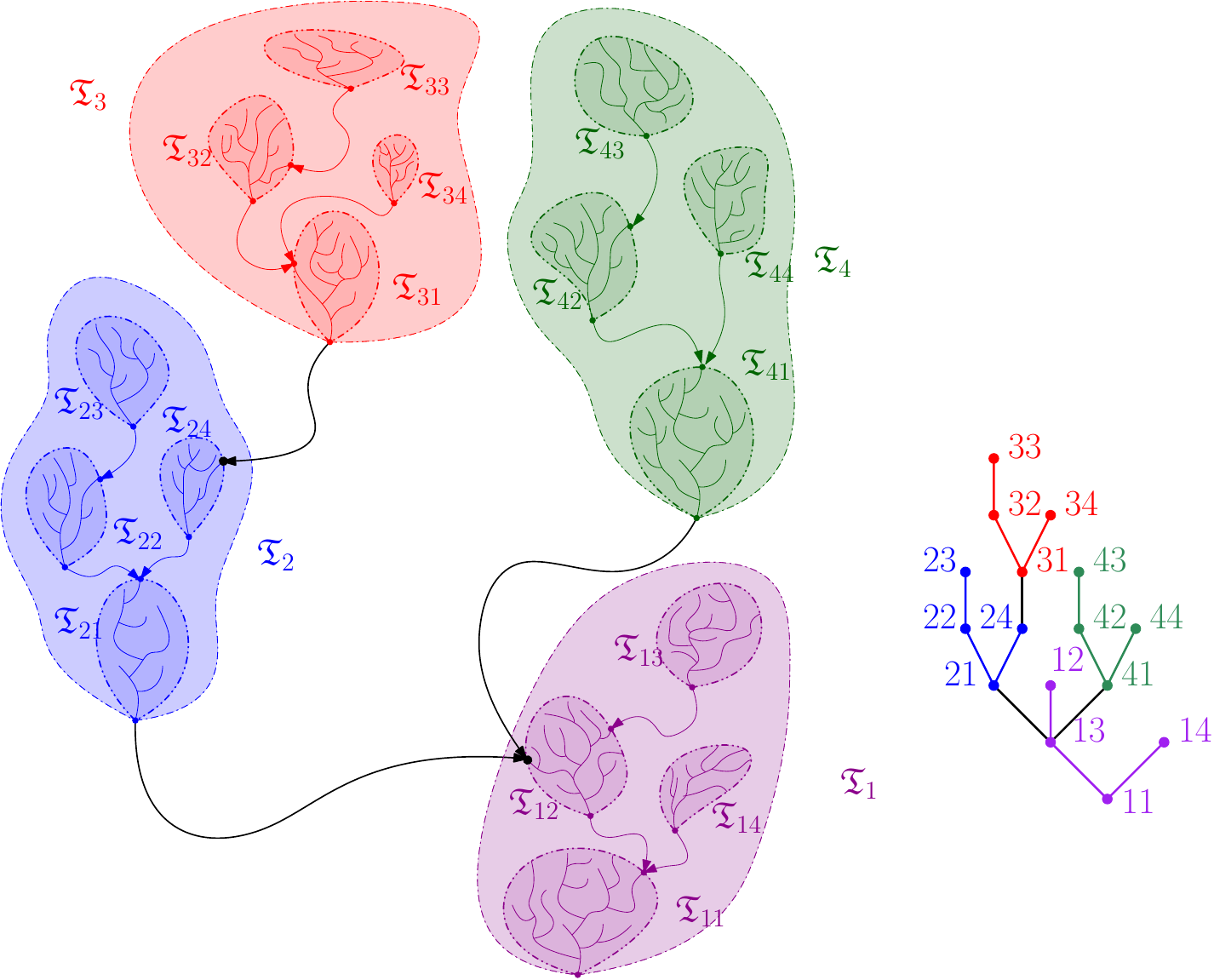}
\caption{\label{fig:phi-squared} The construction of $\fT$ from the first two levels of $\Theta$,
and the corresponding structural tree $\Gamma^2$ on the right-hand side.}
\end{figure}

\subsection{The upper Minkowski dimension: Proof of Theorem~\ref{thm:main2}} 
\label{sec:upbound}


In the following, for the sake of clarity, we write $a \cX$ for the metric space 
$(\cX, a d)$, where a constant $a>0$ and a metric space $(\cX,d)$. 
The approach here is to turn the heuristic presented in \eqref{heuris} into a rigorous argument. 

Given $x>0$, the idea consists in (1) expanding the recurrence for each subspace that has a scaling factor (for distances) greater than $x^\alpha$, until we are left with spaces (potentially in different generations) with scaling factor no greater than $x^\alpha$; (2) prove that, with high probability and for any $\epsilon>0$, we can cover the whole space by using a single ball of radius $x^{\alpha-\epsilon}$ per subspace, which ends up being a collection of no more than $x^{-1-\epsilon}$ balls, and finally (3) show that the bounds ensure that this occurs for all $x$ small enough with probability one. Let us move on to the details. 

With $x>0$ fixed, we first look for the suitable portion of the space that should intuitively be covered with balls of size roughtly $x^\alpha$. For $\vartheta\in \Theta\setminus \{\varnothing\}$, we let $\vartheta^+$ be the direct ancestor (parent) of $\vartheta$ in $\Theta$. Let $\sL_x:=\{\vartheta \in \Theta: \cV(\vartheta)\le x, \cV(\vartheta^+)>x\}$. The set $\sL_x$ separates the root $\varnothing$ from infinity, and no two elements of $\sL_x$ lie on the same ancestral path. It follows that $\sum_{\vartheta\in \sL_x} \cV(\vartheta)=1$ almost surely, and $\sL_x$ provides a natural partition of the entire space $\cT$. The cardinality of $\sL_x$ has been studied finely by Janson and Neininger \cite{JaNe2008}; we only need rather weak results: it turns out that under the conditions of Theorem~\ref{thm:main2} (their Condition A), as one would expect, $\Ec{\# \sL_x} = O(1/x)$ as $x\to 0$, and therefore, for any $\epsilon>0$, by Markov's inequality,
\[\pc{\sL_x > x^{-1-\epsilon}} \le C x^{\epsilon},\]
for some constant $C$. 

Now, for each $\vartheta \in \sL_x$, the corresponding portion of the space is stochastically smaller (coupling of the scale factor) than $x \cX$, and since the height has moments of all orders, by choosing $k$ as the smallest integer such that $k\epsilon > 2$, we have 
\begin{align*}
\pc{\exists \vartheta\in \sL_x: \cV(\vartheta)^\alpha \|\fT^\vartheta\|> x^{\alpha-\epsilon}}
& \le \pc{\exists \vartheta\in \sL_x: \|\fT^\vartheta\|> x^{-\epsilon}}\\
& \le Cx^\epsilon + x^{-1-\epsilon} \Ec{\|\fT\|^k} x^{k\epsilon}\\
& \le Cx^\epsilon + C' x^{1-\epsilon},
\end{align*}
where the constant $C'$ only depends on $\epsilon$. 

As a consequence, using balls of radius $x^{\alpha-\epsilon}$ centered at the roots $\rho^\vartheta$ for $\vartheta \in \sL_x$, it follows that $N_\cT(x^{\alpha-\epsilon})$ is such that $\pc{N_\cT(x^{\alpha-\epsilon})>x^{-1-\epsilon}}\le C'' x^\epsilon$ for $\epsilon\in (0,1/2)$. Choosing $x=2^{-n}$, a straightforward application of the Borel--Cantelli lemma implies that with probability one, $N_{\cT}(2^{-n(\alpha -\epsilon)}) \le 2^{n(1+\epsilon)}$ for all but finitely many natural numbers $n$. This is then easily extended to any value of $x$ by considering the unique $n\in \N$ for which $2^{-n}<x \le 2^{-n+1}$, and the fact that one then has $\cN_\cT(x^{\alpha-\epsilon}) \le \cN(2^{-n})$. It follows that with probability one, 
\[\limsup_{x\to 0} \frac{\log \cN_{\cT}(x^{\alpha-\epsilon})}{\log 1/x^{\alpha-\epsilon}} \le \frac{1+\epsilon}{\alpha-\epsilon}\,.\]
Since $\epsilon\in (0,1/2)$ is arbitrary, this completes the proof of Theorem~\ref{thm:main2}.

\subsection{Lower bound on the Hausdorff dimension: Proof of Theorem~\ref{thm:main3}} 
\label{sec:lowerhaus}
For a real tree $\sT_f$ with $f \in \Cex$, natural candidates for measures giving lower bounds on $\DimH(\sT_f)$ 
using the mass distribution principle \eqref{massdistributionprinciple}
are the push-forward measures $\mu^*_f := \mu^* \circ \pi_f^{-1}$ of measures $\mu^*$ on $[0,1]$ 
under the surjection $\pi_f: [0,1] \to \sT_f$. 
In our setting, in the case $\cR = \cS $ which covers both examples 
in \eqref{eqe} and \eqref{eqZ}, it is intuitive that the Lebesgue measure on $[0,1]$, 
or, equivalently, the canonical measure on $\sT$, leads to an efficient choice. 
But when scaling factors in time and space are independent 
such as in example \eqref{eqH}, it turns out that one first has to find an appropriate 
time-change on the unit interval in order to re-correlate the masses of fragments in the 
tree with the extent of distances in the corresponding subtrees. 
Time-changes constructed in this context are typically 
random, and one is led to construct the pair ``tree+time-change'' simultaneously. 
It is  this situation in which the almost sure construction in Section~\ref{sec:construction_as} 
turns out especially useful.  

The construction of a suitable time-change is presented in the following proposition which requires the introduction
of additional notation.
Let $\Phi_1$ be the analog of the 
map $\Phi$ defined in \eqref{def:phi} for $\alpha = 1$ that combines functions of the 
space $\Co_1 = \{f \in \Co : f \geq 0, f(0) = 0, f(1) = 1\}$. In other words:
$\Phi_1: \Co_1^K \times \Sigma_K^2 \times (0,1)^K \to \Co_1$, such 
that $g = \Phi(f_1, \ldots, f_K, \br, \bs, \bu)$ is the unique function of $\Co_1$ for which 
\begin{align*}
g(x) - g(y) =  r_{w_\ell}
\left[f_{w_\ell}(\varphi_{w_\ell}(x)) -  f_{w_\ell}(\varphi_{w_\ell}(y))\right],  
\end{align*}
for all $1 \leq \ell \leq L$ and $x,y \in I_\ell$ (see around the definition of $\Phi$ for the 
notation). 
Further, for $i \in \Gamma^o$ let $\Lambda_i^{\vartheta-}$ and $\Lambda_i^{\vartheta+}$ 
be the two half-open intervals forming $\Lambda_i^\vartheta$ 
where $\inf \Lambda_i^{\vartheta-} \leq \sup  \Lambda_i^{\vartheta+}$. 
For the various quantities playing a role in the proposition, we refer the reader to 
Section~\ref{tech-expan}.
\begin{prop}\label{prop:time_change}
Almost surely, for any $\vartheta\in \Theta$, there exists a probability measure $\mu^{\vartheta}$ 
on $[0,1] = \Lambda^\vartheta_\emptyset$, such that, for every $\sigma \in \Theta$, 
\begin{align*} 
\mu^{\vartheta}(\Lambda^\vartheta_\sigma)= \cV(\vartheta \sigma) / \cV(\vartheta).
\end{align*} 
The (random) distribution function $F^\vartheta$ of $\mu^{\vartheta}$ is 
measurable with respect to $\{\cR^{\vartheta \sigma}, \cS^{\vartheta \sigma}, 
\Xi^{\vartheta \sigma}: \sigma \in \Theta \}$, and if $\xi$ is an independent 
random variable uniform in $[0,1]$, then $\Ec{F^\vartheta(\xi)} = 1/2$. 
All $F^\vartheta$, $\vartheta\in\Theta$, have the same law and satisfy
\begin{align} \label{fix:tau}
F^{\vartheta} = \Phi_1( F^{\vartheta 1}, \ldots, 
 F^{\vartheta K}, \cR^\vartheta, \cS^\vartheta, \Xi^\vartheta)\,. 
\end{align}
Furthermore, in distribution, $F^\vartheta$ is the unique continuous distribution 
function on $[0,1]$ satisfying \eqref{fix:tau}.
\end{prop}

\begin{proof}
Let $\mu_n^\vartheta$ be the unique probability measure on $[0,1]$ which corresponds to 
the mass distribution on the partition $\{\Lambda^\vartheta_\sigma: \sigma \in \Theta_n \}$
such that $\mu_n^\vartheta(\Lambda^\vartheta_\sigma) = \cV(\vartheta \sigma) / \cV(\vartheta)$, and
$\mu_n^\vartheta$ has constant density on each of the sets $\Lambda^\vartheta_\sigma$, 
$\sigma \in \Theta_n$. 
This construction is consistent 
in the sense that, for $m \geq n$ and $\sigma \in \Theta_n$, we have
$\mu_m^{\vartheta}(\Lambda^\vartheta_\sigma) = \mu_n^{\vartheta}(\Lambda^\vartheta_\sigma).$
The measure $\mu^{\vartheta}$ is constructed as the almost sure limit of 
the sequence of random measures $(\mu_n^\vartheta)_{n\ge 0}$.
Denote by $F_n^\vartheta$ the distribution function of $\mu_n^\vartheta$. 
By construction, almost surely,
\begin{align} \label{FNr} F_n^{\vartheta} = \Phi_1( F_{n-1}^{\vartheta 1}, \ldots, 
 F_{n-1}^{\vartheta K}, \cR^\vartheta, \cS^\vartheta, \Xi^\vartheta). \end{align}
Note that $\{(F_n^\vartheta)_{n\ge 0}:  \vartheta \in \Theta\}$ is a family of 
identically distributed random variables. 
The uniform convergence of $F_n^\vartheta$ is shown analogously to the convergence of 
$Q_{n}^\vartheta$ in Proposition~\ref{prop:conlimit}: write $F_n^\vartheta$ as 
a telescoping sum, prove almost sure convergence of the corresponding series at an
independent uniform point, and bootstrap to almost sure uniform convergence; 
we omit the details. 
Here, it is important to note that, since $F_0^\vartheta(t)  = t$ for all $t \in [0,1]$, 
one can verify inductively that 
$\Ec{F_n^\vartheta(\xi)} = 1/2$ for all $n \geq 1$. 
The relevant constant that takes a value smaller than one, which allows to establish 
the convergence of $F_n^\vartheta$ at a uniformly chosen point by a contraction argument 
is $\sum_{i=1}^K \Ec{\beta_i^2}$ 
(this is similar to the constant in \eqref{boun2} with $\alpha$ replaced by $1$, just 
as $\Phi_1$ is similar to $\Phi$ with $\alpha$ replaced by $1$).
From the convergence of $F_n^\vartheta$ and \eqref{FNr}, it follows that
$F^\vartheta$ satisfies \eqref{fix:tau} and is a continuous distribution function. 
The induced measure $\mu^\vartheta$ satisfies the desired properties.
\end{proof}

In the remainder of the section (and here only), let 
$\bar \mu = \mu^\emptyset \circ \pi_\Xex^{-1}$
with $\mu^\emptyset$ as in Proposition~\ref{prop:time_change}.

\begin{lem} \label{lem:easycond}
Assume that $\Ec{\xoverline {\cR}^{-\delta}} < \infty$ for some $\delta > 0$ and let $\zeta$ be drawn 
on $\sT$ according to $\mu$ and $\bar \zeta$ be drawn according to $\bar \mu$, 
where these random variables are independent given $\fT$. 
Then, there exists an $\varepsilon > 0$, such that,
$$
\max \Big\{ 
\pc{d(\rho, \zeta) < r}, 
\pc{d(\rho, \bar \zeta) < r}, 
\pc{d(\zeta, \bar \zeta) < r} 
\Big\}
= O(r^\varepsilon), \quad r \downarrow 0.$$
\end{lem}

\noi \textsc{Remark}.
In most classical examples, we have $\mu = \bar \mu$ and invariance by 
rerooting at a random point distributed according to $\mu$. This applies in particular to $\mathfrak T_{\sZ}$ and $\mathfrak T_{\mathbf e}$. The tree $\mathfrak T_{\sH}$ is invariant under rerooting but the time-change is non-trivial. The random variable $\mathbf e(\xi)$ has the Rayleigh distribution, thus $ \pc{\mathbf e(\xi) \leq t} = 
1-e^{-t^2/8} = t^2 / 8 + o(t^2)$ as $t \to 0$. 

\begin{proof} 
Since the arguments are similar: we give all the details for the bound on $d(\rho,\zeta)$ 
and only discuss the relevant modifications for $d(\rho,\bar \zeta)$ and $d(\zeta,\bar\zeta)$. 
Recall the construction discussed at the beginning of this section involving 
the random tree $\Gamma^n$ for $n \geq 1$.

{\em i) The lower bound on $d(\rho,\zeta)$.}
Fix $n \geq 1$ and let $\zeta$ be as in the lemma. We define $\Gamma^*$ as the set of nodes
of $\Gamma^n$ containing all vertices $\vartheta \in \Theta_n$ such that $\cp_n(\sT^\vartheta)$ 
intersects the segment $\llb \rho, \zeta \rrb$. Note that $\Gamma^*$ forms a path in $\Gamma^n$.
By construction, $N = \# \Gamma^*$ is distributed as  the number of individuals 
in the $n$-th generation of a discrete-time branching process with offspring distribution
$\nu = \Law \left( 1 + \#E_J \right).$
Note that $\nu(\{0\}) = 0$ and $\nu(\{1\}) < 1$. For $\vartheta \in  \Gamma^*$, the 
contribution of the intersection of 
$\llb \rho, \zeta \rrb$ with $\cp_n(\sT^\vartheta)$ to $d(\rho, \zeta)$ is distributed as a 
scaled copy of $Y$, the height of a random point (see around \eqref{fix:Y}). More precisely, we have 
\begin{align} \label{idy} 
Y \stackrel{d}{=} \sum_{\vartheta \in \Gamma^*} \cV(\vartheta)^\alpha Y^\vartheta, 
\end{align} 
with a family of independent random variables $\{Y^\vartheta : \vartheta \in \Theta_n\}$ 
distributed as $Y$ 
which is independent of $\{(\cR^\vartheta, \cS^\vartheta) : |\vartheta| < n \}, \Gamma^*$. 
(Note that $\Gamma^*$ and $\{(\cR^\vartheta, \cS^\vartheta): |\vartheta| < n\}$ are typically 
not independent.) Therefore, with $m_n = \min \{\cV(\vartheta) : |\vartheta| = n\}$ and
$c > 0$, we obtain
$$\pc{{Y} \leq r}  
\leq \Prob{\sum_{ \vartheta \in \Gamma^*} e^{-c \alpha n } Y^\vartheta < r} 
+ \Prob{m_n < e^{-c n}}.$$ 
We now consider a crude (stochastic) bound on $N$ which turns out to be sufficient: 
as $\nu(\{0\}) = 0$, we may alternatively think of the branching process as a path from the 
root to a leaf in level $n$ such that all nodes on the path produce additional offspring in 
the next generation according to a copy of $\# E_J$, and all nodes created in this way reproduce 
according to $\nu$. Upon keeping only a single node in the offspring of the latter set of 
particles, we can bound $N$ from below by a random variable with a binomial distribution 
with parameters $n$ and $\nu([2, \infty))$ hereafter denoted by $\bin(n, \nu([2, \infty)))$.
Thus, for $0 < \gamma <  \nu([2, \infty))$, we have
\begin{align*} 
\pc{N < \lceil \gamma n \rceil } 
~\leq~ \p{\bin(n,  \nu([2, \infty))) < \lceil \gamma n \rceil} 
~\leq~ Ce^{- \varepsilon n}, 
\end{align*}
for some $C >0 $ and $\varepsilon > 0$ both depending on $\gamma$. Next, choose $\delta > 0$ such that 
$\Ec{\xoverline {\cR}^{-\delta}} < \infty$. Then, 
\begin{align*}
\pc{m_n < e^{-cn} } 
\leq K^n \Prob{\prod_{i=1}^{n} \xoverline {\mathcal R}_i <  e^{-cn}} 
\leq \exp\left(n\left[\log \left(K \Ec{\xoverline {\cR}^{-\delta}}\right) -c \delta\right]\right).
\end{align*}
With $c$ large enough, the right-hand side is $O(e^{-\delta' n })$ for some $\delta' > 0$. 
Hence, with $n = \lceil C \log (1/r)\rceil $ and independent copies $Y_1, \ldots, Y_n$ of $Y$, we obtain
\begin{align*}
\Prob{Y \leq r} 
& \leq \Prob{\sum_{i=1}^{\lceil \gamma n \rceil } e^{-c \alpha n} Y_i \leq r} 
+ O(e^{-\min(\delta',\varepsilon) n}) \\
& \leq \p{Y\le r e^{c\alpha n}}^{\lceil \gamma n \rceil} + O(e^{-\min(\delta',\varepsilon) n})\\
&  \leq  r^{-C\gamma \log \pc{Y \leq e^{\alpha c} r^{1-c \alpha C}}}  + O(r^{C \min(\delta', \varepsilon)}).
\end{align*}
Choosing $C < (\alpha c)^{-1}$  the exponent in the first term actually tends 
to $+\infty$ as $r\to 0$; hence this yields the assertion for $d(\rho, \zeta)$.

{\em ii) The lower bound on $d(\rho, \bar \zeta)$.}
We proceed similarly and write $\bar \Gamma$ for the set of nodes on the path in $\Gamma^n$ 
containing those $\vartheta \in \Theta_n$ with $\cp_n(\sT^\vartheta)$ intersecting
$\llb \rho, \bar \zeta  \rrb$. 
Let $\vartheta^* = 11\ldots 1$ denote the root of $\Gamma^n$. Upon disregarding 
the contribution of $\cT^{\vartheta^*}$, we have the following inequality which 
is the analogue of \eqref{idy} ($\ge_{\mathbb P}$ denotes stochastic order): 
\begin{align} \label{idbar} 
d(\rho, \bar \zeta)
\geq_{\mathbb P} \sum_{\vartheta \in \bar \Gamma \setminus \{\vartheta^*\}} \cV(\vartheta)^\alpha Y^\vartheta. 
\end{align}
Let $\bar N = \# \bar \Gamma - 1$. Then, $\bar N$ is distributed as the number of particles of type 2 in generation $n$ in a two-type branching process, where 
each particle reproduces independently with the following dynamics: the process starts at time
$0$ with a type 1 particle which remains alive forever. 
We can think of this particle as 
the subtree in the decomposition whose corresponding image contains $\bar \zeta$. 
Let $\bar J$ be a random variable satisfying
$\pc{\bar J=j\,|\, \cR,\cS}= \cR_j$ for $j= 1, \ldots, K$.
In each generation, the immortal particle gives birth to an independent number of type 2 
particles distributed as $\# E_{\bar J}$ corresponding to those subtrees intersecting
$\llb \rho, \bar \zeta  \rrb$ which do not contain $\bar \zeta$ and arise in the decomposition 
of the subtree associated with the immortal particle. 
Finally, since particles corresponding to (images of) subtrees intersecting
$\llb \rho, \bar \zeta\rrb$ which were created on earlier levels are further decomposed, 
every type 2 particle generates type 2 offspring according to $\nu$. (Type 2 particles do 
not generate offspring of type 1.)

Similarly to the bound on $N$ derived above, upon only keeping track of type 2 children of 
the immortal particle, we can bound $\bar N$ from below by a random variable with a 
binomial distribution with parameters $n, \Prob{E_{\bar J} \geq 2}$. 
As $\pc{E_{\bar J} \geq 2} > 0$, for $\gamma > 0$ sufficiently small, we find that
$\pc{\bar N  < \lceil \gamma n \rceil} $ decays exponentially in $n$. 
The claim now follows as above.

{\em iii) The lower bound on $d(\zeta,\bar \zeta)$.} 
Write $\Gamma'$ for the path of nodes in $\Gamma^n$ whose associated subtree (image) intersects 
$\llb \zeta, \bar \zeta  \rrb$. Let $\vartheta_1, \vartheta_2$ be the two (non-necessarily distinct) end-points of $\Gamma'$, and $\vartheta_3$ their highest common ancestor. Similarly to \eqref{idy} and \eqref{idbar}, we have 
\begin{equation}\label{eq:decomp-4types}
d(\zeta, \bar \zeta) 
\geq_{\mathbb P} 
\sum_{\vartheta \in \Gamma' \setminus \{\vartheta_2, \vartheta_3\}} 
\cV(\vartheta)^\alpha Y^\vartheta.
\end{equation}
Describing $N' = \# \Gamma' - 2$ now leads to a branching process with as many types as kinds of portions of the path between $\zeta$ and $\bar \zeta$, which correspond to rescaled copies $d(\zeta, \bar \zeta)$ (type 1), $d(\rho, \zeta)$ (type 2), $d(\rho,\bar \zeta)$ (type 3), and $d(\zeta,\zeta')$ (type 4), where $\zeta'$ is an independent copy of $\zeta$. The initial particle has type $1$; once it reproduces, type 1 never reappears and there is an immortal particle of type $3$; all other particles are of type $2$ except the one corresponding to the portion of the path in $\vartheta_3$ which as type $4$; hence the representation in \eqref{eq:decomp-4types}. Again, the number of children generated by the immortal particle up to time $n$ grows linearly in $n$ except on an event of exponentially small probability. We omit further details.
\end{proof}

The proof of Theorem \ref{thm:main3} requires more details about the tree $\Gamma^n$:
For a node $\vartheta \in \Gamma^n$, let $\cC'(\vartheta)$ denote the (random) set of its children 
in $\Gamma^n$. By construction, if $\vartheta=\vartheta_1\vartheta_2\dots\vartheta_n$, 
then $\cC'(\vartheta)$ only contains nodes of the form
$\vartheta_1\dots\vartheta_\ell\gamma 1\dots 1$ with $0 \leq \ell  < n$ and 
$2\le \gamma \le K$ satisfying
$\varpi_\gamma = \vartheta_{\ell+1}$ where $\vartheta_0 := 1$ (that is, $\gamma$ is a child of
$\vartheta_{\ell+1}$ in $\Gamma$).
By $\cC(\vartheta) \subseteq \cC'(\vartheta)$ we denote the subset of children of $\vartheta$ where, 
for any $1\leq \ell < n$, 
if $\vartheta_1 \ldots \vartheta_{\ell} \gamma 1 \ldots 1 \in \cC'(\vartheta)$ 
for some $\gamma$, we keep only that child with minimal $\gamma$. (We also keep a 
child $\gamma 1 \ldots 1$ with $\gamma \geq 2$ in  $\cC(\vartheta)$ if it exists in 
$\cC'(\vartheta)$.)
From Proposition~\ref{prop:degrees} \emph{iii)}, whose proof given in 
Section~\ref{sec:proof_degrees} does not make use of any results from the current section,
we know that all the trees corresponding to nodes in $\cC(\vartheta)$ are glued on the tree 
corresponding to $\vartheta$ at points that are distinct with probability one. By construction, 
$\cC(\vartheta)$ is a maximal set with this property, and $\#\cC(\vartheta) \le n$. 
Informally, $\# \cC(\vartheta)$ counts the number of distinct ``exit points'' of 
the subtree $\cp_n(\sT^\vartheta)$ in the decomposition of $\mathcal T$, that is, the points distinct from the root $\cp_n(\rho^\vartheta)$ where geodesics may 
leave $\cp_n(\sT^\vartheta)$. See Figure~\ref{fig:phi-squared} for an illustration of 
the construction and of the sets $\cC(\vartheta)$. In particular, in this figure, we have 
  $\cC'(12)= \{13, 21, 41\}$, $\cC(12)= \{13, 21\}$ while $\cC'(24) = \cC(24)=\{31\}$.

\begin{proof}[Proof of Theorem~\ref{thm:main3}]
Fix $n \in \N$ and $\gamma < 1 / \alpha$. We need to show that $\DimH(\fT) \geq \gamma$ 
almost surely. 
Let $\cP = \bigcup_{\vartheta \in \Theta_n}  \cp_n(\{\rho^\vartheta\})$. For $x \in \mathcal T \setminus \mathcal P$ let $\vartheta(x)$ be the unique node in $\Gamma^n$ with
$x \in \cp_n(\sT^{\vartheta(x)})$.
Set  $\mathfrak T(x) = 
\fT^{\vartheta(x)}$. 
Furthermore, let $H_x = d(x, \cp_n(\rho^{\vartheta(x)}))$ be the height of $x$ in (the image of) $\sT(x)$ and
$E_x = d(x, \sT \setminus \cp_n(\sT(x)))$ the distance to exit (the image of) $\sT(x)$ from $x$. 
With $\cC(x) := \cC(\vartheta(x))$, for all $x\in \sT$, we have 
\[
E_x = 
\left\{
\begin{array}{ll}
\min_{\sigma \in \cC(x)} d(x, \cp_n (\sT^\sigma)) & \text{if } \vartheta(x)= 1\dots 1\\
\min_{\sigma \in \cC(x)} d(x, \cp_n (\sT^\sigma)) \wedge H_x &\text{if } \vartheta(x)\ne 1\dots 1\,.
\end{array}
\right.
\]
Recall that $B_r(x) = \{y \in \sT: d(x,y)< r \}$ for $x \in \sT$, $r > 0$.
For any $x \in \sT \setminus \mathcal P$ and $r > 0$, we have either $B_r(x) \subseteq \cp_n(\cT(x))$ or $E_x \le r$. 

We aim at using the mass distribution principle with the measure $\bar \mu = \mu^\emptyset \circ \pi_{\mathcal Z}^{-1}$ constructed thanks to Proposition~\ref{prop:time_change}. As in the previous proof, we let 
$\bar \zeta$ be a random variable drawn on $\mathcal T$ according to $\bar \mu$ (given $\mathfrak T$). 
Formally, we can let $\bar \zeta = \pi_{\mathcal Z}((F^{\emptyset})^{-1}(\xi))$ for a random variable $\xi$ with the 
uniform distribution on $[0,1]$ which is independent of all remaining quantities. 
As $\bar \mu$ has no atoms  ($F^\emptyset$ is continuous and Proposition~\ref{prop:degrees} \emph{iii)}), we obtain
\begin{equation*}
\pc{\bar \mu(B_{r}(\bar \zeta)) > r^\gamma} 
\leq \pc{\bar \mu(\cp_n(\sT(\bar \zeta))) > r^\gamma} + \pc{ E_{\bar \zeta} \leq r}.
\end{equation*}
For any $\vartheta \in \Theta_n$, we denote 
by $\sigma_0(\vartheta), \ldots, \sigma_{n-1}(\vartheta)$ 
the potential elements of $\cC(\vartheta)$ where, seen as words on $[K]$,
$\sigma_\ell(\vartheta)$ and $\vartheta$ have a common prefix of length $\ell$. 
Then, abbreviating $\sigma_\ell := \sigma_\ell(\vartheta(\bar \zeta))$, 
\begin{align} \label{eq:stt}  
\pc{E_{\bar \zeta} \le r}  \le \p{H_{\bar \zeta} \le r} 
+ \p{\bigcup_{\ell=0}^{n-1} \{d(\bar \zeta, \cp_n(\cT^{\sigma_\ell})) \leq r, \sigma_\ell \in \cC(\bar\zeta)\}  }. 
\end{align}

Let $\bar \zeta' = \pi_\Xex((F^{\emptyset})^{-1}(\xi'))$ and $\zeta''=\pi_\Xex(\xi'')$, where $\xi',\xi''$ are independent random variables with uniform 
distribution on $[0,1]$, independent of the remaining quantities.
Then, for $\ell \in \{0,1,\dots, n-1\}$, we have
\begin{align*}
\p{d(\bar \zeta, \cp_n(\cT^{\sigma_\ell})) \leq r, \sigma_\ell \in \cC(\bar \zeta)} 
& = \pc{\cV(\vartheta(\bar \zeta))^\alpha \cdot d(\bar \zeta', \zeta'') \leq r, \sigma_\ell \in \cC(\zeta)} \\
&  \leq \pc{\cV(\vartheta(\bar \zeta))^\alpha \cdot d(\bar \zeta', \zeta'') \leq r}. 
\end{align*}
Similarly, $H_{\bar \zeta}$ is distributed like $\cV(\vartheta(\bar \zeta))^\alpha \cdot d(\rho, \bar \zeta')$, for $\bar \zeta'$ an independent copy of $\bar \zeta$. 
Let $\eta \in (0,1)$ be a parameter to be chosen later. 
Applying the union bound on the right-hand side of \eqref{eq:stt} yields
\begin{align*}
\pc{E_{\bar \zeta} \le r} 
\leq  n \big \{ & \pc{\cV(\vartheta(\bar \zeta)) \leq r^{\eta / \alpha}} 
+ \pc{d(\rho, \bar \zeta) \leq r^{1-\eta}}\\ 
& + \pc{d(\bar \zeta, \zeta'') \leq r^{1-\eta}} \big \}\,.
\end{align*}
As $\bar \mu( \cp_n(\sT(\bar \zeta))) = \cV(\vartheta(\bar \zeta)),$ combining the bounds and 
using Lemma~\ref{lem:easycond}, we see that there exists universal constants $\varepsilon_1 > 0$ and $C > 0$ such that
\begin{equation*}
\pc{\bar \mu(B_{r}(\bar \zeta)) > r^\gamma} 
\leq \Prob{\cV(\vartheta(\bar \zeta)) > r^\gamma} 
+ C n \left\{ \pc{\cV(\vartheta(\bar \zeta)) \leq r^{\eta / \alpha}} + r^{\varepsilon_1 ( 1 - \eta)} \right\}\,.
\end{equation*}

We now choose the parameters. Note that $-\log \cV(\vartheta(\bar \zeta))$  is distributed like the sum of $n$ independent copies of the random variable $-\log \sum_{i=1}^K \I{\bar J=i} \cR_i$; furthermore, as $\cR^* := \sum_{i=1}^K \I{\bar J=i} \cR_i$ is stochastically larger 
than $\xoverline \cR$, the tail bound on $\xoverline \cR$ implies that
$-\log \cR^*$ has exponential moments, and hence that the expected value $q:=\Ec{-\log \cR_{\bar J}}<\infty$ governs the asymptotics. First, let $\eta\in (\gamma \alpha,1)$. Then choose $\delta\in (\gamma/q, \eta / (q\alpha))$. Finally, let $n = n(r) =  \lfloor - \delta \log r \rfloor$.  
Hence, by Cram{\'e}r's theorem for large deviations, there exist 
$C_2, \varepsilon_2 >0$ (depending on the remaining parameters but not on $r$), such that
$\pc{\cV(\vartheta(\bar \zeta)) > r^\gamma} \leq C_2 r^{\varepsilon_2}$ and 
$\pc{\cV(\vartheta(\bar \zeta)) \leq r^{\eta / \alpha}} \leq C_2 r^{\varepsilon_2}$.
Summarizing, there exists $C > 0$ (which may depend on all parameters but not on $r$), 
such that, 
\begin{equation*}
\pc{\bar \mu(B_{r}(\bar \zeta)) > r^\gamma} 
\leq C \log(1/r) \cdot (r^{\varepsilon_2} + r^{\varepsilon_1(1-\eta)} ), \quad 0 < r< 1.
\end{equation*}
It follows that for $r_n = 2^{-n}$, 
$$\sum_{n \geq 1} \pc{\bar \mu (B_{r_n}(\bar \zeta)) > r_n^\gamma} < \infty.$$
Hence, by the Borel--Cantelli lemma, almost surely, 
$$\limsup_{r\to 0} \frac{\bar \mu(B_{r}(\bar \zeta))}{r^\gamma} \leq 2.$$
Thus, denoting  $A =  \{x \in \sT: \limsup_{r\to 0} \bar {\mu}(B_{r}(x))/r^\gamma \leq 2\}$, 
we have $1 = \Prob{\bar \zeta \in A} = \Ec{\bar {\mu} (A)}$ 
implying $\bar {\mu} (A) = 1$ 
almost surely.
From \eqref{massdistributionprinciple}, it follows that, almost surely, 
$\DimH(\sT) \geq \DimH(A) \geq \gamma$ which completes the proof. 
\end{proof}
\color{black}

\subsection{Degrees and properties of the encoding: Proof of Proposition~\ref{prop:degrees} }
\label{sec:proof_degrees}

The proofs of the missing parts of Theorem~\ref{thm:treeuniq}, i.e., \emph{iii)} and  
\emph{iv)}, rely on the dynamics governing the number of points in the set $\cC(\vartheta)$ 
as $\vartheta$ follows a path in $\Theta$ away from the root. 
The following lemma is straightforward from the construction, 
since the exit points counted by $\# \cC(\vartheta)$ are chosen on $\cT^\vartheta$ according 
to the mass measure $\mu^\vartheta$. Recall that, for $\vartheta\in \Theta$, 
we write $\scL(\vartheta)$ for the Lebesgue measure of the set
$\Lambda_{\vartheta}$ (see \eqref{def_ell}). Recall also that $\Gamma^o=\Gamma\setminus \partial \Gamma$, where $\partial \Gamma$ is the set of leaves of $\Gamma$.

\begin{lem}\label{lem:MC_exit}
{i)} Let $\varepsilon_1, \varepsilon_2, \ldots \in [K]$ and, for each $n\ge 1$,
$\vartheta_n = \varepsilon_1 \ldots \varepsilon_n 
\in \Theta_n$. Then, $\scL(\vartheta_{n+1}) = \scL(\vartheta_n)\cdot \cS^{\vartheta_n}_{\varepsilon_{n+1}}$.
The sequence $( \# \cC(\vartheta_n), \scL(\vartheta_n)), n \geq 0$ is a 
Markov chain on $\N \times [0,1]$ starting at $(0,1)$, whose evolution can be 
described as follows: given $(\# \cC(\vartheta_n), \scL(\vartheta_n))$, we have
$$
(\# \cC(\vartheta_{n+1}), \scL(\vartheta_{n+1})) 
= (\I{\varepsilon_{n+1} \in \Gamma^o} + \bin\big(\# \cC(\vartheta_n), \cS^{\vartheta_n}_{\varepsilon_{n+1}}
\big), 
\scL(\vartheta_n)\cdot \cS^{\vartheta_n}_{\varepsilon_{n+1}}).
$$
{ii)} Let $\xi$ be uniformly distributed on $[0,1]$, independent of all remaining quantities 
and, for $n \geq 0$, let $\tilde \vartheta \in \Theta_n$ be the unique node with 
$\xi \in \Lambda_{\tilde \vartheta}$. 
Define $(\tilde \cC_n, \tilde{\scL}_n) := ( \# \cC(\tilde \vartheta_n), \scL(\tilde \vartheta_n))$. 
Then, the sequence $( \tilde \cC_n, \tilde \scL_n), n \geq 0$ is a 
homogeneous Markov chain on $\N \times [0,1]$ starting at $(0,1)$, whose evolution 
can be described as follows: given $( \tilde \cC_n, \tilde \scL_n)$, we have
$$
(\tilde \cC_{n+1}, \tilde \scL_{n+1}) 
= (\I{J_{n+1} \in \Gamma^o} + \bin(\tilde \cC_{n}, \cS^{n+1}_{J_{n+1}}), 
\tilde \scL_n \cdot \cS^{n+1}_{J_{n+1}}),
$$
where $(\cS^n)_{n\ge 0}$ is a family of independent copies of $\cS$ 
and $\pc{J_{n+1} = i \,|\,\cS^{n+1}} = \cS^{n+1}_i$.
\end{lem}

\begin{proof}[Proof of Proposition~\ref{prop:degrees} i), ii) and iii)]
We start with the proof of \emph{i)}.  
Let $A = \{ s \in [0,1] : \Xex(s) = 0\}$ denote the zero-set of $\Xex$. 
For $n \geq 1$, let $\Lambda^n = \Lambda_{\vartheta_n}$ where $\vartheta_n = 1\ldots 1$. Then for every 
$n\ge 0$, $\cp_n(\cT^{\vartheta_n})$ is the subtree that contains 
the root $\rho$ of $\fT$. Furthermore, $\Lambda^n$ is the union of $\# \cC (\vartheta_n) + 1$ 
 disjoint intervals. Clearly, $(\Lambda^n)_{n \geq 1}$ is decreasing and we set 
$\Lambda := \bigcap_{n \geq 1} \Lambda^n$. 
Since $\Ec{\cS_1} < 1$, Lemma~\ref{lem:MC_exit} {\em i)} and a routine drift argument (see, e.g., 
Chapter~8 of \cite{MeTw1993}) shows that, almost surely,  $\# \cC (\vartheta_n) = 1$ infinitely 
often, and thus $\Lambda^n$ consists of only two intervals for infinitely many $n$. 
As $\{0,1\}\in A$, for any $n \geq 1$ with this property, we have
\[\Lambda^n 
\subseteq [0, \inf \{t > 0 : t \notin \Lambda^n\}] \cup [\sup \{t < 1 : t \notin \Lambda^n \}, 1]\,.\]
Since $\scL(\vartheta_n)=\Leb(\Lambda^n)\to 0$ with probability one, it follows that 
$\Lambda = \{0,1\}$ almost surely. 

Now, for any $s\not\in \Lambda$, there is some $n$ large enough for which $s\not \in \Lambda^n$ and $\Lambda^n$ consists of two intervals. Then, since the path between $\rho$ and the projection of $s$ in the tree must cross $\varphi_n^\circ(\cT^{\vartheta_n})$, $\Xex(s)$ is at least a rescaled copy of $\Xex(\xi)$, for a uniform random variable $\xi$ which is independent of $\Xex$. However, in the context of fixed point equation \eqref{fix:Y}, we have already seen that $\Xex(\xi) > 0$ almost surely (see \eqref{01law}). It follows that $A \subseteq \Lambda$, and since $\Lambda=\{0,1\}$ the root $\rho$ of $\fT$ is a leaf.

{\em ii)} 
Since $\mu_{\Xex}$ has full support, this reduces to showing that for a random point $\zeta$ 
sampled from $\mu_\Xex$ and any $\varepsilon>0$, with probability one, there exists 
some $x\in \cT = \cT_{\Xex}$ such that $\zeta$ lies in the subtree of $\cT$ 
rooted at $x$, namely $\cT^{\uparrow}(x) := \{u\in \cT: x\in \llb \rho, u \rrb\}$,
and $\mu_{\Xex}(\cT^\uparrow(x))<\varepsilon$. 
To prove this, we choose $\zeta = \pi_{\Xex}(\xi)$ with $\xi$ as in Lemma~\ref{lem:MC_exit} {\em ii)}, and follow $\zeta$ in the refining decomposition of the tree 
according to $\Theta_n$, as $n$ increases. 
For $n\ge 1$, let $\tilde \vartheta_n\in \Theta_n$ be as in Lemma~\ref{lem:MC_exit} {\em ii)}. 
In particular, $\zeta \in \cp_n(\cT_{\Xex^{\tilde \vartheta_n}})$.
Thus, $\zeta$ lies in the subtree of $\cT$ rooted at $\cp_n(\rho^{\tilde \vartheta_n})$, and thus it suffices to show that for any $\varepsilon>0$, 
\[ A_\varepsilon=\inf\{n\ge 0: \mu_{\Xex}(\cp_n(\cT_{\tilde \vartheta_n})) < \varepsilon \}<\infty\,.\]
Observe that, if $\tilde {\scL}_n = \mu_{\Xex}(\cp_n(\cT_{\tilde \vartheta_n}))<\varepsilon$ 
and $\tilde \cC_n =0$, then $A_\varepsilon\le n$. 
As in the proof of {\em i)} above, since $\Ec{\cS^{n}_{J_{n}}}<1$, a classical drift 
argument shows that $(\tilde \cC_n)_{n\ge 0}$ is positive recurrent, and in particular,
$\tilde \cC_n=0$ infinitely often. Then, for some subsequence $(n_i)_{i\ge 1}$ 
with $n_i\to \infty$, we have $\tilde \cC_{n_i}=0$. But, by Lemma~\ref{lem:MC_exit}, 
$\tilde \scL_{n}\to 0$ almost surely, so that 
there is an $i_0$ for which $\tilde \scL_{n_i}<\varepsilon$ for all $i\ge i_0$. One then has 
$A_\varepsilon\le n_{i_0} <\infty$, which completes the proof.

Finally, we consider  \emph{iii)}. Since $d(\rho, \zeta) > 0$ almost surely, no mass can add up at exit points in the construction of the excursion $\Xex$ in \eqref{fix:X}. Hence, 
$$\e \bigg[\sup_{t \in [0,1]} \mu_\Xex ( \{\pi_\Xex(t)\})\bigg] 
\leq \E{\max (\mathcal S_1, \ldots, \mathcal S_K)} \cdot \e \bigg[\sup_{t \in [0,1]} \mu_\Xex ( \{\pi_\Xex(t)\})\bigg].$$ 
It follows that the left-hand side is zero which concludes the proof.
\end{proof}

\begin{proof}[Proof of Theorem~\ref{thm:treeuniq}  iii) and  iv)]
The point {\em iii)} was established in the proof of Proposition~\ref{prop:degrees} {\em i)}.
The argument we have used above for the proof of Proposition~\ref{prop:degrees} {\em ii)} also 
implies {\em iv)}, that is, almost surely, $\Xex$ is nowhere monotonic: indeed, if this  were not the case, 
then, with positive probability, a randomly chosen point would be contained in an 
interval where $\Xex$ is monotonic.
In particular, for some $\varepsilon > 0$ and a uniformly chosen point $\xi$, $\Xex(\xi)$ would fall in a interval of length at least $\varepsilon$ on which $\Xex$ is monotonic with positive probability.
Recall the quantites $\tilde \vartheta_n$ and $\tilde \cC_n$ from the proof  of Proposition~\ref{prop:degrees} {\em ii)}. 
Let $I\subseteq \N$ be the (a.s.\ infinite) set of indices $n$ for which $\tilde \cC_n=0$. 
For $n\in I$, the set $\Lambda_{\tilde \vartheta_n}\subseteq [0,1]$ consists of a 
single half-open interval, and $\xi\in \Lambda_{\tilde \vartheta_{n}}$. 
For every $n\in I$ large enough, we have $\scL(\tilde \vartheta_n) < \varepsilon/2$. 
But $\Xex$ is non-monotonic 
on $\Lambda_{\tilde \vartheta_{n}}$ since $\xi\in \Lambda_{\tilde \vartheta_n}$ 
and $\Xex(\inf \Lambda_{\tilde \vartheta_{n}})=\Xex(\sup \Lambda_{\tilde \vartheta_n})< \Xex(\xi)$ 
almost surely by Theorem~\ref{thm:treeuniq} {\em iii)}. This concludes the proof.
\end{proof}

\begin{proof}[Proof of Proposition~\ref{prop:degrees} iv)]
 Let $\fT = \fT_\Xex$. 
First note that $2 \in \sD(\cT)$ almost surely since $\mu$ has no atoms, 
the set of branch points of $\cT$ is at most countable, and $\cT$ is not reduced to a point. 
Further, since $\Xex$ is nowhere monotonic, there exist local minima, and hence branch points. 
It follows that the maximum degree of $\cT$ is at least $3$ (and, a priori, possibly infinite). 

Next, for every $k \in \sD(\Gamma), k \geq 3$, almost surely, there exists a point $x \in \sT$ with degree $k$. This follows immediately from the fixed point equation satisfied by $\fT$ and the fact that the root  has degree $1$ and $\mu$ is concentrated on the leaves (see parts \emph{i)} and \emph{ii)} of the proposition). For example, 
if $a \in \Gamma$ has degree $k$, then (the image of) the root of $\sT_{a+1}$ is a point in $\sT$ with degree $k$.
The main part of the proof consists in showing that, if $k+1 \notin \sD(\Gamma), k \geq 3$, then there do not exist points in $\sT$ with degree $k+1$. To this end, 
 let $(U_i)_{i\ge 1}$ be a family of independent r.v.\ 
with the uniform distribution on $[0,1]$ and $\zeta_i = \pi_{\Xex}(U_i)$ be the 
corresponding test points on $\sT$. Further, let $U_0 = 0$
and $\zeta_0 = \rho$.
Since $\mu$ has full support on $\cT$, if there 
exists $x\in \cT$ with degree $k+1$, then all connected components of $\cT \setminus \{x\}$ have 
positive mass. In particular, with positive probability, any two of the segments $\llb \zeta_i,\zeta_j\rrb$, 
$0\le i<j\le k$, intersect at a point with degree $k+1$.  Let $\cI$ be the collection of points in $\cT$ that are 
contained in some intersection $\llb \zeta_i,\zeta_j\rrb\cap \llb \zeta_{i'},\zeta_{j'}\rrb$, 
for distinct elements $i,i,j,j' \in \{0,\dots, k\}$. 
Let $T \geq 1$ be the minimal integer such that, for some pairwise distinct nodes $\sigma_i, i=0, \ldots, k$ in $\Theta_T$, we have $U_i \in \Lambda_{\sigma_i}, i = 0, \ldots, k$. Clearly, $T < \infty$ almost surely as $\max_{\sigma \in \Theta_n} \scL(\sigma) \to 0$ almost surely. Let us denote the most recent common ancestor of the nodes $\sigma_i, i=1, \ldots, k$ by $\sigma^*$. (That is, $\sigma^*$ is the node of maximal depth such that all $\sigma_i, i = 1, \ldots, K$ are contained in its subtree.) Write $\sigma^* = \bar \sigma^* \varepsilon$ with $\bar \sigma^* \in \Theta_{T-1}$ and $\varepsilon \in [K]$. $\cI$ is a singleton if and only if the nodes $\sigma_i, i=1, \ldots, k$ lie in subtrees of pairwise distinct nodes of the form $\bar \sigma^* \gamma$, where $\varpi_\gamma = \varepsilon$. In particular, 
$x$ is equal to the roots of trees $\sT_\sigma, \sigma \in \{\bar \sigma^* \gamma : \varpi_\gamma = \varepsilon \}$ under the canonical surjection $\sqcup_{\sigma \in \Theta_T} \sT_{\sigma} \to \sT$. As roots have degree 1 and the mass measure is concentrated on the leaves, the degree of $x$ must be equal to the degree of $\varepsilon$ in $\Gamma$. In particular,  the degree of $x$ lies in the set  $\sD(\Gamma)$. Summarizing, almost surely, the set $\cI$ does not exist of a singleton with degree $k+1 \notin \sD(\Gamma)$ which concludes the argument. 
$\sup\sD(\cT)<\infty$ almost surely can be deduced easily. For, if this was not the case, for any natural number $L \geq 3$, the above construction with $k=L-1$ would show that, with positive probability, there exists $L' \geq L$ such that $\cI$ is reduced to a singleton of degree $L' \in \sD(\Gamma)$. Choosing $L$ strictly larger than the maximal degree in $\Gamma$ contradicts this fact. It remains to prove that, if
$\max \sD(\Gamma) \geq 4$ and $3 \notin \sD(\Gamma)$, then, almost surely, there are no nodes of degree 3 in $\sT$. Set
$\cI = \llb \zeta_1,\zeta_2 \rrb\cap \llb \zeta_{1},\zeta_{3}\rrb \cap \llb \zeta_{2},\zeta_{3}\rrb$. Almost surely, $\cI$ is a singleton, say $x$. There exists a node with degree $3$ in $\sT$ with positive probability if and only if, with positive probability, $x$ has degree 3. Let $\vartheta_1, \vartheta_2 \in [K]$ such that $U_i \in \Lambda_{\vartheta_i}, i =1, 2$. Let $\sigma^*$ be their most recent common ancestor. If $U_1, U_2 \in \Lambda_{\sigma^*}$, then let $U_i'  = \varphi_{\sigma^*}(U_i)$, where $\varphi_{\sigma^*}$ is defined in \eqref{eq:def_sets}. If $U_1 \in \Lambda_{\sigma^*}$ and
$U_2 \notin \Lambda_{\sigma^*}$, then set $U_1' = \varphi_{\sigma^*}(U_1)$ and $U_2' = \xi_{\sigma^*}$. Proceed analogously if the roles of $U_1, U_2$ are interchanged. Note that, given that one of these cases occurs, the random variables $U_1', U_2'$ are independent, uniformly distributed on $[0,1]$ and independent of $\{(\mathcal R^\vartheta, \mathcal S^\vartheta, \Xi^\vartheta) :  |\vartheta|>  0\}$. Let $\zeta_i' = \pi_{\Xex^{\sigma^*}}(U_i'), i=1,2,3$, where $U_3' = 0$. Note that, $x$ has degree $3$ if and only if the node $ \llb \zeta_1',\zeta_2' \rrb\cap \llb \zeta'_{1},\zeta'_{3}\rrb \cap \llb \zeta'_{2},\zeta'_{3}\rrb$ in $\sT_{\sigma^*}$ has degree $3$. By the same arguments as above, if $U_1, U_2 \notin \Lambda_{\sigma^*}$, then the degree of $x$ is equal to the degree of $\sigma^*$ in $\Gamma$ and therefore at least 4. It follows that
\begin{align*} \Prob{\text{deg}(x) = 3} & = \Prob{\text{deg}(x) = 3} (1 - \Prob{U_1, U_2 \notin \Lambda_{\sigma^*}}) \end{align*}
As $\Prob{U_1, U_2 \notin \Lambda_{\sigma^*}} > 0$ it follows that $\Prob{\text{deg}(x) = 3} = 0$ which concludes the proof.
\end{proof}

\subsection{Optimal H\"older exponents} \label{sec:proof_hoelder}

The proof of Corollary~\ref{cor:opt_Hol} merely consists in putting together the information we have gathered in previous sections.

\begin{proof}[Proof of Corollary~\ref{cor:opt_Hol}] {\em (a)} Note that under the conditions 
of Theorem~\ref{thm:main2}, $\Ec{\xoverline{\cR}^{-\delta}}<\infty$ for 
any $\delta\in (0,\gamma)$. Thus the conclusion of Theorem~\ref{thm:main3} holds. 
On the one hand, we have $\DimuM(\cT_{\Xex})\ge \DimH(\cT_\Xex)\ge \alpha^{-1}$ because of 
Theorem~\ref{thm:main3}; on the other hand, $\DimoM(\cT_{\Xex})\le \alpha^{-1}$. 

{\em (b)} By Theorem~\ref{thm:treeuniq}, $\Xex$ is nowhere constant, hence 
\eqref{aiz} holds and $\alpha=\omega_\Xex$. By definition of $\omega_\Xex$, 
for any $\gamma<\omega_\Xex=\alpha$, almost surely, there exists a process $\tilde \Xex$ 
time-change equivalent to $\Xex$ with $\gamma$-H\"older continuous paths.

{\em (c)} This follows immediately from part {\em (a)} since the existence of a process $\Xex'$ whose sample paths 
are $\gamma$-H\"older continuous with positive probability for $\gamma>\alpha$ would contradict the statement there.
\end{proof}

\section{Applications} \label{sec:app}

In the first three sections, we discuss the corollaries stated in Section~\ref{sec:taste_app} 
concerning the processes $\mathbf e$, $\sZ$, and $\sH$. Then, in the fourth section, we study a new 
application.

\subsection{The Brownian continuum random tree} \label{sec:e}

The Brownian CRT $\sT_{\mathbf e}$ encoded by a Brownian excursion is a fundamental 
tree arising as scaling limit for various classes of random trees. We quote the classical 
case of uniform random labelled trees \cite[Theorem 2]{Aldous1991b}, but also binary unordered 
unlabeled trees (Otter trees) \cite{MaMi2011a}, random trees with a prescribed degree 
sequence \cite{BrMa2014a}, general unordered unlabeled trees (P\'olya trees) \cite{PaSt2015a}, 
unlabeled unrooted trees \cite{Stufler2014a}, and random graphs from subcritical classes 
\cite{PaStWe2014a} to name a few examples. (See also \cite{Legall2005}.)

\begin{thm} [Aldous \cite{Aldous1991b}, see also \cite{Legall2005}]
Let $T_n$ be the family tree of a critical branching process 
with offspring mean one and finite 
offspring variance $\sigma^2$ conditioned on having $n$ vertices. Let $d_n$ denote the graph distance 
on $T_n$ and $\mu_n$ the uniform probability measure on the leaves. Then, as $n \to \infty$, in 
distribution with respect to the Gromov--Hausdorff--Prokhorov distance, 
$$(T_n, \tfrac\sigma 2 \cdot n^{-1/2} d_n, \mu_n, \rho_n) \to (\sT_{\mathbf e}, d_{\mathbf e}, \mu_{\mathbf e}, \rho_{\mathbf e}).$$
\end{thm}
Corollary~\ref{coro1} {\em (a)} and Corollary~\ref{cor:uni} {\em (a)} follow immediately from  \eqref{eqe} and Theorem \ref{thm:treeuniq}. Similarly, 
Corollary~\ref{cor:main} {\em (a)} follows from Theorems~\ref{thm:main2} and \ref{thm:main3} noting that $\xoverline {\cR}$ has the $\text{Beta}(1/2, 1)$ 
distribution with density $\tfrac 1 2 t^{-1/2} \mathbf 1_{[0,1]}(t)$.

\subsection{Random self-similar recursive triangulations of the disk} \label{sec:Z}

\begin{figure}[t]
    \centering
    \begin{minipage}{.35\linewidth}
    \begin{tikzpicture}
        \pgftransformscale{.68}
        \def \r {3}
        \def \s {1pt}
        \def \st {2pt}
        \draw (0,0) circle (\r);
        \draw[fill, black] (0:\r) circle (\s) node[right] {$1$};
        \draw[fill, black] (30:\r) circle (\s) -- (180:\r) circle (\s);
        \draw[fill, black] (10:\r) circle (\s) -- (310:\r) circle (\s);
        \draw[fill, black] (20:\r) circle (\s) -- (270:\r) circle (\s);
        \draw[fill, black] (190:\r) circle (\s) -- (260:\r) circle (\s);
        \draw[fill, black] (40:\r) circle (\s) -- (85:\r) circle (\s);
        \draw[fill, black] (90:\r) circle (\s) -- (175:\r) circle (\s);
        
        \tikzstyle{cblue}=[circle, draw, thick, fill=blue!60, scale=0.5]
        
        \foreach \place/\x in {{(90:\r*.5)/1}, {(250:\r*.2)/2}, {(-35:\r*.75)/3}, {(-135:\r*.9)/4}, {(62:\r*.97)/5}, {(135:\r*.9)/6}, {(-20:\r*.95)/7}}
        \node[cblue] (a\x) at \place {};
        
        \path[very thick] (a1) edge (a6);
        \path[very thick] (a1) edge (a5);
        \path[very thick] (a1) edge (a2);
        \path[very thick] (a2) edge (a4);
        \path[very thick] (a2) edge (a3);
        \path[very thick] (a3) edge (a7);
    \end{tikzpicture}
\end{minipage}
\hfil
\begin{minipage}{.35\linewidth}
    \includegraphics[scale=.25]{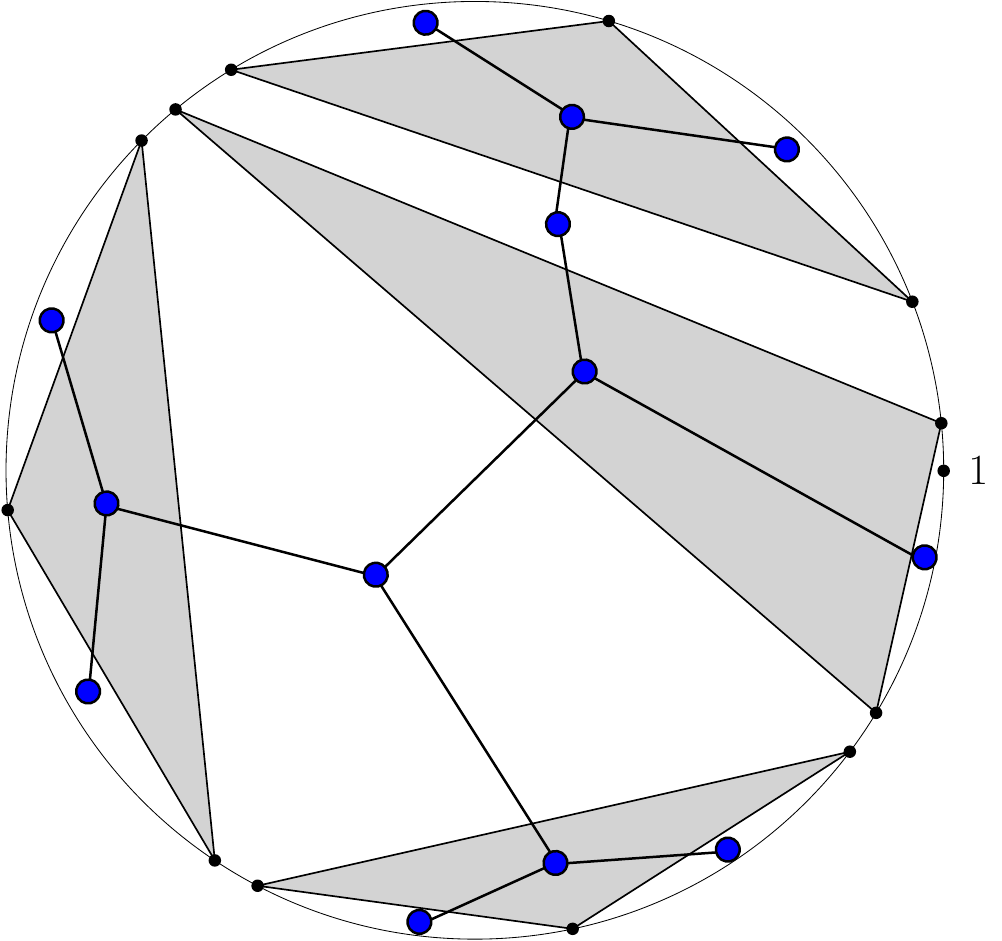}
    \end{minipage}
    \caption{On the left, a lamination and the corresponding rooted dual tree. Distances in the tree correspond to the number of chords separating the fragments in the lamination. 
    On the right, an example of $3$-angulation, together with its dual tree. The tree is rooted at 
    the node containing the point $(1,0) \in \partial \sD$. The shaded portions 
    correspond to the triangles inserted, while the white portions are essential fragments, i.e. the regions of the disk 
    with a positive Lebesgue measure on the circle.}
    \label{fig:lam-dual}\label{fig:k-angulation}
\end{figure}
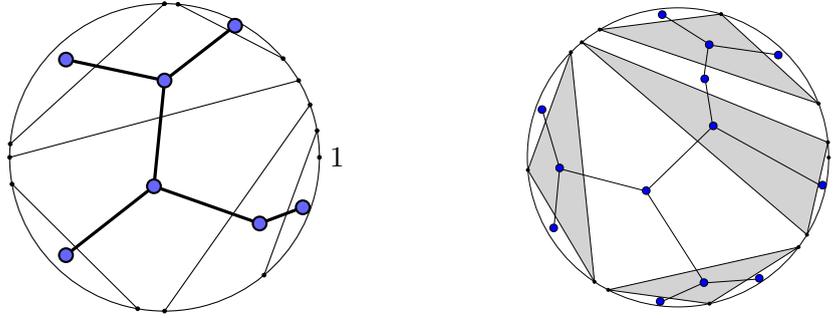

The processes $\sZ$ and $\sH$
arise in the problem of random recursive decompositions of the 
disk by non-crossing chords \cite{legalcu,BrSu2013a}. They encode the trees that are the 
planar dual of the limit triangulation in the same sense that the Brownian CRT is the dual of the limit uniform triangulation of the disk studied by Aldous 
\cite{al94circle,al94}. We now proceed to the precise definitions.  

The unit disk $\sD:=\{x\in \R^2: \|x\|\le 1\}$ is decomposed at discrete time steps as follows: 
At time $n=1$, a chord is inserted connecting two uniformly chosen points on the boundary $\partial\sD$. 
Then, given the configuration at time $n$, at time $n+1$ : pick two independent points on the circle 
$\partial\sD$ uniformly at random; add the chord connecting them if it does not intersect any previously 
inserted chord, otherwise reject the points and continue. This procedure yields an increasing 
sequence of non-intersecting chords $(L_n)_{n\ge 1}$, also called a lamination. 
For each $n\ge 1$, $\sD\setminus L_n$ 
consists of a finite number of connected components, and by $T_n$, we denote the discrete tree which 
is planar dual to the decomposition (nodes correspond to connected components, and two nodes 
are adjacent if the corresponding connected components share a chord). The tree $T_n$ is 
rooted at the node corresponding to a fragment containing a fixed point on the circle, say $(1,0)$. 
(See Figure~\ref{fig:lam-dual}). It has been proved in \cite{BrSu2013a} that $T_n$ suitably 
rescaled converges almost surely towards a limit tree encoded by a certain random process
which satisfies a fixed point equation of type \eqref{fix:X}. More precisely, for 
$\beta:=(\sqrt {17}-3)/2$, with respect to the Gromov--Hausdorff distance and as $n \to \infty$, we have 
\begin{equation}\label{eq:conv_TZ}
(T_n, n^{-\beta/2} d_n) \to (\sT_\sZ, d_\sZ),
\end{equation}
for the unique random excursion $\sZ$ satisfying \eqref{eqZ} with $\E{\sZ(\xi)} = \kappa > 0$, where $\kappa$ denotes a scaling constant whose value is irrelevant 
in the present context. (It is given in Theorem 3 in \cite{BrSu2013a}.)

Corollary~\ref{coro1} {\em (b)}, Corollary~\ref{cor:uni} {\em (b)} and Corollary~\ref{cor:main} {\em (b)} follow from  \eqref{eqZ}, Theorems \ref{thm:treeuniq}, \ref{thm:main2} and \ref{thm:main3} since $\xoverline \cR$ has the uniform distribution on $[0,1]$.

For any $\gamma < \beta$, by  Corollary \ref{cor:opt_Hol}, there exists a process equivalent 
to $\sZ$ with $\gamma$-H{\"o}lder continuous paths. Moreover, for $\gamma > \beta$, no equivalent 
process can be $\gamma$-H{\"o}lder continuous with positive probability. Indeed, by Theorem~1.1 
in \cite{legalcu}, the process $\sZ$ itself has $\gamma$-H{\"o}lder continuous paths for any 
$\gamma < \beta$ and is therefore optimal with respect to regularity. $\sZ$ 
is a good encoding of the real tree $\cT_\sZ$ since its fractal dimension corresponds precisely to what  
should be expected from the regularity of $\sZ$. (The fact that the rescaling of $T_n$ is $n^{-\beta/2}$ 
in \eqref{eq:conv_TZ} rather than $n^{-\beta}$ is reminiscent of the number of chords in $L_n$, which is only of order $\sqrt n$, 
so $T_n$ has only order $\sqrt{n}$ nodes, see \cite{legalcu,BrSu2013a}.)

\subsection{Homogeneous recursive triangulation of the disk} \label{sec:H}

\begin{figure}
\centering
\includegraphics[scale=.40]{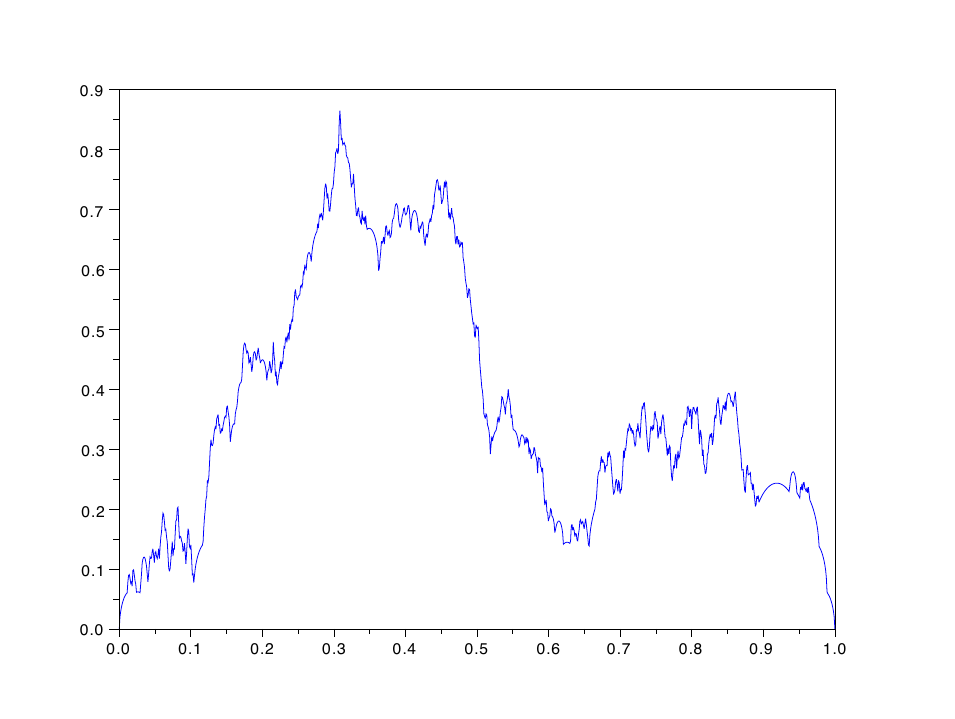}    
\hfil            
\includegraphics[scale=.40]{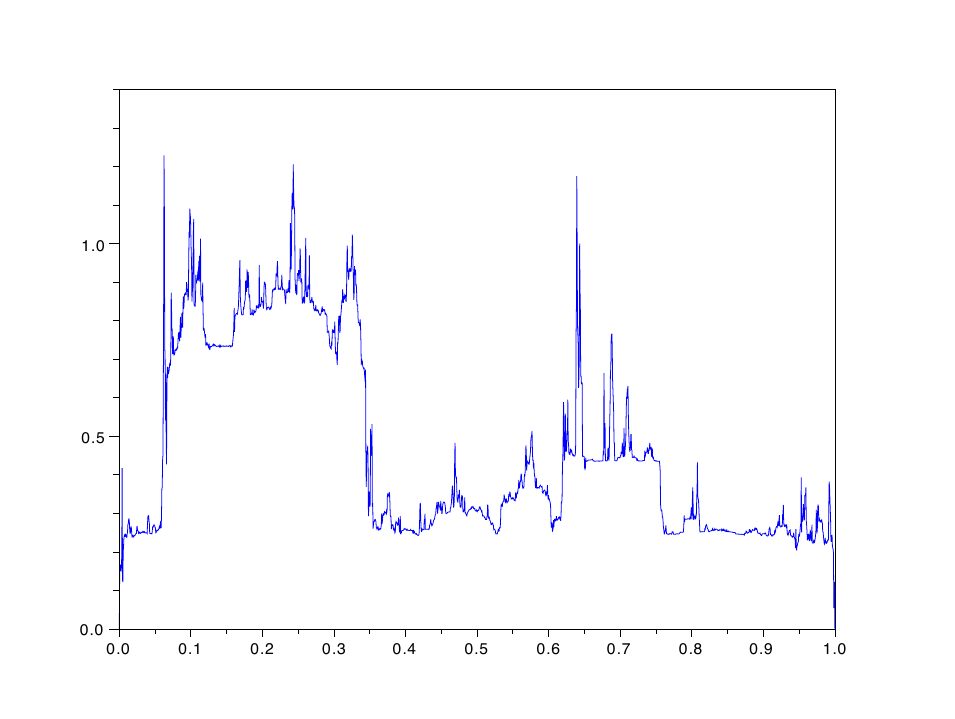}
\caption{\label{fig:Z} A comparison of the processes $\sZ$ and $\sH$. 
On the left, the process $\sZ$ encoding the dual tree of the self-similar 
recursive triangulation of the disk. 
On the right, the process $\sH$ encoding the dual tree of the homogeneous recursive 
triangulation of the disk. (The scales are not given since they are irrelevant.)} 
\end{figure}

We have a completely different situation if we consider a partition of the disk 
$\sD$ using random chords, but this time, the chords are 
inserted using a different strategy that is homogeneous: in each step, given the current configuration, 
one connected component is chosen \emph{uniformly at random} and split by the insertion 
of a chord linking two uniformly random points on the boundary conditioned on splitting the chosen 
component. Now, there is no rejection, and at time $n$ we have a collection of chords $L^h_n$ consisting
of $n$ elements. As before, we can define a tree that is dual to the lamination, and we denote it 
by $T_n^h$ (the discrete tree $T_n^h$ has $n+1$ nodes). It has been proved in \cite{BrSu2013a} that 
a suitably rescaled version of $T_n^h$ converges: in distribution with respect to the 
Gromov--Hausdorff distance and as $n \to \infty$, we have 
\begin{equation}\label{eq:conv_TH}
(T_n^h, n^{-1/3} d_n) \to (\sT_{\sH}, d_{\sH})\,,
\end{equation}
where $\sH$ is the unique random excursion satisfying \eqref{eqH} and $\Ec{\sH(\xi)}= 1/ \Gamma(4/3)$. 
(No characterization of $\sH$ was given in \cite{BrSu2013a}; one is
given in Corollary~\ref{cor:uni} {\em (c)}.) 
The rescaling $n^{-1/3}$ in \eqref{eq:conv_TH} suggests that the 
limit tree $\cT_{\sH}$ should have fractal dimension $3$. However, a first natural grasp that one has 
on the tree $\cT_{\sH}$ is the encoding excursion $\sH$, but a quick look at Figure~\ref{fig:Z} 
suggests that some trouble is around the corner since $\sH$ does not look 
H\"older with exponent $1/3-\varepsilon$ for $\varepsilon>0$ arbitrary. 

It is precisely in this kind of situation that our general framework is most useful, since 
it permits to verify that $\cT_{\sH}$ indeed has fractal dimension $3$, and 
more precisely that $\DimM(T_{\sH})=\DimH(T_{\sH})=3$ with probability one. This is reminiscent of 
the fact that, for any $\gamma < 1/3$, there exists excursions equivalent to $\sH$ that have
$\gamma$-H{\"o}lder continuous paths. As expected, unlike the process $\sZ$, $\sH$ is suboptimal with 
respect to path regularity: the following proposition is given for the sake of completeness, 
and its proof can be found in the supplementary material.

\begin{prop}\label{prop:opt}
Let $\varrho = 1-\frac 2 3 \sqrt{2} = 0.057\ldots$ Then, almost surely, 
\[\sup\{\gamma>0: \sH \text{ is }\gamma\text{-H\"older continuous}\}= \varrho.\]
\end{prop}

Corollaries~\ref{coro1} (c),~\ref{cor:uni} (c) and~\ref{cor:main} (c) follow as in the recursive case discussed in the previous section.

\subsection{Recursive $k$-angulations} \label{sec:kang}

In this section we consider a generalization of the lamination process described in 
Sections \ref{sec:Z} and \ref{sec:H}, where, for some fixed $k \geq 2$, in each step, 
one adds the $k$-gon connecting $k$ points sampled on the circle (for a precise 
definition, see below). Certain quantities in this model were studied by
Curien and Peres \cite{curper}. Again, we are interested in non-intersecting 
structures and investigate  both the recursive and the homogeneous model. 
Of course, for $k = 2$, we recover the processes studied in Sections \ref{sec:Z} 
and \ref{sec:H}.
The techniques in \cite{BrSu2013a} and \cite{legalcu} yield detailed information 
on the height processes of the corresponding dual trees and their limits. 
For example, in \cite{BrSu2013a}, we gave explicit expressions for the leading 
constants and rates of convergence for the mean functions. 
Furthermore, the limit mean function had already been obtained in \cite{legalcu} 
(up to a multiplicative constant). Most of these results do not play a 
significant role in proving convergence of the dual trees or determining the 
fractal dimensions of the limiting objects. (The leading constants in 
Propositions \ref{prop:kgon1} and \ref{prop:kgon2}  could be given 
by lengthy implicit formulas but they are of no particular relevance.)

\paragraph{The recursive $k$-angulation}
In the recursive framework, in each step, we choose $k$ points uniformly at random 
on the circle and insert the corresponding $k$-gon if none of its edges intersects 
any previously inserted one. The dual tree $\sT_n$ is defined analogously to 
the case $k=2$ upon identifying  fragments in the decomposition with nodes in $\sT_n$. 
It is endowed with the graph distance $d_n$. The \emph{mass} of a fragment in the 
decomposition of the disk is the one-dimensional Lebesgue measure of its intersection 
with the circle. Fragments with positive mass will subsequently be called 
\emph{essential} (all fragments are essential for $k=2$.) 
See Figure~\ref{fig:k-angulation} for an illustration. Keeping the notation 
introduced in \cite{BrSu2013a}, we denote by $\mathcal C_n(s)$ the depth of the 
node associated to the fragment covering the point $s$ in the tree $\cT_n$. Here, and subsequently, we identify the unit interval with $\partial \sD$ through $s \mapsto (\cos 2 \pi s, \sin 2 \pi s).$  (We do not 
indicate $k$ in the notation for the height functions.) Denote the first 
inserted $k$ points in increasing order by $0 \leq U_1 \leq \ldots \leq U_k \leq 1$ and define
$\Delta_1 = 1 - U_k + U_1$, $\Delta_i = U_i - U_{i-1}, 2 \leq i \leq k$, 
as well as $\xi^* = U_1 / \Delta_1$. Furthermore, let $I_n = (I_n^{(1)}, \ldots, I_n^{(k)})$, 
where $I_n^{(i)}$ is the number of attempted insertions up to time $n$ in the  fragment 
of mass $\Delta_i$. Given $(U_1, \ldots, U_k)$, for any $1 \leq i \leq k$, 
the random variable $I_n^{(i)}$ has the binomial distribution with parameters $n-1$ 
and $\Delta_i^k$. In particular, we have, almost surely, 
\begin{align}\label{asI}
\frac{I_n}{n} \to (\Delta_1^k, \ldots, \Delta_k^k). 
\end{align}

\begin{prop} \label{prop:kgon1} 
Let $k \geq 2$ and $N_n$ be the number of inserted $k$-gons at time $n$. Then, as $n \to \infty$, we have $n^{-1/k} N_n \to c_k$ in probability and with respect to all moments, where $c_k >0$ is a constant. 
\end{prop}

\begin{proof}
Let $\tau_1, \tau_2, \ldots$ be the times of homogeneous a Poisson point process  with 
unit intensity on $\R^+$. We consider the continuous-time analogue of $N_n, n \in \N$ 
denoted by $\mathcal N_t, t > 0$ where, for all times $\tau_i, i \geq 1$, a set of $k$ 
independent points are drawn at random on the circle and the corresponding $k$-gon inserted 
if the decomposition remains non-crossing. In other words, $\mathcal N_{t} = N_i$ 
for $t \in [\tau_i, \tau_{i+1})$ where $\tau_0 := 0$. It is easy to see and explained 
in detail in \cite{legalcu} for $k = 2$, that this process can alternatively be obtained 
without  the necessity of rejecting any $k$-gons as follows: starting with the disk 
at time $t = 0$, add a $k$-gon chosen uniformly at random after an exponentially 
distributed time with mean one. Then, independently on the $k$ essential sub-fragments, 
run the same process with times slowed down by a factor $x^k$ where $x$ denotes 
the mass of the fragment. 
The masses of essential fragments at time $t > 0$ in this process constitute a 
conservative fragmentation process with index of self-similarity $k$ and reproduction 
law Dirichlet$(1, \ldots, 1)$. 
Hence, by Theorem~1 in \cite{BeGn2004}, we deduce $t^{-1/k} \mathcal N_t \to c_k$ 
in probability and in $L_2$, with $c_k$ as in the proposition
In particular, $\tau_n^{-1/k} \mathcal N_{\tau_n} \to c_k$ in probability 
as $n \to \infty$. In order to  obtain moment  convergence, note that, 
for any $\varepsilon > 0$, by monotonicity and since $N_n \leq n$ almost surely, 
$$\tau_n^{-2/k} \mathcal N_{\tau_n}^2 
\leq ((1-\varepsilon)n)^{-2/k} \mathcal N^2_{(1+\varepsilon)n} 
+ \tau_n^{2-2/k} \I{|\tau_n - n| \notin (-\varepsilon n, \varepsilon n)}.$$
By the $L_2$ convergence for the continuous-time process and the concentration 
of $\tau_n$ (which has a Gamma$(n)$ distribution), the right hand side is uniformly 
integrable. Hence, $\tau_n^{-1/k} \mathcal N_{\tau_n} \to c_k$ in $L_2$.
Since $\tau_n / n \to 1$ almost surely and with convergence of all moments 
and $\tau_n,  \mathcal N_{\tau_n}$ are independent, we obtain the convergence 
in probability and in $L_2$. Finally, let $\tilde N_n = n^{-1/k} N_n$. Then, 
$$ \tilde N_n 
\stackrel{d}{=} 
1 +  \sum_{i=1}^k \left(\frac{I_n^{(i)}}{n}\right)^{1/k} \tilde N^{(i)}_{I_n^{(i)}} ,$$
where $(\tilde N^{(1)}_n)_{n\ge 1}, \ldots, (\tilde N^{(k)}_n)_{n\ge 1}$ are independent 
copies of $(N_n)_{n\ge 1}$, independent of $I_n$. Using \eqref{asI}, it is easy 
to prove that $\tilde N_n$ is bounded in $L_m, m \geq 1$ by induction over $m$ 
since we have already shown it for $m=1, 2$.
\end{proof}

By construction, the random process 
$(\cC_n(s))_{s \in [0,1]}$ satisfies the following recurrence in distribution on the space of c{\`a}dl{\`a}g functions~ 
endowed with the Skorokhod $J_1$-topology:
\begin{align} \label{rec:Cn}
 \cC_n(\,\cdot\,) \stackrel{d}{=}\, & 
\mathbf{1}_{[0,U_1)}(\,\cdot\,) \cC^{(1)}_{I_n^{(1)}}\left(\frac{\,\cdot\,}{\Delta_1}\right) 
+ \mathbf{1}_{[U_k,1]}(\,\cdot\,) \cC^{(1)}_{I_n^{(1)}}\left(\frac{\,\cdot\,-U_k}{\Delta_1}\right) \notag\\
&+ \sum_{i=2}^{k-1} \mathbf{1}_{[U_{i-1},U_i]} (\,\cdot\,) \left(\cC^{(i)}_{I_n^{(i)}}\left(\frac{\,\cdot\,-U_{i-1}}{\Delta_i}\right) + \cC^{(1)}_{I_n^{(1)}}\left(\xi^*\right)\right).
\end{align}
Here $(\cC^{(1)}_i(\cdot))_{i\ge 0}, \ldots, (\cC^{(k)}_i(\cdot))_{i\ge 0}$ are independent 
copies of $(\cC_i(\cdot))_{i\ge 0}$ independent of $(U_1, \ldots, U_k, I_n)$.
The first step of our analysis is to investigate $ \cC_n(\xi)$ for uniformly 
chosen point $\xi$. In \cite{BrSu2013a}, an explicit expression for the mean 
was obtained by solving the underlying recursion. Here, we proceed as in \cite{legalcu} 
relying on results from fragmentation theory. Subsequently, let $\alpha_k \in (0,1)$ 
be the unique solution to 
\begin{align*} 
g_k(x) 
& := \E{\Delta_1^{x}} + (k-1) \E{\mathbf{1}_{[U_1, U_2]}(\xi) \Delta_2^{x}}  \notag\\
& = \frac{k!\Gamma(x+2)}{\Gamma(k+x+1)}  + \frac{(k-1) k!\Gamma(x+2)}{\Gamma(k+x+2)}  =  1. 
\end{align*}
(Note that $g_k(x)$ decreases in $x$, $g_k(0) > 1$ and $g_k(1)  < 1$. Thus, $\alpha_k$ exists.)

\begin{prop} \label{prop:kgon2}
Let $k \geq 2$.   
As $n \to \infty$, in probability and with convergence of all moments,  we have
$n^{-\alpha_k/k} \mathcal C_n(\xi) \to X_k$ for some random variable $X_k$ with mean 
$\kappa_k := \Ec{X_k} > 0$. 
\end{prop}

\begin{proof}
We use the same continuous-time model as in the previous proof. 
Let $\mathcal C_t(\xi)$ be the height of the node associated to $\xi$ in the dual tree at time $t$ and $E_t(\xi)$ be the number of essential fragments associated to nodes on the path from $0$ to $\xi$. For $k \geq 3$, we have $\mathcal C_t(\xi) = 2 (E_t(\xi) - 1)$. As explained in \cite{legalcu} in the case $k = 2$, the sizes of essential fragments form a non-conservative fragmentation process with index of self-similarity $k$ and reproduction law $\Law((\Delta_1, \mathbf{1}_{[U_1, U_2)}(\xi) \Delta_2, \ldots, \mathbf{1}_{[U_{k-1}, U_k)}(\xi) \Delta_k))$. Hence, by \cite[Theorem 1]{BeGn2004}, as $t \to \infty$, $t^{-\alpha_k/k} \E{E_t(\xi)} \to \kappa_k/2$ for some $\kappa_k >0$. Furthermore, by  \cite[Theorem 5]{BeGn2004}, there exists a random variable $X'_k$ such that, $t^{-\alpha_k/k} E_t(\xi) \to X_k'$ in $L_2$. With $X_k = 2 X_k'$, the claim follows by standard depoissonization arguments as in the previous proof.
\end{proof}

Let $\cY_n(s) = \mathcal C_n(s)/\E{\mathcal C_n(\xi)}$. We expect that, as $n \to \infty$, 
we have $n^{-\alpha_k/k} \E{\mathcal C_n(s)} \to m_k(s) $ for some continuous 
excursion $m_k \in \Cex$ with $\E{m_k(\xi)} = \kappa_k$. Thus, from \eqref{fix:kgon}, 
it follows that, if $\cY_n(s) \to \sZ(s)$
for some continuous process $\sZ$, 
then the limit should have mean function $m_k / \kappa_k$ and satisfy 
\begin{align} \label{fix:kgon}
\sZ(\,\cdot\,) 
\stackrel{d}{=} \, & \mathbf{1}_{[0,U_1)}(\,\cdot\,) \Delta_1^{ \alpha_k} \sZ^{(1)}\left(\frac{\,\cdot\,}{\Delta_1}\right) 
+ \mathbf{1}_{[U_k,1]}(\,\cdot\,) \Delta_1^{ \alpha_k} \sZ^{(1)}\left(\frac{\,\cdot\,-U_k}{\Delta_1}\right) \notag \\
&+ \sum_{i=2}^{k} \mathbf{1}_{[U_{i-1},U_i]} (\,\cdot\,)\left( \Delta_i^{ \alpha_k} \sZ^{(i)}\left(\frac{\,\cdot\,-U_{i-1}}{\Delta_i}\right) + \Delta_1^{ \alpha_k} \sZ^{(1)}\left(\xi^*\right)\right),
\end{align}
where $\sZ^{(1)}(\cdot), \ldots, \sZ^{(k)}(\cdot)$ are independent copies of $\sZ(\cdot)$ independent of $(U_1, \ldots, U_k)$. This fixed point equation is of type \eqref{fix:X} where 
$ 
K = k, L = k+1, 
\varpi_3 = \ldots = \varpi_k = 1, 
\mathcal R = \mathcal S \sim \text{Dirichlet}(2, 1, \ldots, 1). 
$  
Let $\sZ$ be the unique process (in distribution) solving \eqref{fix:kgon} with $\E{\sZ(\xi)} = \kappa_k$ whose existence is guaranteed by Theorem \ref{thm:treeuniq}. (We use the same notation for the limit process as in Section \ref{sec:Z} without indicating the choice of $k$.) The verification of $\mathcal Y_n \to \sZ$ in distribution in the space of c{\`a}dl{\`a}g functions can be worked out by the same arguments as in \cite[Section 3]{BrSu2013a} relying on the contraction method both for real-valued random variables and for regular processes. Here, starting with an independent family $(U^{(i)}_1, \ldots, U^{(i)}_k), i \geq 1$ of copies of $(U_1, \ldots, U_k)$ one constructs the sequence $\mathcal Y_n$ and its limit $\sZ$ satisfying \eqref{fix:kgon} on the same probability space and shows convergence in probability. The steps are very similarly to the arguments in the proof of Proposition \ref{prop:conlimit}. First, one shows the convergence $\mathcal Y_n(\psi) \to \sZ(\psi)$ in $L_2$, where $\psi$ corresponds to the point $\xi^*$ in the coupling. From the last proposition we know that this convergence also holds in $L_m$ for all $m \geq 1$. Finally, one shows that $\E{\| \mathcal Y_n - \sZ \|^m} \to 0$ for all $m \geq 1$. (The almost sure convergence in \cite{BrSu2013a} relies on a convergence rate for the mean of $\mathcal Y_n(\xi)$. We do not pursue this line here but note that, sufficient rates in the continuous-time case can be extracted from \cite{BeGn2004}, compare the discussion of Theorem 1 there.)
Summarizing, we obtain the following result.

\begin{thm}
Let $k \geq 3$. In distribution with respect to the Gromov--Hausdorff topology, we have  $$(\sT_{n}, n^{-\alpha_k / k} d_n) \to (\sT_{\sZ}, d_{\sZ}).$$
In distribution, the process $\sZ$ is the unique continuous excursion solving \eqref{fix:kgon} up to a multiplicative constant.
Almost surely, $\DimM(\sT_\sZ) = \DimH(\sT_{\sZ}) = \alpha_k^{-1}$. For any $\gamma < \alpha_k$, almost surely, there exists a process $\tilde {\sZ}$ equivalent to $\sZ$ which has $\gamma$-H{\"o}lder continuous paths.
\end{thm}

\noi \textsc{Remark}. {\bf 1)} For $k = 2$, almost surely, the process $\sZ$ has $\gamma$-H{\"o}lder continuous path 
for any $\gamma < \alpha_2 = \beta$ \cite[Theorem 1.1]{legalcu}. We think that this remains true for all $k \geq 3$, that is, 
the function $\sZ$ is a good encoding of the tree. However,  we do not pursue this here.

\noi {\bf 2)} Let $L_n$ be the set of chords inserted at time $n$. By Proposition~\ref{prop:degrees}, 
$\sD(\cT_{\sZ}) = \{1, 2, k\}$, and it is not hard to see that $\overline{\bigcup_{n \geq 1} L_n}$ is indeed a 
$k$-angulation of the disk: every connected component in its complement is a convex $k$-gon with vertices on the circle. 

\medskip For $k \geq 3$, we have no explicit expression for $m_k =  \E{\sZ(t)}$. It follows from  \eqref{fix:kgon} that
$m_k$ is the unique continuous excursion on $[0,1]$ with $\E{m_k(\xi)} = \kappa_k$ such that \sloppy 
$m_k(t) = \E{\Phi(m_k, \ldots, m_k, \Delta, \Delta, \Xi)(t)}, t \in [0,1]$ where $K = k, L = k+1, 
\varpi_3 = \ldots = \varpi_k = 1$ and $
\Delta\sim \text{Dirichlet}(2, 1, \ldots, 1).$

  Using this observation and some geometric arguments relying directly on the construction of the process, one can show that $m_k$ is infinitely differentiable on $(0,1)$, symmetric at $t = 1/2$ and monotonically increasing and concave on $[0,1/2]$. Since we do not use these observation, we omit the details and the proofs.

\paragraph{The homogeneous $k$-angulation}

In the homogeneous setting, in each step, one \emph{essential} fragment is chosen uniformly at random and $k$ uniformly chosen points selected to create a new $k$-gon. At time $n$ this leads to a decomposition of the disk into
$1 + (k-1) n$ essential fragments and $n$ non-essential fragments. The definitions of $\sT_n^h, \mathcal C_n^h, U_1, \ldots, U_k, \Delta_1 = 1 - U_k + U_1, \Delta_i = U_i - U_{i-1}, 2 \leq i \leq K, \xi^* = U_1 / \Delta_1$ as well as 
$I_n = (I_n^{(1)}, \ldots, I_n^{(k)})$ should be clear by now. By construction, the random variable $I_n$ is independent of $(U_1, \ldots, U_k)$ and grows like the vector of occupation numbers in a Polya urn model. It is well-known that, almost surely, $I_n/n \to (\tilde \Delta_1, \ldots, \tilde \Delta_k),$
where $\tilde \Delta = (\tilde \Delta_1, \ldots, \tilde \Delta_k)\sim \text{Dirichlet}(1/(k-1), \ldots, 1/(k-1))$. By construction, the random process 
$(\cC^h_n(s))_{s \in [0,1]}$ satisfies a recursion analogous to \eqref{rec:Cn}, the only difference being the distribution of $(U_1, \ldots, U_k, I_n)$. 

\begin{prop} \label{prop:kgonhom}
Let $k \geq 2$. As $n \to \infty$, almost surely and with convergence of all moments, we have 
$n^{-1/(k+1)} \mathcal C_n^h(\xi) \to X^h_k$ for some random variable $X^h_k$ with mean $\kappa^h_k := \Ec{X^h_k} > 0.$ 
\end{prop} 

\begin{proof}
In the standard continuous-time embedding of the process, every essential fragment splits into $k$ essential subfragments at unit rate, independently of its mass. Hence, the number of essential fragments $\mathcal N_t, t \geq 0$, forms a continuous-time branching process with offspring distribution 
$\delta_k$.   
It is well-known that $e^{-t(k-1)} \mathcal N_t, t > 0,$ is a uniformly integrable martingale with mean one converging almost surely to a limiting random variable denoted by $\mathcal N$ having the Gamma$((k-1)^{-1}, (k-1)^{-1})$ distribution. For the time of $n$-th split $\tau_n$ in the process, since $\mathcal N_{\tau_n} =  1 + (k-1)n$, by the optional stopping theorem, it follows that $ (k-1)n e^{-\tau_n(k-1)} \to \mathcal N$ almost surely and in mean.
Similarly,  the number of essential fragments on the path from $0$ to $\xi$ in the dual tree denoted by $E^h_t(\xi)$ forms a branching process with offspring distribution
$\Law(1 + \mathbf{1}_{[U_1, U_k)}(\xi))$. Again, the process $e^{-t(k-1)/(k+1)}  E^h_t(\xi), t > 0,$ is a uniformly-integrable martingale with unit mean and we denote its almost sure limit by $\mathcal E$. Writing 
\begin{align*} & e^{-\tau_n(k-1)/(k+1)}  E^h_{\tau_n}(\xi) \\ & = e^{-\tau_n(k-1)/(k+1)}  ((k-1)n)^{1/(k+1)} \cdot ((k-1)n)^{-1/(k+1)} E^h_{\tau_n}(\xi),\end{align*} and noting that the random variables $\tau_n, E^h_{\tau_n}(\xi)$ are independent, it follows that $((k-1)n)^{1/(k+1)} E^h_{\tau_n}(\xi) \to \mathcal E'$ a.s.\ and in mean where $\mathcal E = \mathcal E' \mathcal N^{1/(k+1)}$. 
The claimed convergence follows from the identity $\cC^h_n(\xi)=2(E_{\tau_n}^h(\xi)-1)$.
Convergence of moments can be deduced as in the recursive model. 
\end{proof}

For $s\in [0,1]$, let $\mathcal Y^h_n(s) = \mathcal C^h_n(s)/\Ec{\mathcal C^h_n(\xi)}$. 
A limit  $\sH(s)$ of $\mathcal Y^h_n(s)$  should satisfy $\Ec{\sH(\xi)} = \kappa_k^h$ and 
\begin{align} \label{fix:kgon2}
& \sH \stackrel{d}{=} \, 
\Delta_1^{ 1/(k+1)} \left( \mathbf{1}_{[0,U_1)}(\,\cdot\,) \tilde  \sZ^{h,(1)}\left(\frac{\,\cdot\,}{\Delta_1}\right) 
+ \mathbf{1}_{[U_k,1]}(\,\cdot\,) \tilde \sZ^{h, (1)}\left(\frac{\cdot-U_k}{\Delta_1}\right)  \right) \\
&+ \sum_{i=2}^{k} \mathbf{1}_{[U_{i-1},U_i]} (\,\cdot\,)\left( \tilde \Delta_i^{ 1/(k+1)} \sZ^{h, (i)}\left(\frac{\,\cdot\,-U_{i-1}}{\Delta_i}\right) + \tilde \Delta_1^{ 1/(k+1)} \sZ^{h, (1)}\left(\xi^*\right)\right),\notag 
\end{align}
where $\sZ^{h, (1)}, \ldots, \sZ^{h, (k)}$ are independent copies of $\sH$ independent of $(U_1, \ldots, U_k), \tilde \Delta$. This fixed point equation is of type \eqref{fix:X} where 
$K=k,   
L = k+1, 
\varpi_3 = \ldots = \varpi_k = 1, 
\mathcal R\sim \text{Dirichlet}(1, 1, \ldots, 1)$ and $
\mathcal S\sim\text{Dirichlet}(2, 1, \ldots, 1),$ 
and  $\mathcal R$ and $\mathcal S$ are independent. (We use the same notation for the limit process as in Section \ref{sec:H}.) As in the recursive model, one can prove the following theorem.

\begin{thm}
Let $k \geq 3$. In distribution with respect to the Gromov--Hausdorff topology, we have 
$$(\sT^h_{n}, n^{-1/(k+1)} d^h_{n}) \to (\sT_{\sH}, d_{\sH}).$$
In distribution, the process $\sH$ is the unique continuous excursion satisfying \eqref{fix:kgon2} up to a multiplicative constant.
Almost surely, $\DimM(\sT_{\sH}) = \DimH(\sT_{\sH}) = k+1$. For any $\gamma < 1/(k+1)$, almost surely, there exists a process $\tilde {\sH}$ equivalent to $\sH$ which has $\gamma$-H{\"o}lder continuous paths. 
\end{thm}

{\small 
\setlength{\bibsep}{0.25em}
\bibliographystyle{plainnat}
\bibliography{arxiv_2020}
}

\appendix


\section{Continuity and measurability statements} 

\begin{lem}\label{lem:iota}
The set $\bbK^{\GHP}_\tf$ is a measurable subset of $\bbK^{\GHP}$ and the bijection
$\iota$ between $\bbK^{\GHP}_\tf$ and $\bbK^{\GP}$ is bimeasurable. 
\end{lem}
\begin{proof}
For any $\delta > 0$, the quantity $\kappa_\delta(\fX) = \inf \{ \mu(\{y  \in \cX : d(x,y) 
\leq \delta\}) : x \in \cX \},$ only depends on the $\GHP$-equivalence class of $\fX$. 
Moreover, $\bbK^{\GHP}_\tf = \{ \fX \in \bbK^{\GHP} : \kappa_\delta(\fX) > 0 \text{ for all }\ 
\delta > 0 \}.$ A straightforward application of the Portemanteau lemma shows 
that $\kappa_\delta$ is upper semi-continuous with respect to $\dghp$. (The details are 
given in the proof of Lemma 3.2 in \cite{ALW}.) It follows easily that $\bbK^{\GHP}_\tf$ is 
a measurable set.
The map $\iota$ is continuous since convergence with respect to the Gromov--Hausdorff--Prokhorov 
topology implies convergence with respect to the Gromov--Prokhorov topology.
Its inverse $\iota^{-1}$ is not continuous, and the space $\bbK_\tf^\GHP$ 
(considered as subspace of $\bbK^\GHP$) is not complete. 
Nevertheless, it is important to note that, by \cite[Corollary~5.6]{ALW},
$\bbK_\tf^\GHP$ endowed with the relative topology generated by $\dghp$ is Polish. 
It follows from the Lusin--Souslin theorem, see, e.g.\ \cite[Theorem 15.1]{kechris}, 
that $\iota^{-1}$ is measurable. (We thank Stephan Gufler for pointing out the 
short argument showing the measurability of $\iota^{-1}$.)
\end{proof}

\begin{proof}[Proof of Lemma~\ref{lem:extra}]
 Let $(\sT, d, \mu, \rho)$ be a compact rooted measured real tree. We borrow notation from \cite{duq1}. Recall that $\mathscr B$ denotes the set of branch points of $\sT$. Let $\sigma \in \mathscr B^* := \mathscr B \cup \{\rho\}$ and note that the degree of a point $\sigma \in \mathscr B^*$ is at most countably infinite. For $\sigma \in \mathscr B^*$, let 
$$\mathcal I_\sigma = \begin{cases} \emptyset & \text{if } \text{deg}(\sigma) = 1, \\
\{1, 2, \ldots, \text{deg}(\sigma)-1\} & \text{if } 1 < \text{deg}(\sigma)  < \infty, \\
\N \setminus \{0\} & \text{if } \text{deg}(\sigma) = \infty. \end{cases}$$
Further, let $\mathcal C_\sigma$ be the set of connected components in $\sT \setminus \{\sigma\}$ which do not contain $\rho$. The cardinality of $\mathcal C_\sigma$ is given by the cardinality of $\mathcal I_\sigma$. Fix a bijection $p_\sigma : \mathcal I_\sigma \to \mathcal C_\sigma$ (e.g. by considering the order in which the elements of $C_\sigma$ are visited by a fixed dense sequence in $\sT$).
Let $\mathcal D = \{(\sigma, k) : \sigma \in \mathscr B^*, k \in \mathcal I_\sigma\}$ and $\mathcal U$ be the set of elements $u \in [0,1]^{\mathcal D}$ such that $u_i \neq u_j$ for all $i \neq j$, $i,j\in \mathcal D$. 
For $\sigma, \sigma' \in \sT$, set $\sigma \leq \sigma'$ if $\sigma \in \llb \rho, \sigma'\rrb $ and $\sigma' \leq \sigma$ if $\sigma' \in \llb \rho, \sigma\rrb $. Otherwise, denoting by $\gamma = \sigma \wedge \sigma' \in \mathscr B^*$ the most recent common ancestor of $\sigma$ and $\sigma'$ (that is, the unique node satisfying 
$\llb \rho, \sigma \wedge \sigma' \rrb  = \llb \rho, \sigma \rrb \cap \llb \rho, \sigma' \rrb$),
 and assuming that $\sigma \neq \gamma \neq \sigma'$, there exist $1 \leq i \neq i'$ such that $\sigma \in p_\gamma(i)$ and $\sigma' \in p_\gamma(i')$. Set $\sigma \leq \sigma'$ if $u_{(\gamma,p_\gamma(i))} < u_{(\gamma, p_\gamma(i'))}$ and $\sigma' \leq \sigma$ otherwise. By Proposition 2.4 in \cite{duq1}, $\leq$ is a total order on $\sT$ satisfying the following two properties:
\begin{itemize}
\item [\textbf{(Or1)}] For $\sigma, \sigma' \in \sT$, $\sigma \in  \llb \rho, \sigma'\rrb $ implies $\sigma \leq \sigma'$.
\item [\textbf{(Or2)}] For $\sigma_1 \leq \sigma_2 \leq \sigma_3$ and $\gamma$ defined by $\llb \rho, \gamma \rrb = \llb \rho, \sigma_1 \rrb \cap (\llb \rho, \sigma_2 \rrb \cup \llb \rho, \sigma_3 \rrb),$
we have $\gamma \in \llb \rho, \sigma_2 \rrb$.
\end{itemize}
Note that, by Lemma 2.5 in \cite{duq1}, for $\gamma \in \mathscr B^*$ and two connected components $C, C' \in \mathcal C_\gamma$, we have either $\sigma \leq \sigma'$ for all $\sigma \in C, \sigma' \in C'$ or $\sigma' \leq \sigma$ for all $\sigma \in C, \sigma' \in C'$.

Next, let $\leq$ be a total order on $\sT$ satisfying \textbf{(Or1)} and \textbf{(Or2)}. We construct a function $h \in \Dex$ such that $\fT_h = \fT$ (with respect to isometry classes).
First, set $m(\sigma) = \mu(\{\sigma' \in \sT: \sigma' \leq \sigma\})$. By construction, $m : \sT \to [0,1]$ is  monotonically increasing with respect to the order $\leq$. As $\mu$ has full support and upon setting $\sigma < \sigma'$ if $\sigma \leq \sigma'$ and $\sigma \neq \sigma'$, $m$ is strictly increasing; in particular, 
$m(y) > 0$ for all $y \neq \rho$. 
For $\sigma \neq \rho$, we have $m(\sigma-) := m(\sigma) - \mu(\{\sigma\}) = \lim_{\sigma' \to \sigma, \sigma' < \sigma} m(\sigma')$. For
$\sigma \notin \mathscr B^*$ we also have $m(\sigma) = \lim_{\sigma' \to \sigma, \sigma' > \sigma} m(\sigma')$.
For $\mu(\{\rho\}) \leq x < 1$, let $m^{-1}(x) = \sup \{ \sigma \in \sT: m(\sigma) < x\}$ be the generalized inverse of $m$. For $x < \mu(\{\rho\})$, set $m^{-1}(x) = \rho$. Finally, let $m^{-1}(1) = \lim_{x \to 1} m^{-1}(x)$.
Then, $m^{-1}(x) = \lim_{y \to x, y< x} m^{-1}(y)$ and $m^{-1}(x+) =  \lim_{y \to x, y> x} m^{-1}(y)$. 
Set $h(x) = d(\rho, m^{-1}(x)), x \in [0,1]$. It immediately follows that $h \in \Dex$. Further, $\mathfrak T$ and $\fT_{h}$ are $\GHP$-isometric via $\sigma \mapsto \pi_h(m(\sigma))$, where $\pi_h$ is the canonical surjection from $[0,1]$ onto $\sT_h$. 
From the point of view of measurability, the important observation is that the map which assigns to $u \in \mathcal U$ the function $h$ constructed from the order induced by $u$ is continuous. This follows from the fact that, for any $\varepsilon > 0$, there exist only finitely many points $\sigma \in \mathscr B^*$ such that two distinct elements of $\mathcal C_\sigma$ have mass or diameter exceeding $\varepsilon$. We omit a formal proof. Let $\tilde \eta$ be the distribution of the random variable $h$ when choosing 
$u$ uniformly at random, that is, following the distribution $\Leb^{\otimes \mathcal D}$. Note that, formally, $\tilde \eta$ may depend on $\sT$ and the choice of the bijections $\pi_\sigma, \sigma \in \mathscr B^*$. Thanks to the symmetry of the distribution $\Leb^{\otimes \mathcal D}$ under permuting entries, however, it follows that $\tilde \eta$  only depends on the $\GHP$-isometry class of $\fT$. Therefore, we shall denote this distribution by $\tilde \eta(\fT)$. The corresponding map $\tilde \eta : \bbT^{\GHP}_\tf \to \mathcal M_1(\Dex)$ can be shown to be continuous, hence measurable. Now, for a random variable $\fT$ on $\bbT^{\GHP}_\tf$ with distribution $\nu$, the annealed measure $\eta(\cdot) = \Ec{\tilde \eta({\fT})(\cdot)}$ is the desired probability distribution.
\end{proof}

The following three lemmas concern continuity of functions arising in the definition of $\Psi$ 
in Section \ref{sec:rec_trees} and in the proof of Proposition~\ref{prop:counter_ex}. The proofs  rely on 
the concept of correspondences and have many ideas in common. In our presentation, we focus on a 
detailed proof of Lemma~\ref{lem:continuity_psi} and sketch the arguments needed to prove 
continuity of $\psi_1$ in Lemma~\ref{lem:psi1_cont}. The remaining statements have simpler
proofs, and we omit them.

\begin{lem} \label{lem:conti1}
Let $\mathbf r, \mathbf s \in \Sigma_K$ and $\fX_1, \ldots, \fX_K$ be compact rooted metric 
measured spaces. Let $\chi : \cX_1 \times \ldots  \times \cX_K \to \mathbb K^\GHP$ be the map 
which assigns to each $(x_1, \ldots, x_k) \in \cX_1 \times \ldots  \times \cX_K$ the space
$\mathfrak X$ when the construction described on page \pageref{list} is carried out with these 
values $\br,\bs$ and the glue points $\eta_i = x_i$ for $1\le i\le K$. Then $\chi$ is 
continuous. 
\end{lem}

\begin{lem}\label{lem:continuity_psi}
The map $\psi : (\mathbb K^\GHP)^K \times \Sigma_K^2 \to \cM_1(\mathbb K^\GHP)$ defined in 
Section~\ref{sec:rec_trees}  is continuous.
\end{lem}

\begin{lem}\label{lem:psi1_cont}
The functions $\chi_1: \cX^\N \times c_0^+ \to \bbK^{\GHP}$ and $\psi_1: \bbK^{\GHP} \to \cM_1(\bbK^{\GHP})$ defined in the proof of 
Proposition~\ref{prop:counter_ex}  are continuous.
\end{lem}

\begin{proof}[Proof of Lemma \ref{lem:continuity_psi}]
Let $\varepsilon > 0$ and $\fX_1, \ldots, \fX_K, \fX_1', \ldots, \fX_K'$ be compact rooted 
measured metric spaces. Let  $(Z_1, d^{Z_1})$, \dots, $(Z_k, d^{Z_k})$ be compact metric spaces 
such that $\fX_i, \fX_i' \subseteq Z_i$ and $d^{Z_i} = d_i$ on $\cX_i$ and $d^{Z_i} = d_i'$ on $\cX_i'$ for all $i=1, \ldots, K$. (Such spaces exist. One can, for instance, choose $Z_i = \cX_i \sqcup \cX_i'$.) Further, assume that,  
\begin{enumerate}
\item $\mathrm{d_{\text{\textsc{h}}}^{Z_i}}(\cX_i, \cX_i') <  \varepsilon$,
$d^{Z_i}(\rho_i, \rho_i') <  \varepsilon$, 
and $\mathrm{d_{\text{\textsc{p}}}^{Z_i}}(\mu_i, \mu_i') <  \varepsilon $ for all $i=1, \ldots, K$, 
\item $\mathbf r, \mathbf r', \mathbf s, \mathbf s' \in \Sigma_K$ with $\max\{\|\mathbf r- \mathbf r'\| , \|\mathbf s-\mathbf s'\|\} < \varepsilon$. 
\end{enumerate}
($\Sigma_K$ is endowed with the Euclidean distance.) Finally, assume that the constructions of the 
spaces $\fX = (\cX, d, \mu, \rho)$, and $\fX' = (\cX', d', \mu', \rho')$ described on page
\pageref{list}  are carried out with fixed glue points $\eta_i, \eta_i'$ satisfying
$d^{Z_i}(\eta_i, \eta_i') < \varepsilon$ for all $i=1, \ldots, K$. 
Below, we show that
\begin{align} \label{dghb} 
\dghp(\fX, \fX') \leq (K+1) \left(\varepsilon + \frac{\varepsilon^\alpha}{2} 
\max_{1\le i\le K}   \| \fX_i \|\right). 
\end{align}
Taking this deterministic statement for granted, we can argue as follows to conclude the proof:
First, we keep the compact rooted measured metric spaces fixed and choose the glue points 
randomly. Then, by the previous lemma, (the equivalence classes of) $\fX$ and $\fX'$ are
$\bbK^\GHP$-valued random variables. Further, by a coupling theorem due to Strassen
\cite[Page 438]{str65}, we may sample the pairs of gluepoints $(\eta_i, \eta_i')$ on
$\cX_i \times \cX_i'$ in such a way that $\Law(\eta_i) = \mu_i$, $\Law(\eta_i') = \mu_i'$ and
$\pc{d^{Z_i}(\eta_i, \eta_i') \geq \varepsilon} \leq \varepsilon$ for all $i=1, \ldots, K$. 
Using this coupling and writing $\gamma$ for the right-hand side of \eqref{dghb} yields
$\pc{\dghp(\fX, \fX') \geq \gamma } \leq K \varepsilon$. Hence, again using Strassen's 
theorem, it follows that $\mathrm{d_{\text{\textsc{p}}}}(\Law(\fX), \Law(\fX')) \leq \gamma$. 
Since $\gamma$ can be made arbitrarily small by choice of $\varepsilon>0$, this implies the 
claimed continuity. (The remainder of the proof also shows that, for 
fixed $\mathbf r, \mathbf s$,  the map $\psi$ with domain $(\mathbb K^\GHP)^K$ is 
uniformly continuous.)

It remains to show \eqref{dghb}. To this end, we recall the well-known characterization of 
the Gromov--Hausdorff distance using correspondences: for two sets $S, T$, a set
$\mathfrak R \subseteq S \times T$ is called a \emph{correspondence} if for
all $s \in S$  there exists $t \in T$ with $(s,t) \in \mathfrak R$ and vice versa. For metric spaces $(S,d_S), (T,d_T)$, the \emph{distorsion} of a 
correspondence $\mathfrak R$ is defined by 
\begin{align*} \text{dis}(\mathfrak R) = \sup \{|d_S(s,s') - d_T(t, t')| : (s, t), (s', t') \in \mathfrak R \}. \end{align*}
Define the Gromov--Hausdorff distance between compact rooted metric spaces $(S,d_S, \rho_S)$, and
$(T, d_T, \rho_T)$ using the notation from Eq. \eqref{def:gh} by 
$$\dgh((S,d_S, \rho_S), (T, d_T, \rho_T)) = \inf_{Z, \phi, \phi'} \max 
\left \{ \dha(\phi(S),\phi'(T)), d^Z(\phi(S), \phi'(T)) \right \}. $$
It is standard (\cite[see, e.g.,][Theorem 7.3.5]{BuBuIv2001} in the unrooted set-up) that, for 
compact rooted metric spaces $(S, d_S, \rho_S)$ and $(T, d_T, \rho_T)$, we have 
\begin{align} \nonumber
& \dgh((S, d_S, \rho_S), (T, d_T, \rho_T)) \\
& = \frac 1 2 \inf \{\text{dis}(\mathfrak R) 
: \mathfrak R \text{ is a correspondence with } (\rho_S, \rho_T) \in \mathfrak R\}. \label{epw}
\end{align}
It easy to construct optimal  correspondences explicitly: For $i=1, \ldots, K$, we can set
$$\mathfrak R_i = \{(x,y) : x \in \cX_i, y \in \cX_i', d^{Z_i}(x,y) < \varepsilon\}.$$
By the triangle inequality, we have $\text{dis}(\mathfrak R_i) \leq 2 \varepsilon $.
Recall the projections $\cp, (\cp)'$ in step \emph{ii)} of the construction and define a 
correspondence for $\cX$ and $\cX'$ by 
$$\mathfrak R = \bigcup_{i=1}^K \{(\cp(x), \cp(x')) : (x,x') \in \mathfrak R_i \}.$$
(Note that the roots are already identified in the $\mathfrak R_i$.)
From \eqref{epw}, we see that \sloppy $\dgh((\cX, d, \rho), (\cX', d',\rho')) \leq \frac 1 2 \text{dis}(\mathfrak R)$. In order to extend the result to the Gromov--Hausdorff--Prokhorov distance, we 
consider a specific embedding of the spaces. The following construction is standard, see, e.g.\ the proof of Theorem 7.3.5 in \cite{BuBuIv2001}.
Let $Z = \cX \sqcup \cX'$ and set $d^Z = d$ on $\cX^2$, $d^Z= d'$ on $\cX'\times \cX'$ while, for $x \in \cX, y \in \cX'$, define
$$d^Z(x,y) = \inf \{d(x,x_1) +d'(x_2, y) : (x_1, x_2) \in \mathfrak R\} + \frac 1 2 \text{dis}(\mathfrak R).$$
Finally, for $x \in \cX', y \in \cX$, set $d^Z(x,y) = d^Z(y,x)$. A straightforward computation shows that $d^Z$ is a metric on $Z$. 
Further, by construction, $\dha((\cX, d, \rho), (\cX', d',\rho')) = d^Z(\rho, \rho') = \frac 1 2 \text{dis}(\mathfrak R)$.
To study  the Prokhorov distance between $\mu$ and $\mu'$, let $(\sigma_i, \sigma_i'), i = 1, \ldots, K$ be pairs of random variables on $\cX_i \times \cX_i'$ with $\Law(\sigma_i) = \mu_i, \Law(\sigma_i') = \mu_i'$ and 
$\pc{|\sigma_i - \sigma_i'| \geq \varepsilon} \leq \varepsilon$. By our assumptions on $\mathbf s, \mathbf s'$, there exists a pair of random variables $(J, J')$ with
$\Prob{ J = i} = s_i$ and $\pc{J'=i} = s_i'$ for all $i=1, \ldots, K$ and $\pc{J\neq J'} \leq K \varepsilon$. Note that $\Law(\cp(\sigma_J)) = \mu$ and $\Law((\cp)'(\sigma'_{J'})) = \mu'$. Hence, 
as $d^{Z_i}(x,y) < \varepsilon$ for $x \in \cX_i, y \in \cX_i'$ implies $d^Z(\cp(x), (\cp)'(y)) \leq \frac 1 2 \text{dis}(\mathfrak R)$, it follows that $\pc{|\cp(\sigma_J) - (\cp)'(\sigma'_{J'})| > \frac 1 2 \text{dis}(\mathfrak R)} \leq (K+1)\varepsilon$. Thus, ${\mathrm{d_{\text{\textsc{p}}}}}(\mu, \mu') \leq \max \{\frac 1 2 \text{dis}(\mathfrak R), (K+1) \varepsilon\}$ and the same bound applies to $\dghp(\fX, \fX')$.

It remains to find an upper bound on $\text{dis}(\mathfrak R)$. We have 
\begin{align*} \text{dis}(\mathfrak R) = \sup_{1 \leq i,j \leq K} \sup \{& |d(\cp(x),\cp(y)) - d'((\cp)'(x'),(\cp)'(y'))| : \\
& (x,x') \in \mathfrak R_i, (y,y') \in \mathfrak R_j \}.\end{align*}
For $(x,x'), (y,y') \in \mathfrak R_i$, it follows that under our imposed assumptions, we have $$|d(\cp(x),\cp(y)) - d'(\cp(x'), \cp(y'))| \leq 2 \varepsilon + \varepsilon^\alpha \| \fX_i \|. $$
An application of the triangle inequality shows that 
$$\text{dis}(\mathfrak R) \leq K \Big(2 \varepsilon +  \varepsilon^\alpha  \max_{i=1, \ldots, K} \| \fX_i \|\Big),$$
which completes the proof.
\end{proof}

\begin{proof} [Proof of Lemma \ref{lem:psi1_cont}]
We only sketch the arguments necessary to prove of continuity of $\psi_1$ which go beyond the details presented in the previous proof.
Fix $\fX \in \bbK^{\GHP}$. Let $\varepsilon > 0$ and
$\fX' \in \bbK^{\GHP}$ with $\dghp(\fX, \fX') \leq \varepsilon$. Let $(\cX, d, \mu, \rho)$ and $(\cX', d', \mu', \rho')$ be representatives of these classes embedded in the same compact metric space $(Z, d^Z)$. For $n \in \Z$, let $I_n = [2^{-(n+1)}, 2^{-n})$. We can construct coupled Poisson processes $\mathcal P$ on $\cX \times [0,\infty)$ and $\mathcal P'$ on $\cX' \times [0,\infty)$ by superposing coupled independent Poisson processes $\mathcal P_n$ on $\cX \times I_n$ and
$\mathcal P_n'$ on $\cX' \times I_n$ as follows: first, independently for different values of $n$, let $(U_1^{(n)}, S_1^{(n)}),  (U_2^{(n)}, S_2^{(n)}), \ldots,$ be a sequence of independent and identically distributed $(\cX \times I_n)$-valued random variables, where $U_i^{(n)}, S_i^{(n)}$ are independent, $U_i^{(n)}$ is distributed according to $\mu$ and $S_i^{(n)}$ follows the distribution of a non-negative random variable with density $s^{-1-1/\alpha}$ conditioned on taking a value in $I_n$. As $\mathrm{d^Z_{\text{\textsc{p}}}}(\mu, \mu') \leq \varepsilon$, for any $n \in \Z$, we can construct a sequence of independent and identically distributed $\cX'$-valued random variables $V_1^{(n)}, V_2^{(n)}, \ldots$ such that $\pc{d^Z(U_i^{(n)}, V_i^{(n)}) \geq \varepsilon} \leq \varepsilon$ for all $i \geq 1$.
Let $N_n, n \in \Z$,  be a family of independent random variables which is independent of all previously defined quantities, where
$N_n$ follows the Poisson distribution with parameter $\int_{I_n} s^{-1-1/\alpha} ds = \alpha 2^{n/\alpha}(2^{1/\alpha}-1)$.
The sets of points $\{ (U_i^{(n)}, S_i^{(n)}) : 1 \leq i \in N_n\}$ and $\{ (V_i^{(n)}, S_i^{(n)}) : 1 \leq i \leq N_n\}$ constitute Poisson processes $\mathcal P_n$ on $\cX \times I_n$, 
and $\mathcal P'_n$ on $\cX' \times I_n$, respectively. Upon superposing the processes for different values of $n$, we obtain the sought coupled processes $\mathcal P$ and $\mathcal P'$.
It should be clear and can be shown using correspondences that, for any $\delta > 0$, $\sup_{\fX'}  \Prob{\dghp(f(\mathcal P), f(\mathcal P')) \geq \delta} \to 0$ as $\varepsilon \to 0$, where the supremum is taken over all $\fX' \in \bbK^{\GHP}$ satisfying $\dghp(\fX, \fX') \leq \varepsilon$. (Recall that $\fX$ is kept fixed.) Thus, $\sup_{\fX'} \mathrm{d_{\text{\textsc{p}}}}(\Law(f(\mathcal P)), \Law(f(\mathcal P')) \to 0$ as $\varepsilon \to 0$ showing the claimed continuity.
\end{proof}

\begin{lem}\label{lem:Hs_meas}
For any $s>0$, the map $H^s$ is measurable for $\dgh$, and thus $\DimH$ as well.
\end{lem}
\begin{proof}For fixed $\delta > 0$, let
$$H^s_\delta(A) := \inf \Bigg \{\sum_{i \geq 1} |U_i|^s: 
A \subseteq \bigcup_{i \geq 1} U_i \text{ and } |U_i| \leq  \delta 
\text{ for all } i \ge 1  \Bigg \}.$$
Then, $H^s = \lim_{\delta \downarrow 0} H^s_\delta$ and it is enough to prove the measurability 
of $H^s_\delta : \bbK^\GH \to [0,\infty]$ for fixed $\delta > 0$. To this end, we show that 
the function is upper-semicontinuous. Let $\delta > 0, s \geq 0$ and $\fX, \fX_1, \fX_2, \ldots$ 
be compact metric spaces with $d_\GH(\fX_n, \fX) \to 0$. 
For $\varepsilon  > 0$, by compactness, there exists a natural number $N$ and sets $U_1, \ldots, 
U_n$ with $|U_i| < \delta$ for all $i=1, \ldots, N$ such that $\cX \subseteq \bigcup_{i=1}^N U_i$ 
and $H_\delta^s(\fX) \geq \sum_{i=1}^N |U_i|^s - \varepsilon$. 
Now, let $0 < \varepsilon' <  \varepsilon$ be sufficiently small such that
$|U_i^{\varepsilon'}| < \delta$ for all $i=1, \ldots N$.
Choose $n$ large enough such that $d_\GH(\fX_n, \fX) < \varepsilon' /2$. 
We may assume that $\cX_n$ and $\cX$ are embedded in a compact metric space $(Z, d^Z)$ such that
$\dha(\fX_n, \fX) < \varepsilon'/2$. Then, $H_\delta^s(\fX_n) \leq \sum_{i=1}^N |U_i^{\varepsilon'}|^s \leq (1+\varepsilon) H_\delta^s(\fX) + 2 \varepsilon$. As this inequality is true for sufficiently large $n$, we can take the limit superior on the left-hand side. Then, letting $\varepsilon \to 0$ on the right-hand side shows the claim. 
Finally, it is easy to deduce measurability of $\DimH$, e.g.\ from the representation
$$\{\fX \in \bbK^\GH : \DimH(\fX) > t \} = \bigcup_{q > t} \bigcap_{\varepsilon > 0} \bigcap_{\delta > 0} \{\fX \in \bbK^\GH : H_\delta^s(\fX) < \varepsilon\}, \quad t \geq 0,$$
where $q,\varepsilon, \delta$ take rational values.
\end{proof}

The proof of the next lemma runs along the same lines as the proof of Lemma~\ref{lem:Hs_meas}.  
(Technically, $N_\cdot(\delta)$ plays a very similar role as $H^s_\delta(\cdot)$.) 
We omit the details.
\begin{lem}\label{lem:NXdelta}
Let $\delta>0$. The map from $\bbK^{\GH}$ to $\N$ that to $\fX\in \bbK^{\GH}$ 
assigns $N_\cX(\delta)$, the smallest number of open balls of radius $\delta$ needed to 
cover $\cX$, is upper-semicontinuous 
In particular, the maps $\DimuM, \DimoM : \bbK^{\GH} \to [0, \infty]$ are measurable.
\end{lem}

\section{H\"older continuity of $\sH$}

\begin{proof}[Proof of Proposition~\ref{prop:opt}]
The positive result for $\alpha < \varrho$ follows from bounds on the moments of $\sH$ provided in Lemma~\ref{prop:hoelderHpos} 
and Kolmogorov's criterion.  
Next, recall from \cite{BrSu2013a} that $h(t) := \E{\sH(t)} =  \kappa' \sqrt{t(1-t)}$ for some $\kappa' > 0$.
Let $\varrho < \alpha < \gamma$. With $Q_n^\vartheta$ defined in Eq. \eqref{def:qn}, where $Q_0^\vartheta = h$, the uniform almost sure limit $\cX$ of $Q^\emptyset_n$ in Proposition \ref{prop:conlimit} is measurable with respect to $\{(\cR^{\sigma}, \cS^{ \sigma}, \Xi^{ \sigma}): \sigma \in \Theta\}$. Thus, since $\gamma$-H{\"o}lder continuity is a tail event of this $\sigma$-algebra, by Kolmogorov's 
zero-one law, it suffices to show that $\cX$ fails to be $\gamma$-H{\"o}lder continuous with 
positive probability. For $n \geq 1$, let $C_n = \cup_{|\vartheta| = n} \partial \Lambda_\vartheta $. 
Observe that, for some $r > 0$, we have $|h(y) - h(x)| \geq r|y-x|^2$ for all $x,y \in [0,1/2]$ and $x,y \in [1/2,1]$. Fix $\vartheta \in \Theta_\ell$. Let $x,y \in \Lambda_\vartheta$ with $(x,y) \cap C_\ell = \emptyset$ 
and denote by $x_\ell, y_\ell$ their relative positions inside $\Lambda_\vartheta$, that is $x_\ell = (x - \inf \Lambda_\vartheta)/\scL(\vartheta)$, analogously for $y_\ell$. 
If $x_\ell,y_\ell \in [0,1/2]$ or $x_\ell, y_\ell \in [1/2,1]$, then
$$\cV(\vartheta)^{1/3} \geq \ell^{2-\alpha} \scL(\vartheta)^\alpha$$
implies
\begin{align*}
|Q^\emptyset_\ell(y) - Q^\emptyset_\ell(x)| &= \cV(\vartheta)^{1/3} |h(y_\ell) - h(x_\ell)|  \\ 
&\geq \ell^{2-\alpha}\scL(\vartheta)^{\alpha} |h(y_\ell) - h(x_\ell)| \geq r \ell^{2-\alpha} |y-x|^\alpha |y_\ell - x_\ell|^{2-\alpha}.
\end{align*}
As $\Lambda_\vartheta$ is the union of at most $\ell+1$ intervals, we can find 
$x,y \in \Lambda(\vartheta), (x,y) \cap C_\ell  = \emptyset$ satisfying the latter inequality with 
$|y_\ell - x_\ell| \geq 1/(4\ell)$ where $x_\ell,y_\ell \in [0,1/2]$ or $x_\ell, y_\ell \in [1/2,1]$.  
Hence, for these $x,y$ we deduce
$$|Q^\emptyset_\ell(y) - Q^\emptyset_\ell(x)| \geq \frac{|y-x|^\alpha}{16/r}.$$
As $n \to \infty$, almost surely, the maximal distance between consecutive points in $C_n$ converges 
to zero. Hence,  $\cX$ is not $\gamma$-H{\"o}lder continuous if there 
exists $n \in \N$ and an infinite path $\vartheta = \varepsilon_1 \varepsilon_2 \ldots$ such that, 
for all $k \in \N$, with $\vartheta_n = \varepsilon_1 \ldots \varepsilon_n$, 
$$ \cV(\vartheta_{kn})^{1/3} \geq (kn)^{2-\alpha} \scL(\vartheta_{kn})^\alpha.$$
Below, we will show that this event has positive probability for some $n \in \N$ (in fact, for 
all $n$ large enough).
For $\vartheta, \sigma \in \Theta$, let  
$$A_{\sigma}^\vartheta = \left \{  \frac{\cV(\vartheta \sigma)^{1/3}}{\cV(\vartheta)^{1/3}} \geq |\sigma|^{2-\alpha} \frac{\scL(\vartheta \sigma)^{\alpha}}{\scL(\vartheta)^{\alpha}} \right\}\,.$$

\medskip 
Let $N:=2^n$ and $\Theta^*$ be the complete $N$-ary tree with nodes on level $k$ denoted by $\Theta_{k,n}$. 
Moreover, let $S$  be the random subtree of $\Theta^*$ in which a node $\vartheta^* = 
\vartheta_1^*\ldots \vartheta_k^*$ with $\vartheta_1^*, \ldots, \vartheta_k^* \in \Theta_n$ on 
level $k$ exists if, for all $0 \leq i \leq k-1$, the event 
$A_{\vartheta_{i+1}^*}^{\vartheta_1^*\ldots \vartheta_i^*}$ occurs. 
By construction,  for fixed $n \geq 1$, 
$$\left \{\big(\mathbf{1}_{A_{\sigma}^{\vartheta}}\big)_{\sigma \in \Theta_n}
~:~\vartheta \in \Theta_{k,n}, k \in \N \cup \{0\}\right\}$$ is a family of independent 
and identically distributed random vectors. 
Thus, $S$ is a branching process with offspring mean 
$$\sum_{\vartheta \in \Theta_n} \pc{A_{\vartheta}^\emptyset} 
=  \sum_{\vartheta  \in \Theta_n} \Prob{\cV(\vartheta)^{1/3} 
> n^{2-\alpha}\scL(\vartheta)^\alpha}.$$
By the elementary formula $\Prob{A \cup B} = \Prob{A} + \Prob{B} - \Prob{A\cap B}$ it is easy to see 
that 
$$\sum_{\vartheta \in \Theta_n} \pc{A_{\vartheta}^\emptyset} 
= 2^n  \Prob{ \prod_{i=1}^n W_i^{1/3} > n^{2-\alpha}\prod_{i=1}^n \xoverline W_i^\alpha},$$
where $W_1, \xoverline W_1, \ldots, W_n, \xoverline W_n$ are independent and identically uniformly distributed on 
the unit interval. We may assume $\alpha < 1/3$. Let $\delta > 0$ and $E_1, F_1, \ldots, E_n, F_n$ be independent standard exponentials.
Then, by an application of Cram{\'e}r's theorem for sums of independent and identically distributed random variables with finite momentum generating function in a neighborhood of zero, for all $n$ sufficiently large, 
\begin{align*}
\Prob{\prod_{i=1}^n W_i^{1/3} > n^{2-\alpha} \prod_{i=1}^n \xoverline W_i^\alpha} 
& \geq \Prob{ \sum_{i=1}^n  \Big(\alpha F_i - \frac 1 3 E_i\Big)  > \delta n } \\
&= \exp(-I(\delta)n + o(n)),
\end{align*}
where $I(x), x \in [\alpha - 1/3, \infty)$ denotes the large deviations rate function of the random 
variable $\alpha F_1 - E_1/3$ (see, e.g., \cite{DeZe1998}) given by 
\begin{align*} I(x) = \sup_{-3 < s < 1/\alpha}  s x - \log \frac{3}{3+s} \frac{1}{1- s\alpha} & \leq  
\frac{x}{\alpha} - \inf_{-3 < s < 1/\alpha} \log \frac{3}{3+s} \frac{1}{1- s\alpha}  \\ & = \frac {x}{\alpha} - \log  \frac {12\alpha}{(3\alpha+1)^2}.\end{align*}
Thus, for all $n$ large enough,
\begin{align*}
\sum_{\vartheta \in \Theta_n} \pc{A_{\vartheta}^\emptyset} 
& \geq (2c (1+o(1)))^n, \quad c = \frac {24\alpha}{(3\alpha+1)^2}  e^{-\delta / \alpha}.
\end{align*}
Since $\alpha > \varrho$, upon choosing $\delta > 0$ sufficiently small, we obtain $c > 1/2$.
Thus, for all $n$ sufficiently large, 
with positive probability, there exists an infinite path $\vartheta_1^* \vartheta_2^* \ldots$  in 
$\Theta^*$ with $\vartheta_i^* \in \Theta_n$ such the events 
$A_{ \vartheta_{i+1}^*}^{\vartheta_1^*\ldots \vartheta_i^*}$ occur for all 
$i \in \N \cup \{0\}$. Along this path written as $\vartheta = \varepsilon_1\varepsilon_2 \ldots$, 
we deduce
\begin{align*}
\cV(\vartheta_{kn}) \geq n^{(2-\alpha)k}\scL(\vartheta_{kn})^{\alpha}  \geq (kn)^{2-\alpha} \scL(\vartheta_{kn})^{\alpha}.
\end{align*}
for all $k \in \N$. This concludes the proof.
\end{proof}

The final lemma generalizes Proposition 4.1 in \cite{legalcu} to non-integer values of $p$.
\begin{lem} \label{prop:hoelderHpos}
For all $\varepsilon > 0$ and $p \in [0,\infty)$, there exists $K >0$ such that, for all $x \in [0,1]$, 
$$\Ec{\sH(x)^p} \leq K (x(1-x))^{2p /(p+3) - \varepsilon}$$
\end{lem}
\begin{proof}
We provide the minor modifications necessary to extend Proposition 4.1 in \cite{legalcu} to the non-integer case without presenting tedious calculations. First, by Jensen's inequality, 
since $\E{\sH(x)} = \kappa' \sqrt{x(1-x)}$ with $\kappa' = 1 / \Gamma(4/3)$, we have
\begin{align}
&\E{\sH(x)^p} \leq (\kappa') ^p (x(1-x))^{p/2}, \quad 0 \leq p \leq 1, \label{b80} \\
&\E{\sH(x)^p} \geq (\kappa') ^p (x(1-x))^{p/2}, \quad p \geq 1. \label{b81}
\end{align}
Thus, for $0 \leq p \leq 1$, the assertion follows immediately from \eqref{b80}. For $p \in (1,\infty)$, we do not have a integral recursion for $m_p(t) = \E{\sH(t)^p}$ such as (17) in \cite{legalcu} unless $p$ is integer. However, applying the inequality $(a+b)^p \leq a^p + b^p + C_1 (a^{p-1} b + a b^{p-1})$ for $a, b \geq 0$ and some $C_1 = C_1(p)$, to the stochastic fixed point equation Eq. \eqref{eqH}, we have, in a stochastic sense, 
\begin{align}  
& \sH(t)^p \notag \\ 
& \leq  \In{[0,U_1)}(t) W^{p/3} (\sH^{(1)})^p \left(\frac{t}{\Delta_1}\right) 
 + \In{[U_2,1]}(t)  W^{p /3} (\sH^{(1)})^p \left(\frac{t-\Delta_2}{\Delta_1}\right)  \nonumber \\
&  + \In{[U_1,U_2)}(t) \left( (1-W)^{p/3} (\sH^{(2)})^p \left(\frac{t-U_1}{\Delta_2} \right) 
   +  W^{p/3} (\sH^{(1)})^p \left(\xi \right)\right) \nonumber \\ 
 &  + \In{[U_1,U_2)}(t) \left(  C_1 (1-W)^{(p-1)/3} W^{1/3} (\sH^{(2)})^{p-1} \left(\frac{t-U_1}{\Delta_2} \right)   \sH^{(1)} \left(\xi \right)\right) \label{b82} \\
 &  + \In{[U_1,U_2)}(t) \left(C_1 (1-W)^{1/3} W^{(p-1)/3} (\sH^{(2)}) \left(\frac{t-U_1}{\Delta_2} \right)   (\sH^{(1)})^{p-1} \left(\xi \right) \right),  \label{b83}
\end{align}
with conditions as in Eq. \eqref{eqH} on the right hand side.
Subsequently, we consider $0 \leq t \leq 1/2$ which suffices by symmetry. 
With $q = 3/(3+p)$, taking the expectation on both sides of the last display leads to
$$m_p(t) \leq 2q (1-t)^2 \int_{0}^t m_p(x) (1-x)^{-3} dx +  2q t^2 \int_{t}^1 m_p(x) x^{-3} dx + s_p(t),$$
where $s_p(t)$ is the sum of the expectation of \eqref{b82} and \eqref{b83}. As in the proof of Proposition 4.1 in \cite{legalcu}, relying on first, Lemma 4.2 there for $\sH$ instead of $M$, and second, a stochastic inequality inverse to the display above based on $(a+b)^p \geq a^p + b^p$ for $a,b \geq 0$, one can show that the first summand has negligible contribution as $t \to 0$. In other words, for any $\delta > 0$, there exists $t_0$ such that, for $t \leq t_0$, we have
\begin{align} \label{b87} m_p(t) \leq   2(1+\delta)q t^2 \int_{t}^1 m_p(x) x^{-3} dx + 2s_p(t). \end{align}
Furthermore, for some $C_2 = C_2(p, t_0, \delta) > 0$, 
$$m_p(t) \leq   2(1+\delta)q t^2 \int_{t}^{t_0} m_p(x) x^{-3} dx + 2s_p(t) + C_2 t^2.$$
Now, if we were to drop $s_p(t)$, then, by applying Gronwall's lemma to the function $m_p(t) t^{-2}$, we could deduce
$$m_p(t) \leq C_3 t^{2p/(p + 3) - \delta q}$$
for all $t \in [0,1]$ and some $C_3 = C_3(p, t_0, \delta)$.  This would give the assertion as $\delta$ was chosen arbitrarily.
For a rigorous verification, we start with the case $1 < p \leq 2$. Then, a direct computation shows that, for some $C_4 = C_4(p)$, we have $s_p(t) \leq C_4 t^{(p + 2)/2}$. Thus, by \eqref{b81}, $s_p(t)$ is asymptotically negligible compared to $m_p(t)$. Using \eqref{b87}, for any $\delta' > 0$, upon decreasing $t_0$ if necessary, we have
$$m_p(t) \leq   2(1+\delta)(1 + \delta') q t^2 \int_{ t}^1 m_p(x) x^{-3} dx.$$
As indicated, the claim now follows from Gronwall's lemma with a suitable choice of $\delta$ and $\delta'$.
For $p > 2$, we proceed by induction. We may assume that $\varepsilon >0$ is chosen small enough such that $3/2 + 2(p-1)/(p+2) - \varepsilon > 2$. 
By the induction hypothesis, there exists $C_5 = C_5(p)$, such that $s_p(t) \leq C_5 t^{3/2 + 2(p-1)/(p+2) - \varepsilon}$. Thus, $s_p(t)t^{-2}$ is bounded on $[0,1]$ and the result follows as indicated from inequality \eqref{b87}.
\end{proof}

\end{document}